\documentclass[11pt]{article}
\usepackage[a4paper, margin=1in]{geometry}
\usepackage{amsmath, amssymb, amsthm}
\usepackage{graphicx}
\usepackage{hyperref}
\usepackage{authblk} 
\usepackage{abstract} 

\usepackage{mathrsfs}
\newtheorem{theorem}{Theorem}[section]
\newtheorem{proposition}{Proposition}[section]
\newtheorem{remark}{Remark}[section]%
\newtheorem{lemma}{Lemma}[section]%
\newtheorem{definition}{Definition}[section]%

\numberwithin{equation}{section}

\usepackage{comment}
\usepackage{tikz}
\usetikzlibrary{shapes.misc, positioning}
\usetikzlibrary{fit}
\usetikzlibrary{shapes.symbols}

\title{The diffusion approximation of the Multiclass Processor Sharing queue}
\author{Mohamed Ghazali \thanks{\href{mailto:m.ghazali@uhp.ac.ma}{m.ghazali@uhp.ac.ma}}}
\author{Abdelghani Ben Tahar}
\author{Amal Ezzidani}

\affil{Hassan First University of Settat, Faculty of Sciences and Techniques, Settat, 26000, Morocco}

\date{} 

\begin{document}

\maketitle

\begin{abstract}
This paper considers a multiclass processor-sharing queue with feedback. Jobs arrive according to renewal processes, and service times follow general distributions. Upon service completion, jobs may either depart the system or re-enter as a different class according to a probabilistic, Jackson-like routing mechanism. Under heavy-traffic conditions, we establish a diffusion approximation for a measure-valued process tracking the residual service times of jobs. 
\end{abstract}

\section{Introduction}

We consider a single-server, infinite-capacity queue serving $K$ classes of jobs. All jobs present in the system are served simultaneously according to the egalitarian processor-sharing discipline: every present job simultaneously receives an equal fraction of the server's capacity. Jobs arrive from the outside with a given class, and their service time distribution depends on the class. Upon service completion, jobs may leave the system or re-enter it as jobs of a different class, according to a probabilistic, Jackson-like routing mechanism. We call this system a Multiclass Processor Sharing (MPS) queue. 

\subsection{Literature and contribution}

Many studies examine queueing systems with the processor-sharing (PS) discipline. We focus on research that considers general interarrival and service time distributions, especially those that develop fluid and diffusion approximations.
 \cite{gromoll2002fluid} introduced a measure-valued process called the state descriptor, which tracks the residual service times of jobs, encompassing the queue size and total workload. They established fluid limit results for the critically loaded processor sharing queue.
Later, \cite{gromoll2004diffusion} used these results, along with \cite{puha2004invariant}, to derive the diffusion approximation for the
 $GI/GI/1/PS$ queue.
 \cite{puha-al-2006} studied the fluid model of overloaded $GI/GI/1/PS$ queue and gave the asymptotic behaviors (for large time values) of fluid solutions. 
Other authors  used the same framework in analyzing the queueing systems operating under service discipline related to PS. 
For instance, \cite{gromoll2007,gromollRobertZwart} studied, respectively, the diffusion approximation of PS queue with soft deadlines, and the fluid limit of PS queue with impatience. \cite{zhang2011diffusion} established the diffusion approximation for the limited PS queue.
 An extension to the multiclass case with feedback was carried out by \cite{tahar2012fluid}, where a fluid approximation was established for our Multiclass Processor Sharing (MPS) queue. Building on this work, the present paper develops a corresponding diffusion approximation.\\

Our MPS queue belongs to a class of queueing networks known as \textit{multiclass queueing networks}. 
These networks have a finite number of servers, each with unlimited buffer capacity. Jobs move through the system, receiving service at different stations before eventually leaving. As they do, they may switch classes between stations or even loop back within the same station. The transitions between job classes follow a Markovian process, characterized by a fixed substochastic transition matrix.
For a detailed discussion on multiclass queueing networks, see \cite{williams2016stochastic}. 
Our model, which consists of a single station where jobs can feedback within that station, is a specific example of this type of network.

\cite{williams1998diffusion} provided general sufficient conditions under which a heavy-traffic limit theorem holds for open multiclass queueing networks with head-of-the-line (HL) service disciplines (jobs within each class are served in a first-in-first-out (FIFO) manner).
\cite{bramson1998state} showed that a multiplicative state space collapse (SSC) holds for two HL network families: FIFO networks of Kelly type and head-of-the-line proportional processor sharing (HLLPS) networks. Under heavy traffic, SSC implies the queue length process is a function of the workload process.
 Together, these results yield a diffusion approximation for both network families. Building on this framework, \cite{bramsonDai2001} established diffusion approximations for multiclass queueing networks with other HL disciplines.

Deriving diffusion approximations for networks with non-head-of-the-line (non-HL) disciplines, such as processor sharing, is significantly more challenging.
Beyond the previously mentioned processor-sharing literature, research on other non-HL disciplines has primarily focused on single-class systems or their variants. For example, \cite{limic2000behavior} and \cite{gromoll2011diffusion} established heavy-traffic approximations for $GI/GI/1$  queues under last-in-first-out (LIFO) and shortest remaining processing time (SRPT), respectively. Some studies have considered multiclass, multi-station networks without feedback. \cite{kruk2004earliest} derived a heavy-traffic approximation for acyclic networks under earliest deadline first (EDF), while \cite{loeser2024fluid} established a fluid approximation with reneging for random order of service (ROS).
However, these works do not account for feedback, even in the single-station case. Our work, to our knowledge, is one of the first to establish a diffusion approximation for a non-HL, single-station, multiclass queueing network.

\subsection{Methodology and results}

We largely follow the notation and terminology of \cite{tahar2012fluid}.  We model the dynamics of the multiclass processor-sharing (MPS) queue with the following \( K \)-dimensional stochastic processes: \( A(\cdot) \), \( D(\cdot) \), \( Z(\cdot) \), and \( \mu(\cdot) \), taking values in \( \mathbb{N}^K \), \( \mathbb{N}^K \), \( \mathbb{N}^K \), and \( \mathcal{M}^K \), respectively. Here, \( \mathcal{M} \) denotes the space of finite, non-negative Borel measures on \( \mathbb{R}_+ \), equipped with the topology of weak convergence. For each class \( k \), the components of these vectors represent the cumulative number of arrivals \( A_k(t) \) (both internal and external) and departures \( D_k(t) \) up to time \( t \), the number of jobs present \( Z_k(t) \) at time \( t \), and the measure-valued process \( \mu_k(t) \) tracking the residual service times of class \( k \) jobs at time \( t \).
  A random vector \( E(t) \) represents the number of external arrivals (from outside the network) at each class by time \( t \).  The evolution of the MPS queue is described by queueing equations \eqref{eq:arrivee-depart-1}--\eqref{eq:muk-1}, which relate the vectors $A(t)$, $D(t)$, $Z(t)$, $E(t)$, and $\mu(t)$.
The feedback in our MPS queue is characterized by a routing matrix $P$.
  Entry $p_{lk}$ of this matrix represents the probability that a job completing service in class $l$ immediately re-enters the queue as a class $k$ job. This routing decision is independent of past events, as jobs transition between classes in a Markovian fashion. The routing matrix $P$ is time-homogeneous and state-independent.  Therefore, a job's subsequent destination (either another class or departure from the system) upon service completion is determined solely by its current class and the corresponding probabilities in $P$.

\par  As is common in the literature for heavy traffic approximation, we consider a sequence of MPS queues indexed by $r$. 
 Denote by $\hat{\mu}^r(t) = \mu^r(r^2 t)/r$ the diffusion-scaled version of the state descriptor $\mu^r(\cdot)$.  
The main goal of this paper is to prove Theorem \ref{thm:main-result}, which states that, under mild conditions, the following convergence in distribution holds:
\begin{equation*}
\hat{\mu}^r(\cdot) \Rightarrow \Delta^{\nu} W^*(\cdot) \quad\text{as } r \to \infty.
\end{equation*}
  Here, $\Delta^{\nu}:\mathbb{R}_+\rightarrow \mathcal{M}^K$ is the lifting map (cf. Definition \ref{def:lift-map}), and $W^*(\cdot)$ is a one-dimensional reflected Brownian motion arising as the limit of the diffusion-scaled workload process (cf. Proposition \ref{cv:workload}). 

The dynamics of $\mu_k(\cdot)$ for a fixed class $k$ in \eqref{eq:muk-1} mirror those of the state descriptor $\mu(\cdot)$ in \cite{gromoll2004diffusion} (see Equation (2.6) therein), with two key differences: 
\textbf{a)} In this multiclass setting, the cumulative service function $S^r(\cdot)$ (see \eqref{eq:barS}) incorporates the total queue size $\sum_{k} Z^r_k(\cdot)$, rather than the class-specific queue size $Z_k^r(\cdot)$. 
\textbf{b)} The arrival process $A^r_k(\cdot)$ replaces the exogenous arrival process $E^r_k(\cdot)$.  Unlike $E^r_k(\cdot)$, the feedback mechanism introduces uncertainty regarding the renewal nature of $A^r_k(\cdot)$.
 
 Our proof of the main result proceeds in three steps. First, because internal job transitions do not impact the total number of jobs (and hence the processing rate), the total system size evolves similarly to a $GI/GI/1/PS$ queue.
 In this analogous system, each job's service time represents the sum of their individual service requirements in the multiclass model.  
Based on this, we use a measure-valued process \(\gamma^r(\cdot)\),  which describes the total residual service time of all jobs present.
In particular, the number of jobs present in the system and the workload at time $t$  are determined by:
$
\langle 1, \gamma^r(t) \rangle = e\cdot Z^r(t)$ and  $\langle \chi, \gamma^r(t) \rangle = W^r(t)$ for $t\geq 0$.
Denote by $\widehat{\gamma}^{r}(t)= \gamma^r(r^2t)/r$ the diffusion-scaled version of the state descriptor $\gamma^r(\cdot)$. 
We leverage existing results on single-class $GI/GI/1/PS$ queues to establish the diffusion approximation for $\gamma^r(\cdot)$ (cf. Proposition \ref{prop:gamma-global}). This implies the convergence in distribution of the diffusion-scaled total queue size process $\sum_k \hat{Z}_k^r(\cdot)$.

 Second, since jobs may revisit classes in our queueing system, we are interested in the number of visits a job makes to a given class $k$ along its route, and the sequence of classes it occupies during its sojourn. This is captured by a measure $\gamma_{lk}(t)$ representing the total residual service time across all visit instances (past, present, and future) to class $k$, for all jobs currently in class $l$ at time $t$ that initially entered the system as class $l$ jobs. The family of state descriptors $\{\gamma_{lk}(\cdot), l,k=1,\ldots,K\}$ encapsulates much of the relevant system information. Specifically, the total mass of $\gamma_{lk}(t)$ represents the remaining number of visits to class $k$ for jobs initially of class $l$ that are present in the system at time $t$.
In our work, we find it more practical to work with the measure-valued process $\mathcal{Q}(t) =(\mathcal{Q}_1(t),\ldots,\mathcal{Q}_K(t))$, where $\mathcal{Q}_k(t)=\sum_l\gamma_{lk}(t)$. Under heavy traffic and mild assumptions on the primitive data and initial state, we establish the following convergence in distribution:
\begin{equation*}
\hat{\mathcal{Q}}^r(\cdot) \Rightarrow \mathcal{B}*\Delta^{\nu}W^*(\cdot) \quad \text{as } r \to \infty.
\end{equation*}
Here, $\mathcal{B}=\sum_{n=0}^\infty (BP^{\prime})^{n}$, where $B$ is a diagonal matrix with entries $(B_k, k=1,\ldots,K)$, and $B_k$ is the service time distribution function for class $k$ jobs. 
 This convergence of the process $\hat{\mathcal{Q}}^r(\cdot)$ plays a key role in proving the convergence of 
the sequence $(\hat{A}^r(\cdot), \hat{D}^r(\cdot), \hat{Z}^r(\cdot))$ 
(cf. Proposition \ref{theo:conv-of-A-D-Z}). 
 
Finally, having established the diffusion approximation for the arrival process at each class, proving the diffusion approximation for the state descriptor $\mu_k(\cdot)$ requires a minor adaptation of the framework in \cite{gromoll2004diffusion,gromoll2007}.

\subsection{Organization of the paper}


This paper is organized as follows. Section \ref{sec:mult-queue-model} introduces the multiclass processor-sharing model, including the primitive data, processes, and evolution equations (Sections \ref{subsec:queueing-model} and \ref{subsec:model-equation}), as well as the corresponding fluid model (Section \ref{sect:fluid-mod}).  The main result (Theorem \ref{thm:main-result}) under heavy-traffic conditions is presented in Section \ref{MPS-model}. Section \ref{sect:Main-steps-proof} leverages new state descriptors to prove the main result, with supporting results provided in  \ref{unif_conv_sect_app} and \ref{Sec5Q}.

\subsection{Notations} 
In this paper, we will use the following notations. We label classes by $k=1,\ldots,K$, we use $\mathcal{K}$  to denote the set of all classes.
Let $\mathbb{N},\mathbb{R}$ denote the set of natural numbers and real numbers respectively.
 Let $\mathbb{R}_+$ denote the non-negative real numbers, and let 
$\mathbb{N}^{\ast},\mathbb{R}_+^{\ast}$ denote the positive natural numbers and real numbers respectively. Let $\mathbb{R}^K_+$ denote the $K$-dimensional Euclidean space, and $\mathbb{N}^K$ the $K$-ary Cartesian power of $\mathbb{N}$.
 For $a \in \mathbb{R}$, write $a^+$ for the positive part of $a$ and let $\lfloor a \rfloor$ be the integer part of $a$ and $\lceil a \rceil$ the smallest following integer. For $a,b \in \mathbb{R}^K$, let $a \vee b$ denote the componentwise maximum of $a$ and $b$.
  Denote the indicator of Borel set $B \subset \mathbb{R}_+$  by $1_B$ and let $B^{\varepsilon} =  \lbrace  y \in \mathbb{R}: \inf_{x \in B} |x-y| < \varepsilon \rbrace$ for a given $\varepsilon>0$. Let $I(x)$ denotes the interval $[x,\infty)$.
Vectors will be normally arranged as a column. As an exception, the vector $e$ stands for a row vector of ones. Inequalities between vectors in $\mathbb{R}^K$ should be interpreted
componentwise.  
 The transpose of a vector or matrix is denoted by a prime. The $K\times K$ diagonal matrix whose entries are given by the components of $x$ will be denoted by $diag\{x\}$.
For a vector $x \in \mathbb{R}^K$ and a $K \times K$ matrix $A$, we define the following normes respecively
$$ |x| = \max_{k=1}^K |x_k|, \qquad 
 |A| =\max_{k=1}^K \sum_{j=1}^K |A_{kj}|.$$
For a function $g:\mathbb{R}_+ \mapsto \mathbb{R}^K$, let $\| g\|_T= \sup_{t \in [0,T]} \, |g(t)|$ for each $T\geq 0$.
 The following real-valued functions will be used
repeatedly: $\chi (x)=x$ for $x \in \mathbb{R}_+$, and  $\varphi$:
\begin{align} \label{varphi}
\varphi(x)=1/x ~~ \text{ for } x \in (0,\infty), ~~ \text{ and } ~ \varphi(0)=0.	
\end{align}
\par The set of finite, nonnegative Borel measures on $\mathbb{R}_+$ is
denoted by $\mathcal{M}$, and the $K$-ary Cartesian power of $\mathcal{M}$ is denoted by $\mathcal{M}^K$. For $\mu\in \mathcal{M}$ and a Borel
measurable function $g$ which is integrable with respect to $\mu$,  we write $\langle
g,\mu\rangle=\int g\, d\mu$. If $\mu \in \mathcal{M}^K$ and $g$ is integrable with respect to $\mu_k$ for all $k \in \{1, \ldots,K \}$, we write $\langle
g,\mu\rangle = \left( \langle g,\mu_1 \rangle, \langle g,\mu_2 \rangle, \ldots, \langle g,\mu_K \rangle \right)$.
If $g=1_A$ where $A$ is a measurable set, then we simply write $\mu(A)$.
 Let $\langle  g(\cdot - a),\nu \rangle$ denotes $ \int_{[a,\infty)} g(x-a) \nu(dx)$ for all $a>0$.   

The space $\mathcal{M}$ is endowed with the weak topology, where a sequence $(\mu_n)_{n\geq 1} \subset \mathcal{M}$ converges weakly to $\mu \in \mathcal{M}$ if and only if $\langle g , \mu_n \rangle \rightarrow \langle g, \mu \rangle$ for all $g :\mathbb{R}_+ \rightarrow \mathbb{R}$ that are bounded and continuous. In this topology,  $\mathcal{M}$ is a Polish space.
Denote the weak convergence of $\mu_n$ to $\mu$ by $\mu_n \stackrel{w}{\longrightarrow} \mu$. For $\xi , \eta \in \mathcal{M}$, we define on $\mathcal{M}$ the metric $\boldsymbol\rho$ by  
\begin{align*}
\boldsymbol\varrho (\xi , \eta) = \inf \lbrace \, \varepsilon > 0 : \, \xi (B) \leq \eta (B^{\varepsilon}) + \varepsilon  ~ \mbox{ and }  ~ \eta (B) \leq \xi(B^{\varepsilon}) + \varepsilon  ~ \mbox{ for all closed sets} ~ B \subset \mathbb{R}_+  \rbrace.
\end{align*}
This metric that induces the topology of weak convergence on $\mathcal{M}$, is complete.
 The sequence $(\xi_n)_{n \geq 1} \subset \mathcal{M}^K$ converges weakly to $\xi \in \mathcal{M}^K$ if $(\xi_n)_k \stackrel{w}{\longrightarrow} \xi_k $ for all $k \in \mathcal{K}$ as $n \rightarrow \infty$. In this case, we denote the weak convergence of $\xi_n$ to $\xi$ by $\xi_n \stackrel{w}{\longrightarrow} \xi$. 
For $\xi , \eta \in \mathcal{M}^K$, we define the metric $\mathbf{d}$ on $\mathcal{M}^K$     
\begin{align*}
\mathbf{d} (\xi , \eta) = \max_{k \in  [\![ 1,K   ]\!] } \boldsymbol\rho (\xi_k , \eta_k) .
\end{align*} 
The measure $\delta^+_x$ denotes the element
of $\mathcal{M}$ with mass one at $x>0$. The symbol
$\mathbf{0}$ denotes the zero measure of $\mathcal{M}^K$, the
dimension $K$ being always clear from the context.  Let
$\mathcal{M}^{c,K}=\{\xi \in \mathcal{M}^{K}:\xi_{k}(\{x\})=0$ for all
$x\in\mathbb{R}_{+}$ and $k\in \mathcal{K}\}$ be the set of vectors of
finite, non-negative Borel measures on $\mathbb{R}_{+}$ that have no
atoms, and let $\mathcal{M}^{c,p,K}=\{\xi \in \mathcal{M}^{c,K}:\xi
\neq \textbf{0} \}$ be the set of positive measures of
$\mathcal{M}^{c,K}$.

\par All stochastic processes are assumed to be right continuous with finite left
limits. Let $S$ a general metric space and let $L>0$, we denote by $\mathbf{D} ([0,L], S)$ [resp. $\mathbf{D} ([0,\infty), S)$] the space of all right continuous
$S$-valued functions with finite left limits defined on the interval $[0,L]$ (resp. $[0,\infty)$). This space is endowed with the Skorohod $J1$-topology \cite{ethier1986markov}.

\par We will use $\mathbb{P}$ and $\mathbb{E}$
to denote the probability measure and expectation operator with
whatever space the relevant random element is defined on, and
$\Rightarrow$ to denote convergence in distribution of a sequence of
random elements of a metric space.

\par Let $f$ be a locally bounded Borel measurable function and $g$ a right continuous function that is locally bounded. The convolution of $f$ and $g$ is defined by
$ (f \ast g)(x)= \int_{0}^x f(x-y) dg(y)$.
For two matrices of measurable functions $F(\cdot)$ and $G(\cdot)$ defined on
$\mathbb{R}_+$, we denote by the matrix-valued functions $(F\ast
G)(x)$ for $x\in \mathbb{R}_+$, the matrix convolution formed of the
elements: $(F\ast G)_{ij}(x)=\sum_k(F_{ik}\ast G_{kj})(x)$. 
This operation is associative and distributive over matrix addition.
The multiplication by a constant matrix $C$ can be seen as a convolution,
where each element $C_{ij}$ is interpreted as the function
$C_{ij}1_{x\geq 0}$.  Associativity therefore holds for mixed scalar
products and convolutions. 
The $n^{th}$ convolution power of a matrix
$F(x)$ is denoted with $F^{\ast n}(x)$.
For a continuously differentiable function $g$, we write $\dot{g}(x)=\frac{d}{dx}g(x)$. 


\section{The Multiclass Queuing Model}
\label{sec:mult-queue-model}
\vspace*{0.2cm}
\subsection{Primitives data and initials conditions}
\label{subsec:queueing-model}
\vspace*{0.2cm}
For each class $k \in \mathcal{K}$, we assume that there are two  
sequences of random variables, $u_{k}=\{u_{k}(i),i\geq 1\}$ and
$v_k=\{v_{k}(i),i\geq 1\}$ and a  sequence of $K$-dimensional
random vectors $\varphi ^{k}=\{\varphi^{k}(i),i\geq 1\}$, such that $\{u_{k}(i),i\geq 2\}$,
$v_k$ and $\varphi ^{k}$ are i.i.d. $u_k (1)$ is assumed to be independent of $\{u_{k}(i),i\geq 2\}$ and to be strictly positive with finite mean. Each element of $\{u_{k}(i),i\geq 2\}$, $v_k$ and $\varphi^{k}$ takes values respectively in
$\mathbb{R}_+$, $\mathbb{R}^{\ast}_+$ and $\{e_0,e_1,...,e_K\}$, where
$e_0$ is the $K$-dimensional vector of all components $0$, and $e_k$
is the $K$-dimensional vector with $k^{th}$ component being $1$ and other
components being $0$. 
These sequences have the following
interpretation: $ u_{k}(1)$ is the time of the first externally arriving job at class $k$, and for each $i\geq 2$ ,
$u_{k}(i)$ is the interarrival time between the $(i-1)^{th}$ and the
$i^{th}$ externally arriving job at class $k$. For each $i \geq 1$, $v_{k}(i)$ is the
service times for the $i^{th}$ class $k$ job, and $\varphi^{k}(i)$ represent the routing matrix, where $\varphi^{k}(i)=e_l$ with $l = 1 , \cdots , K$ means that the $i^{th}$ job of class $k$ which completes service, returns to the queue as class $l$ job, and $\varphi^{k}(i)=e_0$ means the job leaves the queue. We assume that the sequences
$$ u_1,...,u_K, v_1,...,v_K, \varphi^{1},...,\varphi^{K} $$
are mutually independent. They
constitute the \textit{primitive data} of the queue.
\par From this data the following parameters are derived. The real valued vectors
$\alpha=(\alpha_k, k \in \mathcal{K})$ and $a= (a_k, k \in \mathcal{K})$ are defined as $\alpha_k =
[\mathbb{E}(u_{k}(2))]^{-1}$  and $a_k=var(u_k(2))$ for each $k\in\mathcal{K}$. Denote by
\begin{equation}
\label{eq:cov-arrivee}
 \Pi = diag\{\alpha^3_k a_k, k \in \mathcal{K}\}
\end{equation}
We allow that $\alpha_k=0$ for some
$k$ and we set $\mathcal{A}=\{k:\alpha_k\neq 0\}$. The vector
$\nu=(\nu_1,...,\nu_K)$ is formed of $\nu_k$, the Borel probability
measure of $v_k$, with mean $\beta_k = \langle \chi , \nu_k \rangle > 0$ and variance $b_k =  \langle \chi^2 , \nu_k \rangle - \beta_k^2 < \infty$. Denote by
\begin{equation}
\label{eq:cov-service}
\Sigma=diag\{b_k, k \in \mathcal{K}\}
\end{equation}
It is assumed that for each $k\in\mathcal{K}$, the distribution $\nu_k$ does not charge the origin, i.e $\nu_k(\{0\})=0$.  Let $p_{kl} =\mathbb{P}(\varphi^{k}(1)=e_l)$ be the probability of departing class $k$ and becoming class $l$. Our networks are assumed to be open, that is the routing matrix $P$ satisfies
\begin{equation*}
\label{eq:P}
  Q:=I+ P'+(P')^2+...
\end{equation*}
is finite, which is equivalent to requiring that $(I-P')$ be
invertible, or that $P$ has a spectral radius less than 1. In that case, $Q=(I- P')^{-1}$. Note that for each $k\in \mathcal{K}$
\begin{equation}
 \mathbb{E}(\varphi^k(i))=P^{k}, ~~~\text{and}~~Cov(\varphi^k(i))=H^k, \label{cov:Phi}
\end{equation}
where $P^{k}$ is the $k^{th}$ column of the matrix $P$ and $H^k$ is $K\times K$-matrix defined as
\begin{equation}
 H_{lm}^k=
\left\{
\begin{array}{lll}
p_{kl}(1-p_{kl}) & \text{if} &~~l= m \\
-p_{kl}p_{km} &\text{if} &~~l\neq m.
     \end{array}
\right.
\label{eq:cov-routing}
\end{equation}
For each $k\in \mathcal{K}$, we denote
\begin{align*}
E_{k}(t)=\sup\left\{n: \sum_{i=1}^{n}u_{k}(i)\leq t\right\} , \qquad \Phi_{k}^l(n)=\sum_{i=1}^{n}\varphi_{k}^{l}(i), 
\end{align*}
where $E_k (t)$ is the number of exogenous arrivals to class  $k$ by time $t$, and $\Phi_{k}^l(n)$ is the number of jobs that move from class $l$ to class $k$, among the $n$ first jobs of class $l$.
We denote 
\begin{align*}
E(t)=\left(E_{1}(t),...,E_{K}(t) \right), \quad \Phi(n)= \left(\Phi_{k}^l(n): l,k\in \mathcal{K} \right).
\end{align*}
 The processes $E=(E(t),t\geq 0)$ and $\Phi=(\Phi(n),n\geq 0)$ are the \textit{primitive processes} of our queueing systems.
\par For each $k\in \mathcal{K}$, we assume
that there exists an integer random variable with finite mean $Z_k(0)$
and an i.i.d.~sequence of strictly positive random variables
$v^0_k=\{v_{k}^{0}(i),i\geq 1\}$ with a common Borel probability
measure $\nu _{k}^{0}$, such that
\begin{displaymath}
  v^0_1,...,v^0_K,v_1,...,v_K,\varphi^1,...,\varphi^K, Z_1(0),...,Z_K(0)
\end{displaymath}
are mutually independent. Any job belonging class  $k$ at time zero
in the system is referred to as an ``initial job in class $k$''. Then
let $Z_K(0)$ be the number of initial job in class $k$ and
$v_{k}^{0}(i)$ be the service times requirement of the $i^{th}$
initial job at class $k$. 
\subsection{Queueing equations}
\label{subsec:model-equation}
Given the primitive data and the primitives processes defined in the previous section. Let
$$ A(t)=(A_{1}(t),...,A_{K}(t)),~ D(t)=(D_{1}(t),...,D_{K}(t)), ~Z(t)=(Z_{1}(t),...,Z_{K}(t))$$
be a random processes such that $A_{k}(t), D_{k}(t)$ and $Z_{k}(t)$
are respectively, the total number of arrivals by time $t$ at, the
number of departures by time $t$ from, and the number of jobs
present at time $t$ in, class $k$.  Jointly, those processes satisfy
the following queueing equations:
\begin{align}
  A_{k}(t) &= E_{k}(t)+\sum_{l=1}^{K}\Phi_{k}^l(D_{l}(t))
  \label{eq:arrivee-depart-1}
  \\[0.05cm]
  D_{k}(t) &= \sum_{i=1}^{Z_{k}(0)}1_{\{v_{k}^{0}(i)\leq
    S(t)\}}+\sum_{i=1}^{A_{k}(t)}1_{\{v_{k}(i)\leq S(\sigma_{k}(i),t)\}}
  \label{eq:depart-1}
  \\[0.15cm]
  Z(t) &= Z(0)+A(t)-D(t)
  \label{eq:nbre-client-1} \\[0.2cm] 
  S(t) &= \int_{0}^{t}\varphi (e. Z(s) ) \, ds, ~~ \text{and} ~ S(s,t):=S(t)-S(s).
  \label{eq:cumu-service-1}
\end{align}

Here, for all $k\in \mathcal{K}$, $A_{k}(0)=0$, $\sigma_{k}(i)$ is
the arrival epoch of the $i^{th}$ job to arrive at class $k$, and
$\varphi(.)$ is defined in \eqref{varphi}. 
The function $S(t)$ is known as the cumulative service. It represents
the amount of service received by one particular job in the
interval $[0,t]$. Since the Processor Sharing gives the same amount of
service to all present jobs, this quantity is the same for all
jobs present in the interval.  Hence, $S(s,t):=S(t)-S(s)$ is the amount
of service received in the interval $[s,t]$.
For each $k\in \mathcal{K}$, define the measure-valued function of time
$\mu_{k}:[0,\infty)\rightarrow\mathcal{M}$ by
\begin{equation}
  \mu_{k}(t)
  ~=~
  \sum_{j=1}^{Z_{k}(0)}\delta^+_{(v_{k}^{0}(j)-S(t))^{+}}
  +\sum_{i=1}^{A_{k}(t)}\delta^+ _{(v_{k}(i)-(S(\sigma _{k}(i),t))^{+}}
  ~.
  \label{eq:muk-1}
\end{equation}
At each time $t$, $(v_{k}^{0}(j)-S(t))^{+}$ and
$(v_{k}(i)-(S(\sigma _{k}(i),t))^{+}$ are the residual service
times within class $k$ of, respectively $j^{th}$ initial job, and
$i^{th}$ job. Recall that $\delta^+_x$ is the Borel measure on
$\mathbb{R}_+$ with mass one at $x>0$, and that the random measure
$\mu_{k}(t)$ takes values in the space $\mathcal{M}$ of finite,
positive Borel measures on $\mathbb{R}_{+}$. $\mu_{k}(.)$ is
measure-valued stochastic process with simple path in the polish space
$\mathbf{D}([0,\infty ),\mathcal{M})$. In \cite{gromoll2002fluid}, this process is
referred to as the state descriptor.
The number of jobs of class $k$ at time $t$ is given by
\begin{equation} \label{bef-wor}
Z_{k}(t)=\langle 1,\mu _{k}(t)\rangle.
\end{equation}  
Let $W(t)$ be the workload at time $t$, which is the total amount of residual service times of all jobs in the system at the time $t$, plus the sum of their remaining service times when they re-enter the system until their final departure.
The quantity $\langle \chi, e. \mu(t) \rangle$ represents the total amount of residual service times without taking into consideration the future residual service times. 
In our case, the workload at the time $t$ equals  $\langle \chi, \gamma (t) \rangle$, where the process $\gamma(\cdot)$ is defined in Section \ref{Singleclass}.
\subsection{Fluid model and fluid solution}
\label{sect:fluid-mod}
Let $(\alpha,\nu,P)$ be the parameters associated with primitive data,  we will refer to it simply as data. We define the vector $\lambda = Q \alpha$, where $\lambda_k$ represents the global arrival rate
to the class $k$. We also define the load factor of the queue by $\rho = \sum_{k=1}^K \lambda_k \beta_k = e M \lambda$, where $M=diag\{\beta_k, \, k\in \mathcal{K} \}$. 
\par In this paper, we assume that $\rho=1$, which means the data $(\alpha,\nu,P)$  is \textit{critical} \cite[see Section 2.3]{tahar2012fluid}.
 The following definition is an adaptation of  \cite[Definition 2.1]{tahar2012fluid} to our case. By equations (2.12) and (4.20) in \cite{tahar2012fluid},
the time range $t_\rho= +\infty$ if $\xi \neq 0$, and zero otherwise. Note $\mathcal{V}= diag \{\nu_k , k \in \mathcal{K} \}$.

\begin{definition}[Fluid Solution Model]
  \label{def:Fluid-model}
  Let $(\alpha,\nu,P)$ be some data and $\xi \in \mathcal{M}^{c,K}$
  be an initial state. A fluid solution is a triple $(\bar{A}(t),
  \bar{D}(t), \bar{\mu}(t))$ of two real-, and one measure-valued
  vectors of continuous functions: $\bar{A},$ $\bar{D}: \mathbb{R}_{+}\rightarrow
  \mathbb{R}_{+}^{K}$, and
  $\bar{\mu}=(\bar{\mu}_{1},...,\bar{\mu}_{K}):\mathbb{R}_{+}
  \rightarrow \mathcal{M}^{K}$ such that $\bar{\mu}(0)=\xi$, and
  \begin{itemize}
\item[i)] $\bar{A}$ and $\bar{D}$ are increasing componentwise,
\item[ii)] If $\xi \neq 0$. The triple satisfies the relations 
  \begin{align}
    \bar{A}(t) &= \alpha t+P^{\prime }\bar{D}(t)
    \label{eq:bar-arrivee-depart} \\[0.2cm]
    \langle 1,\bar{\mu}(t)\rangle &=\langle 1,\xi \rangle +\bar{A}
    (t)-\bar{D}(t)
    \label{eq:nbre-classek}\\[0.2cm]
    \bar{\mu}(t)(I(x)) &=
    \xi (I(x+\bar{S}(t)))
    +\int_{0}^{t} \mathcal{V} \left(I \left(x+\bar{S}(s,t) \right) \right) \, d\bar{A}(s),
    \label{eq:evol-barmu}
  \end{align}    
for every $x\in \mathbb{R}_{+},$  and  $t \geq 0$, where
 \begin{align}
 \bar{S}(t) =\int^t_0 \varphi  \left(\langle1,e.\bar{\mu}(s)  \rangle \right)  \, ds, ~~ \text{ and } ~~ \bar{S}(s,t):=\bar{S}(t)-\bar{S}(s).
\label{eq:barS}
\end{align}  
\item[iii)]
If $\xi = 0$, the triple satisfies
\begin{equation}
\label{eq:solution_nul} \bar{A}(t)=\bar{D}(t)=\lambda t~~~\bar{\mu}(t)=0 ~~~ \text{ for all } ~ t\geq 0.
\end{equation}
\end{itemize}
\end{definition}


According to \cite[Theorem 3.1]{tahar2012fluid}, for each measure $\xi \in \mathcal{M}^{c,K}$, there exists a unique fluid solution $(\bar{A}(\cdot),\bar{D}(\cdot),\bar{\mu}(\cdot))$
of the model above such that $\bar{\mu}(0)=\xi$. 
Moreover, the fluid analogue of the workload  process
satisfies for all $t \geq 0$
\begin{align*}
\bar{W}(t)=\bar{W}(0)= e(M^{0}+MP'Q)\bar{Z}(0) ~~ \text{ with} ~~
M^0:= diag \{ \langle \chi,\nu_k (0) \rangle, k \in \mathcal{K} \}.
\end{align*}



\section{Heavy traffic Results} \label{MPS-model}
Consider a sequence of multiclass processor sharing queues indexed
by numbers $r\in (0,\infty)$. Assume that this model is defined on
probability space $(\Omega^r,\mathcal{F}^r ,\mathbb{P}^r)$, and to
have the same basic structure as described in
Section~\ref{subsec:queueing-model}. Furthermore, the number of classes
$K$ and the set $ \mathcal{A}= \lbrace k: \alpha^r_k \neq 0 \rbrace$ remain fixed for all $r$. The primitive increments are
denoted by $u^r_k= \{ u^r_k (i), i \geq 1 \}, v^r_k= \{ v^r_k (i), i \geq 1 \}  $ and 
$\varphi^{r,k} = \{ \varphi^r_k (i), i \geq 1 \}$, for all $k \in \mathcal{K}$. The data of the $rth$ queue is $(\alpha^r,\nu^r,P^r)$. 
The aim of this work is to
establish the limit under the diffusion scaling of $(\mu^r,r>0)$, which
is defined by
\begin{equation*}
 \widehat{\mu}^r(t)=\dfrac{1}{r}\, \mu^r(r^2t).
\end{equation*}
\subsection{ Scaling}
The diffusion scaled versions of the queue length is $\widehat{Z}^r(t) = Z^r(r^2t)/r$.
We denote the following diffusion scaled processes:
$$
\begin{array}{ll}
 \widehat{A}^r(t)  =\dfrac{A^r(r^2t) - \lambda^r \, r^2 t }{r},~~~~& \widehat{D}^r(t)  =\dfrac{D^r(r^2t) - \lambda^r \, r^2 t }{r} , \\ \\
\widehat{E}^r(t)=\dfrac{E^r(r^2t)-\alpha^rr^2t}{r}  ,~~~~& \widehat{\Phi}^{k,r}(t)=\dfrac{\Phi^{k,r}(\lceil r^2t \rceil)-P^{k,r} \lceil r^2t \rceil}{r}  .
\end{array}
$$
\par For a given sequence of processes $(X^r, r > 0)$, the fluid-scaled processes are defined by $\bar{X}^r(t) = X^r(rt)/r$, with the exception of $\bar{S}^r(t) := S^r(rt)$.
\subsection{Heavy traffic conditions} \label{sub-3-2}
To establish results on the convergence of the above sequence of stochastic processes, we need the following conditions, which are quite general and standard. Assume that there are $\theta>0$, a nonnegative vectors $\alpha = (\alpha_1, \ldots, \alpha_K), ~a = (a_1, \ldots, a_K)$, and a nonnegative matrix $P=(p_{kl})_{k,l\in \mathcal{K}} $ with $\rho \, (P)<1$ such that for $k\in\mathcal{K}$ and as $r\to \infty$,
\begin{eqnarray}
\mathbb{E}(u_k^r (1))/r  &\longrightarrow& 0.
\label{eq:assum1}\\
\mathbb{E}(u_k^r (2); u_k^r (2) > r ) &\longrightarrow&  0.
\label{eq:assum2}\\
\limsup_{r \rightarrow \infty} \mathbb{E}^r ((u_k^r(2))^{2+\theta}) &<& \infty.
\label{eq:assum3} \\
\left(\alpha^{r},a^r \right) &\rightarrow & \left(\alpha,a \right).
\label{eq:assum4}\\
P^{r}  &\rightarrow& P.
\label{eq:assum5}
\end{eqnarray}
Let $\nu$  be a vector and of probability measure does not charge the origin and let $(\nu^r,r>0)$ be sequance of distribution of service times defined in section     \ref{subsec:queueing-model}. We assume that
\begin{eqnarray}
\langle \chi^{4+\theta} ,\nu\rangle  &<& \infty,
\label{eq:assum6}\\
\nu^{r}  &\stackrel{w}{\longrightarrow}&  \nu ~~\mbox{as}~~ r\longrightarrow \infty,
\label{eq:assum7}\\
(\beta^r,b^r)  &\longrightarrow&   (\beta,b) ~~\mbox{as}~~ r\longrightarrow \infty,
\label{eq:assum8}\\
\limsup_{r\rightarrow\infty}  \, \langle  \chi^{4+\theta}  ,\nu^{r}\rangle  &<& 
\infty,
\label{eq:assum9}
\end{eqnarray}
Let $\lambda^r =Q^r\alpha^r$, $M^r = diag\{\beta^r_k, k\in \mathcal{K}\}$. The assumptions \eqref{eq:assum4},\eqref{eq:assum5} and \eqref{eq:assum8} imply that $\lambda^r \rightarrow \lambda$ and $M^r \rightarrow M$. Define the the  traffic intensity of the $r^{th}$ system by
$\rho^r=e M^r \lambda^r$. Assume the following heavy traffic condition:
\begin{equation}
 \lim_{r\rightarrow\infty}r(1-\rho^r)=\sigma ~~\mbox{for some}~~\sigma \in \mathbb{R}. 
\label{eq:assum10}
\end{equation}
Let $ \Pi$, $\Sigma$ and $(H^k, k\in \mathcal{K})$ be the matrices defined respectively in \eqref{eq:cov-arrivee}-\eqref{eq:cov-service} and \eqref{eq:cov-routing}. The conditions \eqref{eq:assum3}-\eqref{eq:assum4} imply the Lindeberg conditions, which are essential to prove the functional central limit theorem for the triangular arrays $(u_i^r, i \geq 1, r>0)$ \cite[cf.][Theorem 7.3]{billingsley}.
Thus 
\begin{align}
\left( \widehat{E}^r , \widehat{\Phi}^{k,r}  \right)   \Longrightarrow  \left( E^{\ast}
,\Phi^{\ast,k}  \right), \label{eq:cons-assum1}
\end{align}
where $E^{\ast},~\Phi^{\ast,1},\ldots,\Phi^{\ast,K}$ are 
$K+1$ independent driftless $K$-dimensional Brownian Motions with
respectively covariance matrix $\Pi,H^1,...,H^K$ and initial
state $0$. 
This implies the following functional weak law of
large numbers:
\begin{align} \label{eq:cons-assum2}
\left( \bar{E}^r , \bar{\Phi}^{k,r}  \right)   \Longrightarrow  \left( \alpha(\cdot),  
P^k(\cdot) \right),
\end{align}
where $\alpha (t)=\alpha t$ and $P^k(t)=P^k t$  for all $t \geq 0$. 
We also make assumptions regarding the initial state. For each $k\in\mathcal{K}$, let $v^{0,r}_k=(v^{0,r}_k(i),i\geq 1)$ be an i.i.d sequence of service times of jobs who are initially in the system with  distribution $\nu_k^{0,r}$, and let $\bar{Z}^r_k(0)$ denote the initial number of jobs. Assume that
$$ 
v^{0,r}_1\ldots,v^{0,r}_K,v^{r}_1,\ldots,v^{r}_K,\bar{Z}^r_1(0),\ldots,\bar{Z}^r_K(0)
$$
are mutually independent and $\langle \chi, \nu^{0,r} \rangle$. Moreover, assume that there exists a vector $\bar{Z}(0)=(\bar{Z}_1(0), \cdots , \bar{Z}_K(0)) \in \mathbb{R}_+^K$ and a measure-valued vector $\nu^0 = (\nu^0_1, \cdots, \nu^0_K ) \in \mathcal{M}^{c,K}$ such that,
\begin{eqnarray}
\bar{Z}^r (0) &\Rightarrow& \bar{Z}(0), 
\label{eq:assum-init1}\\
\nu^{0,r} &\stackrel{w}{\longrightarrow}& \nu^0,
\label{eq:assum-init2}\\
\limsup_{r\rightarrow\infty} \langle  \chi^{4+\theta}  ,\nu^{0,r}\rangle  &<& 
\infty,
\label{eq:assum-init3}
\end{eqnarray}
For each $k\in\mathcal{K}$, we define the excess lifetime distribution $\nu^e_k$ associated with $\nu_k$ by $\langle\nu^e_k,1_{[0,x]}\rangle=\beta_k^{-1}\int_0^x\langle \nu_k, 1_{[y,\infty)}\rangle dy$. Let $\nu^e=(\nu^e_1, \ldots,\nu^e_K)$ and $\Lambda=diag\{\lambda_k, k\in \mathcal{K}\}$. The invariant manifold associated with $\nu$ is defined by $ \mathcal{M}^{\nu}=\{c M\Lambda \nu^e ~:c\in\mathbb{R}_+\}$. Denote by $\xi = \mathcal{V}^0 \bar{Z}(0)$, where $\mathcal{V}^0=diag\{\nu^0_k , k\in\mathcal{K}\}$ and assume that
\begin{equation}
\xi \in \mathcal{M}^{\nu}.
\label{eq:assum-init4}
\end{equation}
\subsection{Main Result}
\label{main-result}

In this section, we present our main result Theorem \ref{thm:main-result}. For this we need the following definition.
\begin{definition}  \label{def:lift-map}
Let $\Delta^{\nu}:\mathbb{R}_+\rightarrow \mathcal{M}^{K}$ be the lifting map associated with $\eta$ given by
\begin{equation}
 \Delta^{\nu} w=\frac{w}{e( \frac{1}{2} M^{(2)}+MP^{\prime} QM)  \lambda}  M\Lambda \nu^e, 
\end{equation}
where $M^{(2)}=diag\{\langle \chi^2,\nu_k \rangle, k\in\mathcal{K}\}$.
\end{definition}

\begin{theorem}
\label{thm:main-result} 
Consider a sequence of MPS queues as defined in Section \ref{MPS-model}, satisfying primitive assumptions \eqref{eq:assum1}--\eqref{eq:assum10} and initial condition assumptions \eqref{eq:assum-init1}--\eqref{eq:assum-init4}. Then
\begin{equation}
\widehat{\mu}^r(\cdot) \Rightarrow \Delta^{\nu} W^{\ast}(\cdot) ~~\text{as}~~ r\rightarrow\infty,  \label{eq:mu*}
\end{equation}
where $W^{\ast}$ is a reflected Brownian motion with initial value $W^{\ast}(0) = e(M^0+MP'Q)\bar{Z}(0)$, drift $-\sigma$, and variance $\Gamma$ given by
\begin{equation}
\Gamma=e\left(\Lambda\Sigma+MQ\left(\Pi+\sum_{k=1}^K\lambda_kH^k\right)Q' M\right)e'. \label{eq:cov-RBM}
\end{equation}

\end{theorem}

The reflected Brownian motion $W^{\ast}$ arises as heavy trafic approximation of the workload process (cf. Proposition \ref{cv:workload}).

\vspace*{0.6cm}

\cite{harrison1992brownian} have conjectured that for a multiclass queue operating under the processor sharing discipline, the limiting queue length process for each job class is a constant multiple of the reflected Brownian motion that represents the limiting workload process for the queue (see (A.58), (A.60), and (A.61) in \cite{harrison1992brownian}).
In the single-station multiclass case with probabilistic feedback, their conjecture for the diffusion queue length limit can be expressed as follows:
\begin{equation*}
Z^{\ast}(t) \simeq \Delta_F W_F(t), ~~ \text{ with } ~ \Delta_F = \dfrac{2 M \lambda}{e.M^{(2)} \lambda},   \label{ccc}
\end{equation*}
where $W_F(\cdot)$ is a reflected Brownian motion with drift $-R \theta$ and covariance matrix $R^2 \Gamma$. Here,
$$ R= \dfrac{e M^{(2)} \lambda}{2 e(\frac{1}{2} M^{(2)} +MQP'M) \lambda},$$
and $W_F(\cdot)$ arises as the diffusion limit of the workload scaled process defined in \cite[(A.56)]{harrison1992brownian}.
We can clearly see that the conjecture holds in our case. Specifically, the multiplicative constant is $\langle 1, \Delta^v \rangle = R \Delta_F$, and the limiting workload process is the reflected Brownian motion $W^*(\cdot)$, which coincides with the process $W_F(\cdot)/R$.

\section{Proof of the main result}
\label{sect:Main-steps-proof}


\subsection{Mapping to the single class queue}
\label{Singleclass}
 The purpose of this section is to demonstrate the convergence of the diffusion process $\hat{\gamma}^r(\cdot)$ (cf. Proposition \ref{prop:gamma-global}). To achieve this, we use the fact that the process $\gamma^r(\cdot)$ behaves like a single class, and we leverage results from the literature concerning processor-sharing single-class server \cite{gromoll2004diffusion}. 
To establish Proposition \ref{prop:gamma-global}, two key tasks must be addressed: firstly, ensuring that the process $\gamma^r(\cdot)$ meets the conditions outlined in \cite{gromoll2004diffusion}, which requires demonstrating Lemma \ref{lem:dist-V-k}. Secondly, proving the convergence of the diffusion workload $\hat{W}^r(\cdot)$ process, as indicated in Proposition \ref{cv:workload}.
\vspace{0.2cm}
\par For each $k\in\mathcal{K}$, we define the sequences $\{V^r_k(i); i \geq 1\}$, where $V^r_k(i)$ represents the total service time required by the $i$th exogenous job of class $k$ until their departure from the server. Similarly, we define $\{V_k^{0,r}(i); i \geq 1\}$ as the sequence of total service times for initial jobs of class $k$. We then introduce the measure-valued process $\gamma^r$ as follows:
\begin{equation}
  \gamma^r(t)= \sum_{k=1}^{K}\left(
    \sum_{i=1}^{Z^r_{k}(0)}\delta^+_{\left(
        V_{k}^{r,0}(i)-S^r(t) \right)^{+}}
    +\sum_{i=1}^{E^r_{k}(t)}\delta^+ _{\left(
        V^r_{k}(i)-S(U^r_{k}(i),t) \right)^{+}}
  \right)
  ~.
  \label{eq:gamma-r}
\end{equation} 
Here, $U^r_k(i)$ denotes the arrival time of the $i$th exogenous class-$k$ job. The term $(V^r_{k}(i)-S^r(U^r_{k}(i),t))^{+}$ represents the residual service time of the $i$th exogenous class-$k$ job.  This quantity captures the remaining service required for this job across all potential future class visits before its eventual departure from the system. Similarly, $(V_{k}^{r,0}(i)-S^r(t))^{+}$ reflects this for initial jobs of class $k$.
The measure $\gamma^r(t)$ represents the total residual service time, including service required due to future feedback (i.e., until final departure), for all jobs present at time $t$. This includes both exogenous arrivals and jobs initially present in the system. 

 A job is considered present at time $t$ if and only if its residual service time is positive. In particular, the number of jobs present in the system at time $t$ is determined by
\begin{equation}
 \label{eq:mass-total}
\langle 1, \gamma^r(t) \rangle = e\cdot Z^r(t) ~~~\mbox{for}~~t\geq 0.
\end{equation}
This implies that equations \eqref{eq:muk-1} and \eqref{eq:gamma-r} are coupled by \eqref{eq:mass-total}. 
Let $W^r(t)$ denote the workload at time $t$, which is the total amount of residual service times of all jobs in the system at the time $t$, plus the sum of their remaining service times when they re-enter the system until their final departure. This satisfies
\begin{equation}
 \label{eq:workload}
W^r(t)= \langle \chi ,\gamma^r(t) \rangle~~~\mbox{for}~~t\geq0.
\end{equation}
Denote by $\mathcal{B}(t)=(\mathcal{B}_{kl}(t))_{0\leq k,l\leq K}$ the matrix function defined by 
\begin{equation}
\label{eq:mathcal-B}
    \mathcal{B}(t)=\sum_{n\geq 0}\left( BP^{\prime} \right)^{*n}(t),
\end{equation}
where $B(t)=diag\{B_k(t),~k\in\mathcal{K}\}$, and $B_k(t)$ is the distribution function of $\nu_k$. 
 $\mathcal{B}(\cdot)$ aggregates the total service time a job spends in the system, accounting for all possible paths via the feedback. Each term  corresponds to the contribution from jobs who have undergone exactly $n$ transitions between classes.

We recall the following notations which will be used repeatedly in this paper: $\mathcal{V}$ and $\mathcal{V}^0$ are diagonal matrices with elements $\nu_k$ and $\nu^0_k$, respectively. Similarly, $\mathcal{B} \ast \mathcal{V}$ and $\mathcal{B} \ast \mathcal{V}^0$ denote matrices with elements $\mathcal{B}_{kl}\ast\nu_l$ and $\mathcal{B}_{kl}\ast\nu_l^0$ for $k,l\in\mathcal{K}$. The same notation is used when the parameters are indexed by superscript $r$.

 The following lemma is very useful, as the measures $\mathcal{B} \ast \mathcal{V}$ and $\mathcal{B} \ast \mathcal{V}^0$, along with their respective counterparts indexed by $r$, are frequently used 
throughout the paper.
\begin{lemma}
\label{lem:conv-zeta-r}
Consider a sequence of MPS queues as defined in Section \ref{MPS-model}, satisfying assumptions  \eqref{eq:assum5}, \eqref{eq:assum7} and \eqref{eq:assum-init2}. Then, we have
\begin{gather}
\langle 1,\mathcal{B} \ast\mathcal{V}\rangle=Q ~~ \text{ and } ~~ \langle 1,\mathcal{B} \ast\mathcal{V}^0\rangle=Q
 \label{eq:b-ast-nu0}\\
 \langle \chi,\mathcal{B} \ast\mathcal{V}\rangle=QMQ ~~ \text{ and } ~~ \langle \chi,\mathcal{B} \ast\mathcal{V}^0\rangle= Q(M^0+MP'Q)
 \label{eq:b-ast-nu1}\\
 \langle \chi^2,\mathcal{B} \ast\mathcal{V}\rangle=Q(M^2+2MQP^{\prime}M)Q
 \label{eq:b-ast-nu2}\\
 \mathcal{B}^r \ast\mathcal{V}^r \stackrel{w}{\longrightarrow} \mathcal{B} \ast\mathcal{V}.
 \label{eq:b-ast-nu3}
 \end{gather}
If there exists $\theta>0$ such that either $\langle \chi^{\theta},\mathcal{V}\rangle <\infty$, then 
 \begin{equation}
 \langle \chi^{\theta},\mathcal{B} \ast \mathcal{V}\rangle <\infty.
 \label{eq:b-ast-nu4}
 \end{equation}
Furthermore, if $\displaystyle\limsup_{r \rightarrow \infty} \langle \chi^{\theta},\mathcal{V}^r \rangle <\infty $, then
 \begin{equation}
 \displaystyle\limsup_{r \rightarrow \infty} \langle \chi^{\theta}, \mathcal{B}^r \ast\mathcal{V}^r \rangle <\infty. 
 \label{eq:b-ast-nu5}
\end{equation}
The results in \eqref{eq:b-ast-nu3}-\eqref{eq:b-ast-nu5} remain valid when substituting $\mathcal{V}$ and $\mathcal{V}^r$  with $\mathcal{V}^0$ and $\mathcal{V}^{0,r}$.
\end{lemma}
\begin{proof} From the identity $\mathcal{B}=I+BP^{\prime}\ast\mathcal{B}$, the definition of product convolution and  Newton's binomial we have for all $p\in\mathbb{N}$ 
\begin{eqnarray}
 \langle \chi^p,\mathcal{B}\ast\mathcal{V} \rangle &=& \langle \chi^p,\mathcal{V} \rangle+ \langle \chi^p,BP^{\prime}\ast(\mathcal{B}\ast\mathcal{V}) \rangle
 \nonumber\\
  &=&\langle \chi^p,\mathcal{V} \rangle+\int_0^{\infty}\int_0^{\infty}(x+y)^p \mathcal{V}(dx)P^{\prime} \mathcal{B}\ast\mathcal{V}(dy)
  \nonumber\\
  &=& \langle \chi^p,\mathcal{V} \rangle+\sum_{k=0}^pC_p^k \langle \chi^{p-k}, \mathcal{V}\rangle P^{\prime} \langle \chi^k, \mathcal{B}\ast\mathcal{V}\rangle
 \label{eq:moment-entier-1}
\end{eqnarray}
Similarly, we get
\begin{align}
\langle \chi^p,\mathcal{B}\ast\mathcal{V}^0 \rangle =\langle \chi^p,\mathcal{V}^0 \rangle+\sum_{k=0}^p C_p^k \langle \chi^{p-k}, \mathcal{V}\rangle P^{\prime} \langle \chi^k, \mathcal{B}\ast\mathcal{V}^0\rangle. \label{eq:moment-entier-1-initial}
\end{align}
Rewritting \eqref{eq:moment-entier-1} and \eqref{eq:moment-entier-1-initial} for $p=0$, $p=1$ and $p=2$  and using the fact that $\langle 1,\mathcal{V} \rangle=I$, $\langle 1,\mathcal{V}^0 \rangle=I$ and $Q=(I-P^{\prime})^{-1}$ we have respectively  \eqref{eq:b-ast-nu0},  \eqref{eq:b-ast-nu1} and \eqref{eq:b-ast-nu2}.
\par Let $\widehat{B}$ and $\widehat{\mathcal{B}}$ represent the Laplace transform functions of the distribution associated with $\nu$ and the matrix function $\mathcal{B}$, respectively. 
The distribution function of $\mathcal{B}\ast\mathcal{V}$ is denoted by $\mathcal{B}\ast B$, with Laplace transform
\begin{equation*}
 \widehat{\mathcal{B}\ast B}(t)=\widehat{\mathcal{B}}(t).\widehat{B}(t)=\left(I-\widehat{B}(t)P^{\prime}\right)^{-1}\widehat{B}(t)
\end{equation*}
We use the same notation when the functions are indexed by $r$. By assumption, we have $\widehat{B}^r(t)\to \widehat{B}(t)$ and, through a simple argument, we verify that $\widehat{\mathcal{B}^r}(t)\to \widehat{\mathcal{B}}(t)$ for all $t\geq 0$. Therefore,
$\widehat{\mathcal{B}^r\ast B^r}(t)\to \widehat{\mathcal{B}\ast B}(t)$. 
This implies \eqref{eq:b-ast-nu3}
\par Let $p\in\mathbb{N}$ and $q\in\mathbb{N}^*$, by the same argument as to obtain \eqref{eq:moment-entier-1} and using the inequality $(x+y)^{1/q}\leq x^{1/q}+y^{1/q}$ for $x,~y\geq 0$ we have
 \begin{eqnarray*}
  \langle \chi^{p/q},\mathcal{B}\ast\mathcal{V} \rangle&=&\langle \chi^{p/q},\mathcal{V} \rangle+\langle \chi^{p/q},BP^{\prime}\ast(\mathcal{B}\ast\mathcal{V}) \rangle
  \nonumber\\
  &=& \langle \chi^{p/q},\mathcal{V} \rangle+ \int_0^{\infty}\int_0^{\infty}\left((x+y)^{1/q}\right)^p \mathcal{V}(dx)P^{\prime} \mathcal{B}\ast\mathcal{V}(dy)
  \nonumber\\
  &\leq&\langle \chi^{p/q},\mathcal{V} \rangle+\int_0^{\infty}\int_0^{\infty}(x^{1/q}+y^{1/q})^p \mathcal{V}(dx)P^{\prime} \mathcal{B}\ast\mathcal{V}(dy)
  \nonumber\\
  &\leq&\langle \chi^{p/q},\mathcal{V} \rangle+ \sum_{k=0}^pC_p^k \langle (\chi^{1/q})^{p-k}, \mathcal{V}\rangle P^{\prime} \langle (\chi^{1/q})^k, \mathcal{B}\ast\mathcal{V}\rangle
  \nonumber\\
 &\leq&\langle \chi^{p/q},\mathcal{V} \rangle+ P^{\prime}\langle \chi^{p/q}, \mathcal{B}\ast\mathcal{V}\rangle+ \sum_{k=0}^{p-1}C_p^k \langle \chi^{(p-k)/q}, \mathcal{V}\rangle P^{\prime}\langle \chi^{k/q}, \mathcal{B}\ast\mathcal{V}\rangle.
 \end{eqnarray*}
Because $Q=(I-P^{\prime})^{-1}$  exists, we have 
\begin{equation}
\label{eq:moment-rational-1}
 \langle \chi^{p/q},\mathcal{B}\ast\mathcal{V} \rangle\leq Q\langle \chi^{p/q},\mathcal{V} \rangle +Q \sum_{k=0}^{p-1}C_p^k \langle \chi^{(p-k)/q}, \mathcal{V}\rangle P^{\prime}\langle \chi^{k/q}, \mathcal{B}\ast\mathcal{V}\rangle.
\end{equation}
Thus, by induction on $p$, we establish \eqref{eq:b-ast-nu4} for a rational number.
 Let  $\theta$ be be a strictly positive real number. Due to the density of $\mathbb{Q}$ in $\mathbb{R}$, there exists a strictly decreasing sequence $(\theta_n,n\geq 0)$  of rational numbers such that $\theta_n\to\theta$ as $n\to\infty$. 
 Assume that $\langle \chi^{\theta},\mathcal{V} \rangle<\infty$ and prove that \eqref{eq:b-ast-nu4} holds. 
 Considering the increasing sequence of function $\chi^{\theta_n}1_{[0,1]}$. By the Convergence Monotone Theorem, we have   
\begin{equation*}
\label{eq:lim-sup-r1}
\langle \chi^{\theta_n}1_{[0,1]},\mathcal{V} \rangle \nearrow \langle \chi^{\theta}1_{[0,1]},\mathcal{V} \rangle ~~ \text{ as } n\to\infty.
\end{equation*}
Similarly, for the decreasing sequence of functions
$\chi^{\theta_n}1_{(1,\infty)}$,  again by the Monotone Convergence Theorem, we have 
\begin{equation*}
\label{eq:lim-sup-r2}
\langle \chi^{\theta_n}1_{(1,\infty)},\mathcal{V} \rangle\searrow \langle \chi^{\theta}1_{(1,\infty)},\mathcal{V} \rangle ~~ \text{ as } n\to\infty.
\end{equation*}
Consequently, 
$$\langle \chi^{\theta_n},\mathcal{V} \rangle \longrightarrow \langle \chi^{\theta},\mathcal{V} \rangle ~~ \text{ as } n\to\infty.$$  
 Hence, there exists $n_0$ such that for all $n\geq n_0$, we have $\langle \chi^{\theta_n},\mathcal{V} \rangle<\infty$, implying
$ \langle \chi^{\theta_n},\mathcal{B}\ast\mathcal{V} \rangle<\infty$ for all $n\geq n_0$.
Therefore, as we let $n \rightarrow \infty$, we obtain 
$$ \langle \chi^{\theta},\mathcal{B}\ast\mathcal{V} \rangle < \infty.$$
\par The proof of \eqref{eq:b-ast-nu5} follows by employing a similar argument to that used in proving \eqref{eq:b-ast-nu4}.
\end{proof}

In the following lemma, we present the conditions under which the diffusion approximation of $\gamma^r(\cdot)$ holds.
\begin{lemma} 
\label{lem:dist-V-k}
 For each $k\in\mathcal{K}$ and $r>0$, both sequences $\{V^r_k(i); i \geq 1\}$ and $\{V^{0,r}_k(i); i \geq 1\}$ are  i.i.d with distribution respectively 
\begin{equation*}
\label{law-nu-k}
\zeta^r_k   = (e(I-P^{\prime,r})(\mathcal{B}^r \ast \mathcal{V}^r))_k~~\mbox{and}~~\zeta^{0,r}_k  = (e(I-P^{\prime,r})(\mathcal{B}^r \ast \mathcal{V}^{0,r}))_k,
\end{equation*}
Moreover, for each $k\in\mathcal{K}$ we denote by $\zeta_k = (e(I-P^{\prime})(\mathcal{B} \ast \mathcal{V}))_k$ and $\bar{\gamma}(0)=\zeta^0.\bar{Z}(0)$. If conditions \eqref{eq:assum5}-\eqref{eq:assum9} and \eqref{eq:assum-init1}-\eqref{eq:assum-init4} hold. Then we have 
\begin{gather}
 \langle \chi^{4+\theta},\zeta_k\rangle<\infty
 \label{eq:zeta-1}\\
 \limsup_{r \rightarrow \infty} \langle\chi^{4+\theta},\zeta^r_k\rangle<\infty
 \label{eq:zeta-2}\\
 \zeta^r_k \stackrel{w}{\longrightarrow}  \zeta_k ~~\mbox{as}~~r\to\infty
  \label{eq:zeta-3} \\
  \langle\chi, \zeta^r_k \rangle \longrightarrow \langle\chi, \zeta_k \rangle
  \label{eq:zeta-4} \\
  \langle\chi^2 ,\zeta^r_k \rangle \longrightarrow \langle\chi^2, \zeta_k \rangle
  \label{eq:zeta-5} \\
  (\bar{\gamma}^r(0),\langle\chi,\bar{\gamma}^r(0)\rangle,\langle\chi^{1+\theta},\bar{\gamma}^r(0)\rangle)\Rightarrow  (\bar{\gamma}(0),\langle\chi,\bar{\gamma}(0)\rangle,\langle\chi^{1+\theta},\bar{\gamma}(0)\rangle).
 \label{eq:zeta-6}
\end{gather}
\end{lemma}
\begin{proof}
The first part of this lemma is proven in Lemma C.1 \cite{tahar2012fluid}. 
 Conditions \eqref{eq:zeta-1}-\eqref{eq:zeta-5}
  follow from conditions
\eqref{eq:assum5}-\eqref{eq:assum9} and Lemma \ref{lem:conv-zeta-r}.  
From \eqref{eq:gamma-r}, we have
 \begin{equation}
  \bar{\gamma}^r(0)= \frac{1}{r} \sum_{k=1}^{K}  \sum_{i=1}^{r\bar{Z}^r_{k}(0)}\delta^+_{V_{k}^{r,0}(i)}.
 \end{equation}
Fix $k\in\mathcal{K}$. The sequence $\{V^{0,r}_k(i); i \geq 1\}$ is i.i.d. with distribution $\zeta_k^{0,r}$, which converges weakly to $\zeta_k^{0}$ as $r\to\infty$ by condition \eqref{eq:assum-init2} and Lemma \ref{lem:conv-zeta-r}.
Therefore, for any continuous and bounded function $g:\mathbb{R}^+\to\mathbb{R}$,
\begin{align}
\mathbb{E}(g(V^{0,r}_k(i))) = \langle g, \zeta_k^{0,r} \rangle \rightarrow  \mathbb{E}(g(V^{0}_k(i))) = \langle g, \zeta_k^{0} \rangle <\infty  \label{bound:EV0}
\end{align}
This implies that
\begin{align*}
\lim_{r \rightarrow \infty} \mathbb{E}\left( g(V^{0,r}_k(i)) \,;\, g(V^{0,r}_k(i))>r \right) = 0.
\end{align*}
Then, by the weak law of large numbers for triangular arrays, we have
\begin{align*}
\dfrac{1}{r} \sum_{i=1}^{\lfloor rt \rfloor} g(V^{0,r}_k(i))\Rightarrow \langle g,\zeta^0_k \rangle t.
\end{align*}
Since the sequence of random variables $\bar{Z}^r(0)$ satisfies \eqref{eq:assum-init1}, the random time change formula yields
\begin{align*}
\dfrac{1}{r} \sum_{i=1}^{ r \bar{Z}^r_k(0)} g(V^{0,r}_k(i))\Rightarrow \langle g,\zeta^0_k \rangle \bar{Z}_k(0).
\end{align*}
Therefore,
\begin{align}
\langle g, \bar{\gamma}^r(0)\rangle\Rightarrow\langle g, \bar{\gamma}(0)\rangle, \label{eq:cv-gamma_g}
\end{align} 
for all continuous and bounded functions  $g:\mathbb{R}^+\to\mathbb{R}$.
Therefore, component of \eqref{eq:zeta-6} satisfies
 \begin{equation}
  \label{eq:cv-gamma0}
  \bar{\gamma}^r(0)\Rightarrow\bar{\gamma}(0).
 \end{equation}
On the other hand, by \eqref{eq:assum-init2}-\eqref{eq:assum-init3} and Lemma \ref{lem:conv-zeta-r}, we have 
$$ \mathcal{B}^r \ast \mathcal{V}^{0,r} \stackrel{w}{\longrightarrow}  \mathcal{B} \ast \mathcal{V}^{0}  ~~ \text{ and } ~~ \limsup_{r \rightarrow \infty} \langle \chi^{4+\theta}, \mathcal{B}^r \ast \mathcal{V}^{0,r} \rangle < \infty.$$
Hence, by applying Lemma 3.5 \cite{gromoll2004diffusion} to the measure $\mathcal{B}^r \ast \mathcal{V}^{0,r}$, we obtain 
\begin{equation}
\langle \chi^{1+p}, \mathcal{B}^r \ast \mathcal{V}^{0,r} \rangle \longrightarrow \langle \chi^{1+p}, \mathcal{B} \ast \mathcal{V}^0 \rangle < \infty, \label{bound:Bv0}
\end{equation} 
 for all $p \in [0,3+\theta)$. 
 Consequently
 \begin{equation}
\langle \chi,\zeta^{0,r}\rangle\rightarrow\langle \chi,\zeta^{0}\rangle~~\mbox{and}~~\langle \chi^{1+\theta},\zeta^{0,r} \rangle\rightarrow\langle \chi^{1+\theta},\zeta^{0} \rangle.
 \label{bound:zet0}
\end{equation} 
Following the same argument from \eqref{bound:EV0} to \eqref{eq:cv-gamma_g},
 we obtain
\begin{equation}
 \label{eq:cv-gamma1}
 \langle \chi,\bar{\gamma}^r(0)\rangle\Rightarrow\langle \chi,\bar{\gamma}(0)\rangle~~\mbox{and}~~\langle \chi^{1+\theta},\bar{\gamma}^r(0)\rangle\Rightarrow\langle \chi^{1+\theta},\bar{\gamma}(0)\rangle.
\end{equation}
As the limits in \eqref{eq:cv-gamma0} and \eqref{eq:cv-gamma1} are determinstic, then \eqref{eq:zeta-6} holds.
\end{proof}

By Proposition 5.2 \cite{tahar2012fluid}, the sequence of the fluid scaled state descriptor $\bar{\gamma}^r (\cdot)$ converges in distribution to $\bar{\gamma}(\cdot)$ as $r\rightarrow\infty$, where $\bar{\gamma}$ is the fluid solution of the following fluid model:
\begin{itemize}
 \item[(i)] $\bar{\gamma}:[0,\infty)\to\mathcal{M}^K$ is continuous function such that
 for all $t\geq 0$ and $x\geq 0$,
 \begin{equation}
 \bar{\gamma}(t)(I(x))
    =
    \bar{Z}(0) \cdot \zeta^0 (I(x+\bar{S}(t)))
    + \int_{0}^{t}\alpha \cdot \zeta \left( I(x+\bar{S}(s,t)) \right) d s~~~~~~  
    \label{eq:bar-gamma}
\end{equation}
\item[(ii)] $\bar{\gamma}(0)=0$  then $\bar{\gamma}(t)=0$ for all $t\geq 0$
\end{itemize}

\begin{remark}~\label{rem:gamma}
\begin{itemize}
 \item[(i)] The function $\bar{S}(\cdot)$ in \eqref{eq:bar-gamma} is the one defined in
\eqref{eq:barS}. Since $\langle 1,\gamma^r(t) \rangle = e.Z^r(t)$ then  
$ \bar{S}(t) = \int_0^t \varphi (\langle 1,\bar{\gamma} (s) \rangle ds$.
\item[(ii)] The equation \eqref{eq:bar-gamma} can be interpreted as the fluid equation in the single class case, with law of
service $\alpha.\zeta/e.\alpha$,  initial state $\bar{\gamma}(0)=\bar{Z}(0) \cdot \zeta^0$ and external arrival rate $e.\alpha$. Consequently, the steady-state results of \cite{puha2004invariant} can be applied.  Under mild conditions, the critical fluid solution $\bar{\gamma}(t)$ converges to the invariant state $\frac{1}{\langle \chi, (\alpha \cdot \zeta)^e \rangle} \langle \chi, \bar{\gamma}(0) \rangle (\alpha \cdot \zeta)^e$.  By Lemma \ref{lemma:zeta-properties} (i) and (ii), this implies that the critical fluid solution $\bar{\gamma}(t)$ converges to the invariant state $e(I-P')(\mathcal{B} * \Delta^{\nu})\bar{W}(0)$.
\end{itemize}
\end{remark}
\begin{proposition}
\label{cv:workload}
Consider a sequence of MPS queues as defined in Section \ref{MPS-model}, satisfying primitive assumptions \eqref{eq:assum1}-\eqref{eq:assum10} and initial condition assumptions \eqref{eq:assum-init1}-\eqref{eq:assum-init4}. Denote by $\widehat{W}^r(t) = W^r(r^2t)/r$ the diffusion scaled of the workload process. Then $\widehat{W}^r(\cdot)$ converges in distribution to $W^{\ast}(\cdot)$ as $r\rightarrow\infty$, where $W^{\ast}(\cdot)$ is the reflected Brownian motion defined in Theorem \ref{thm:main-result}.
\end{proposition}
\begin{proof}
 Let $\mathbb{V}^r_{k}(n)=\sum_{i=1}^nV^r_k(i)$ and $ L^r(t)=\sum_{k\in\mathcal{K}}\mathbb{V}^r_{k}(E_k^r(t))-t$. Denote by $\widehat{\mathbb{V}}^r_{k}(t)=(\mathbb{V}^r_{k}(\lfloor r^2 t \rfloor)-\langle\chi, \zeta_k \rangle\lfloor r^2 t \rfloor)/r $ the diffusion scaled process of $\mathbb{V}^r_{k}(t)$. The diffusion scaled process $\widehat{L}^r(t)=L^r(r^2t)/r$ satisfies,
\begin{equation*}
 \widehat{L}^r(t) =\sum_{k=1}^K\left(\widehat{\mathbb{V}}^r_{k}(\bar{E}^r_k(t))+(eM^rQ^r)_k\widehat{E}_k^r(t)\right)-(1-\rho^r)rt.
\end{equation*}
The diffusion scaled of the workload process $W^r(t)$ satisfies the nonidling discipline equation, 
\begin{equation}
\label{eq:nonidling}
 \widehat{W}^r(t)= \widehat{W}^r(0)+\widehat{L}^r(t)+ \sup_{0\leq s\leq t}{\{(\widehat{W}^r(0)+\widehat{L}^r(s))^-\}}.
\end{equation}
Conditions \eqref{eq:zeta-2}-\eqref{eq:zeta-5} imply, via functional central limit theorem for triangular arrays, the following weak convergence:
\begin{equation}
 \widehat{\mathbb{V}}^r_{k}(\cdot)\Rightarrow \mathbb{V}^{\ast}_{k}(\cdot)~~\mbox{as}~~r\longrightarrow \infty,
 \label{eq:cv-V-k}
\end{equation}
for each $k\in\mathcal{K}$, where $\mathbb{V}^{\ast}_k (\cdot)$ is a Brownian motion with drift $0$ and variance  $\langle\chi^2, \zeta_k \rangle-(\langle\chi, \zeta_k \rangle)^2$. 
By the weak convergences \eqref{eq:cons-assum1}, \eqref{eq:cons-assum2} and \eqref{eq:cv-V-k}, and the fact that the processes $E^{\ast}_k(\cdot)$ and $\mathbb{V}^{\ast}_k(\cdot)$ are independant, and the process $\alpha_k (\cdot)$ is deterministic, then 
\begin{align*}
\left( \hat{E}^r_k(\cdot), \hat{\mathbb{V}}^r_k(\cdot), \bar{E}^r_k(\cdot)  \right) \Rightarrow \left( E^{\ast}_k(\cdot), \mathbb{V}^{\ast}_k(\cdot), \alpha_k (\cdot) \right) ~~~\text{for each } k \in \mathcal{K}.
\end{align*}
It follows from the random time change theorem (cf. \cite{billingsley} Section 17) and the continuous mapping theorem (cf. \cite{billingsley}, Theorem 5.1) that the weak convergence $\widehat{L}^r(t)\Rightarrow L^{\ast}(t)$ holds as $r\rightarrow \infty$, where $L^{\ast}$ is a Brownian motion with drift $-\sigma$ and variance
\begin{equation}
\label{eq:variance-L}
\sum_{k=1}^K\left(\langle\chi^2, \zeta_k \rangle-(\langle\chi, \zeta_k \rangle)^2\right)\alpha_k+(\langle\chi, \zeta_k \rangle)^2a_k\alpha_k^3,
\end{equation}
which is equal to the variance $\Gamma$ given by \eqref{eq:cov-RBM}. In fact, from conditions \eqref{eq:b-ast-nu1}, \eqref{eq:b-ast-nu2} and by using the fact that $\Lambda e^{\prime}= Q \alpha$ and $I+QP'=Q$, and after some calculations, it follows that
\begin{gather}
 \sum_{k=1}^K\langle\chi^2, \zeta_k \rangle\alpha_k = e\Sigma\Lambda e^{\prime}+eMQ(\Lambda-P^{\prime}\Lambda P)Q^{\prime} M e^{\prime}
 \label{eq:var-L1}\\
 \sum_{k=1}^K(\langle\chi, \zeta_k \rangle)^2\alpha_k= eMQD_{\alpha}Q^{\prime}Me^{\prime}
 \label{eq:var-L2}\\
 \sum_{k=1}^K(\langle\chi, \zeta_k \rangle)^2a_k\alpha_k^3=eMQ \Pi Q^{\prime}Me^{\prime},
 \label{eq:var-L3}\\
 \Lambda-D_{\alpha}-P'\Lambda P=\sum_{k=1}^K\lambda_k H^k
 \label{eq:var-L4}
\end{gather}
where $D_{\alpha}=diag\{\alpha_k\}$ and the matrices $\Pi,~\Sigma$ and $(H_k,k\in\mathcal{K})$ are defined by \eqref{eq:cov-arrivee}, \eqref{eq:cov-service} and \eqref{eq:cov-routing}. 
 Next, it follows from \eqref{eq:zeta-6} that $\widehat{W}^r(0)=\langle\chi,\bar{\gamma}^r(0)\rangle\Rightarrow\langle\chi,\bar{\gamma}(0)\rangle$, so the continuous mapping theorem  applied to \eqref{eq:nonidling} implies the result.
\end{proof}

\par Given that the measure-valued process $\gamma^r$ defined by \eqref{eq:gamma-r} evolves similarly to the corresponding process in the single-class processor sharing model, and given the conditions of Lemma \ref{lem:dist-V-k} (which are consistent with the assumptions in \cite{gromoll2004diffusion}), along with the convergence of the diffusion-scaled workload process established in the previous proposition, we have the following proposition:

 \begin{proposition}
  \label{prop:gamma-global}
Consider a sequence of MPS queues as defined in Section \ref{MPS-model}, satisfying primitive assumptions \eqref{eq:assum1}-\eqref{eq:assum10} and initial condition assumptions \eqref{eq:assum-init1}-\eqref{eq:assum-init4}. Denote by $\widehat{\gamma}^r(t) = \gamma^r(r^2t)/r$ the diffusion scaled of $\gamma^r$. Then 
\begin{align*}
\widehat{\gamma}^r(\cdot) \Rightarrow e(I-P') (\mathcal{B} \ast \Delta^{\nu}) W^{\ast}(\cdot) ~~~\text{ as } ~~ r\rightarrow\infty,
\end{align*}
 where $W^{\ast}(\cdot)$ is the reflected Brownian motion defined in Theorem \ref{thm:main-result}.
\end{proposition}

\subsection{Descriptor to the multiclass queue}
\label{subsec:mult-queue}  

The primary aim of this section is to demonstrate Theorem \ref{th:cv-visits}, which establishes the convergence of the diffusion-scaled process $\hat{\mathcal{Q}}^r(\cdot)$. 
\vspace{0.2cm}
\par For each $l,k\in \mathcal{K}$ and $i\geq 1$, let $N^{r,l}_{k}(i)$ denote the total number of visits to class $k$ by the $i^{th}$ job entering the system as a job of class $l$.
 For each $n=1,\ldots,N^{r,l}_{k}(i)$, let
$V^r_{lk}(i,n)$ represent the sum of service times required by this job from its arrival until its $n^{th}$ visit to class $k$
(included). 
Specifically, $V^{0,r}_{lk}(i,n)$ denotes the total service time required by the $i^{th}$ initial job of class $l$ until its $n^{th}$ visit to class $k$. Let $i\geq1$, $n=1,\ldots,N^{r,l}_k$, and denote by $U^r_{l}(i)$ the exogenous arrival epoch of the $i^{th}$ job of class $l$. 
\par The  properties of the distributions of $(N^{r,l}_{k}(i),i\geq 1)$, $(V^r_{lk}(i,n);i,n\geq 1)$, and $(V^{0,r}_{lk}(i,n);i,n\geq 1)$ are provided in Lemmas D.1, D.2, and D.3 \cite{tahar2012fluid}. 
In the following lemma, we present additional properties that are useful in our analysis.

\begin{lemma}~
 \label{lem:pro-N-V-lk}
\begin{itemize}
 \item[1.] For each $l\in\mathcal{K}$, denote by $N^l(i) = (N^l_{1}(i) ,N^l_{2}(i), \ldots, N^l_{K}(i) )$. The sequence $\{ N^l(i), i \geq 1 \}$ is i.i.d with mean $\mathbb{E}(N^l)=(Q_{kl},k\in\mathcal{K})$ and covariance matrix $B^l$  with entries $B_{k\ell}^l=Q_{k\ell} Q_{\ell l} + (Q P^{\prime})_{\ell k} Q_{kl} - Q_{kl} Q_{\ell l}$ satisfies
\begin{equation}
\label{eq:cov-B-H}
\sum_{l=1}^K \alpha_lB^l = Q \left(\sum_{l=1}^K \lambda_l H^l \right) Q',
\end{equation}
where $\{ H^l, l \in \mathcal{K} \}$ are the matrices defined in \eqref{eq:cov-routing}.
\item[2.] For each $n\geq 1$ and $k,l\in\mathcal{K}$, the sequence $(V_{lk}(i,n),i\geq 1)$ is i.i.d such that for all function $g:\mathbb{R}^+\to\mathbb{R}^+$,
\begin{equation*}
\mathbb{E} \left( \sum_{n=1}^{N^l_k(i)} g(V_{lk}(i,n)) \right) =\langle g,\mathcal{B}_{kl}\ast\nu_l\rangle.
\end{equation*}
\end{itemize}
\end{lemma}
\begin{proof}~
\begin{itemize}
 \item[1.]    
 Let $(X_n)_{n \geq 0}$ be a homogenous Markov chain with state space $\mathcal{K}$ and transition matrix $P$. For each $l,k \in \mathcal{K}$,  $N^l_{k}$ is the number of visits to state $k$ from state $l$ of the markov chain $X_n$. Hence, $N^l_{k}$ can be expressed as
 $$ N^l_{k} = \sum_{ n \geq 0}  1_{\lbrace X_n = k \lvert X_0=l\rbrace}.   $$
 The sequence $(N^l(i),i\geq 1)$ is i.i.d. because routing events of different jobs are independent.
Given that the spectral radius of $P$ is less than 1, one can easily obtain
$\mathbb{E}(N^l_k) = Q_{kl}$. 
Furthermore, by performing a series of calculations using some known properties of Markov chains, we derive
$$ \mathbb{E}(N^l_k,N^l_{\ell}) =Q_{k\ell} Q_{\ell l} + (Q P^{\prime})_{\ell k} Q_{kl}.  $$
Therefore,
 $$cov(N^l_k,N^l_{\ell})=Q_{k\ell} Q_{\ell l} + (Q P^{\prime})_{\ell k} Q_{kl} - Q_{kl} Q_{\ell l}=B^l_{k \ell}.$$ 
By definition of the matrix $B^l$, and after performing some matrix calculations, we obtain 
\begin{align*}
 \sum_{l=1}^K \alpha_lB^l = \Lambda + P^{\prime}Q \Lambda + \Lambda Q^{\prime}P - Q D_{\alpha} Q^{\prime}
 = Q\left(\Lambda - D_{\alpha} - P^{\prime} \Lambda P\right) Q^{\prime}.
\end{align*}
This, in conjuction with \eqref{eq:var-L4} implies \eqref{eq:cov-B-H}.
\item[2.] See Lemma D.3 \cite{tahar2012fluid}.
\end{itemize}
\end{proof}
Following the approach in \cite[Section 5.2]{tahar2012fluid}, we introduce a family of state descriptors $\gamma^r_{lk}(t)$ for each $l,~k \in \mathcal{K}$ and $t\geq 0$, defined as: 
\begin{equation}
  \gamma^r_{lk}(t) =
  \sum_{i=1}^{Z^r_l(0)}\sum_{n=1}^{N^{l,r}_{k}(i)}
  \delta^+_{(V^{0,r}_{lk}(i,n)-S^r(t) )}
  +
  \sum_{i=1}^{E^r_l(t)}\sum_{n=1}^{N^{l,r}_{k}(i)}
  \delta^+_{(V^r_{lk}(i,n)-S^r(U^r_l(i),t) )}.
  \label{eq:gamma^r_k0k}
  \end{equation}
The meaning of this equation can be explained as follows.  
Consider the \( i \)th job arriving in class \( l \) at time \( U^r_l(i) \). At a given time \( t \), this job has completed their \( n \)th visit to class \( k \) (which results in a departure from class \( k \)) if and only if the condition  
$
V^r_{lk}(i,n) \leq S^r(U^r_l(i),t)
$  
is met.  
Furthermore, at time \( t \), the expression  
$
(V^r_{lk}(i,n) - S^r(U^r_l(i),t))^+
$  
represents the remaining service required before the job finishes their \( n \)th visit to class \( k \). This can be interpreted as a form of residual service time, but it is specifically measured relative to this \( n \)th departure from class \( k \). Similar reasoning applies to the initial jobs in the system.
   Let $\mathcal{Q}^r(t)=(\mathcal{Q}^r_1(t),\ldots,\mathcal{Q}^r_K(t))$, where
 $$\mathcal{Q}^r_k(t) = \sum_{l=1}^K\gamma^r_{lk}(t) ~~ \text{ For  each } ~  k\in\mathcal{K}.$$
  We will be interested below in the diffusion approximation of the descriptor 
  $\mathcal{Q}^r$ taking values in $\mathbf{D}([0,\infty),\mathcal{M}^K)$ and satisfies
  \begin{equation}
  \mathcal{Q}^r(t)(I(x))=\mathcal{Q}^r(0)(I(x+S^r(t)))+\sum_{l=1}^K\sum_{i=1}^{E^r_l(t)}\vartheta^{r,l}(i)(I(x+S^r(U^r_l(i),t))),
  \label{eq:mathcal-Q}
\end{equation}
where for each $l\in\mathcal{K}$, $r>0$ and $i=1,2,\ldots$ the random element  $\vartheta^{r,l} (i)$ takes values in $\mathcal{M}^K$ and defined for each $k\in\mathcal{K}$ by
\begin{equation}
\label{es:mathca-X}
\vartheta_{k}^{r,l} (i):= \sum_{n=1}^{N_{k}^{r,l} (i)}  \delta_{V^r_{lk}(i,n)}^+ 
\end{equation} 
and $\mathcal{Q}^r(0)$ is the initial state of \eqref{eq:mathcal-Q}, which is given in function of the initial data of the MPS queue by
\begin{equation}
  \mathcal{Q}^r_k(0)=\sum_{l=1}^K\sum_{i=1}^{Z^r_l(0)}\sum_{n=1}^{N^{r,l}_{k}(i)}
  \delta^+_{V^{0,r}_{lk}(i,n)}.
  \label{eq:mathcal-Q-0}
\end{equation}
Note that $\vartheta_{lk}^r (i)(\{0\})=0$ and $\mathcal{Q}^r(t) (\{0\})=0 $, since $V^r_{lk}(i,m) >0$ and the fact that no jobs will be in the system if the residual job service $\left( V^r_{lk}(i,n)-S^r(U^r_l(i),t) \right)$ equals zero. This implies that
\begin{equation}
\langle 1, \vartheta_{lk}^r (i) \rangle = \langle 1_{(0,\infty)}, \vartheta_{lk}^r (i) \rangle ~~ \text{ and } ~~ \langle 1, \mathcal{Q}^r(t) \rangle = \langle 1_{(0,\infty)}, \mathcal{Q}^r(t) \rangle.
\label{1glob}
\end{equation}

 Denote by $\bar{\mathcal{Q}}^r (t)=\mathcal{Q}^r(rt) / r$. Then, by Proposition 5.3 \cite{tahar2012fluid}, the sequence of the fluid scaled state descriptor $\bar{\mathcal{Q}}^r (\cdot)$ converges in distribution to $\bar{\mathcal{Q}}(\cdot)$ as $r\rightarrow\infty$, where $\bar{\mathcal{Q}}$ is the fluid solution of the following fluid model:
\begin{itemize}
 \item[(i)] $\bar{\mathcal{Q}}:[0,\infty)\to\mathcal{M}^K$ is continuous function such that for all $t\geq 0$ and $x\geq 0$,
 \begin{equation}
 \label{fluid-mathcal-Q}
 \bar{\mathcal{Q}}(t)(I(x))
    =
    \mathcal{B}\ast\mathcal{V}^0 (I(x+\bar{S}(t))) \bar{Z}(0)+\int_{0}^{t} 
\mathcal{B}\ast \mathcal{V} (I(x+\bar{S}(s,t)))ds~\alpha.
\end{equation}
\item[(ii)] $\bar{\mathcal{Q}}(0)=0$  then $\bar{\mathcal{Q}}(t)=0$ for all $t\geq 0$.
\end{itemize}
Moreover, the fluid solution $\bar{\mathcal{Q}}(t)$ satisfies:
\begin{align}
\langle 1,\bar{\mathcal{Q}}(t) \rangle = Q \bar{Z}(t). \label{glob_mathcal_Q}
\end{align}

\par The following theorem establishes the state space collapse for the measure-valued process $\mathcal{Q}^r$, which is essential for proving the convergence of the diffusion-scaled process in Theorem \ref{th:cv-visits}.

\begin{theorem} 
\label{ssc2} 
 Consider a sequence of MPS queues as defined in Section \ref{MPS-model}, satisfying primitive assumptions \eqref{eq:assum1}-\eqref{eq:assum10} and initial condition assumptions \eqref{eq:assum-init1}-\eqref{eq:assum-init4}. For all $T>0$, 
  \begin{equation*}
  \sup_{ t \in [0,T] }\mathbf{d} \left( \widehat{\mathcal{Q}}^r(t) ,  \mathcal{B}\ast\Delta^{\nu} \widehat{W}^r(t) \right) \Rightarrow 0 ~~ \text{as}~~r\rightarrow\infty.
  \end{equation*}
  \end{theorem}

\vspace{0.2cm}

The proof of Theorem \ref{ssc2} is lengthy and highly technical. In fact, the entire section \ref{Sec5Q} is dedicated to its proof, with the final result presented at the end, following the establishment of all necessary lemmas.


\vspace{0.2cm}
 
We now state this section's main theorem.

\begin{theorem}
  \label{th:cv-visits}
Consider a sequence of MPS queues as defined in Section \ref{MPS-model}, satisfying primitive assumptions \eqref{eq:assum1}-\eqref{eq:assum10} and initial condition assumptions \eqref{eq:assum-init1}-\eqref{eq:assum-init4}. Denote by $\widehat{\mathcal{Q}}^{r}(t)= \mathcal{Q}^r(r^2t)/r$ the diffusion scaled version of $\mathcal{Q}^r(t)$. Then 
\begin{align*}
\widehat{\mathcal{Q}}^r(\cdot) \Rightarrow \mathcal{B}\ast\Delta^{\nu} W^{\ast}(\cdot) ~~~~ \text{ as } ~~ r\rightarrow\infty.
\end{align*}
  \end{theorem}

\begin{proof}[Proof of Theorem \ref{th:cv-visits}]
By Proposition \ref{cv:workload}, and 
 $\widehat{W}^r (\cdot) \Rightarrow  W^{\ast} (\cdot)$  as $r \rightarrow \infty$.
By continuity of the map $\mathcal{B}\ast\Delta^{\nu}$ and the continuous mapping theorem, we have
$$
\mathcal{B}\ast\Delta^{\nu}  \, \widehat{W}^r (\cdot) \Rightarrow \mathcal{B}\ast\Delta^{\nu} \,  W^{\ast} (\cdot).
$$
Therefore, the result of Theorem \ref{th:cv-visits} follows from Theorem \ref{ssc2} and the \textit{convergence together lemma}  \cite[Theorem 4.1]{billingsley}. 
\end{proof}

\subsection{Proof of Theorem \ref{thm:main-result}} \label{sect:proof_main_res}

This section contains the proof of the paper’s main result. The following proposition is essential
for this proof. 

\begin{proposition}
Consider a sequence of MPS queues as defined in Section \ref{MPS-model}, satisfying primitive assumptions \eqref{eq:assum1}--\eqref{eq:assum10} and initial condition assumptions \eqref{eq:assum-init1}--\eqref{eq:assum-init4}. Then
\label{theo:conv-of-A-D-Z}
\begin{equation}
\label{cv:A^r-D^r-Z^r}
\left(\widehat{A}^{r}(\cdot),\widehat{D}^{r}(\cdot),\widehat{Z}^{r}(\cdot)\right) \Rightarrow \, \left(A^*(\cdot),D^*(\cdot),
Z^*(\cdot) \right)~~\text{as}~~ r\rightarrow\infty,
\end{equation} 
where  
\begin{gather}
D^*(t) = Q \left( \bar{Z}(0) +  \sum_{l=1}^K \Phi^{*,l} (\lambda_l t) + E^{*}(t)  - Z^*(t) \right),
\label{eq:Z*-A*-D*} \\[0.2cm]
A^*(t) = Q \left( P^{\prime}(\bar{Z}(0)-Z^*(t)) + E^*(t) + \sum_{l=1}^{K} \Phi_l^{*,l}( \lambda_l t)\right),
\label{eq:input-output*} \\[0.2cm]
Z^*(t)= \langle1,\Delta^{\nu}\rangle W^{\ast}(t),
\label{eq:Z*-Q*}
\end{gather}
\end{proposition}

 The proof of Proposition \ref{theo:conv-of-A-D-Z} uses the following Lemma. 

\begin{lemma}
\label{C-tight}
Let $\mathscr{C}^K$ be the subspace of the continuous functions in $\mathscr{D}^K$.
 \begin{itemize}
  \item[i)] If a sequence of stochastic processes $(X_n,n\geq1)$ is C-tight, then for
any of its subsequences, there exists a weakly convergent further subsequence. Furthermore, if $X^{\ast}$ is the weak limit of a subsequence $(X_{n_k},k\geq1)$, then $\mathbb{P}(X^{\ast}\in\mathscr{C}^K)=1$.
\item[ii)] If $(X_n,n\geq1)$  converges weakly to $X^{\ast}$ and $X^{\ast}$ is almost surely in $\mathscr{C}^K$ then $(X_n,n\geq1)$ is C-tight.
\item[iii)] If both $(X_n,n\geq1)$ and $(Y_n,n\geq1)$  are C-tight, so is $(aX_n+bY_n,n\geq 1)$, where $a$ and $b$ are any given real numbers.
\item[iv)]  If both $(X_n,n\geq1)$ and $(Y_n,n\geq1)$  are C-tight. Then, the sequence $\{ (X_n,Y_n), n \geq 1 \}$ is C-tight.
 \end{itemize}
\end{lemma}

\begin{proof}~
 Refer to Proposition 4.1 and Lemma 4.2 in \cite{chen2000} for properties (i)-(iii). Property (iv) follows from (iii).
\end{proof}

\begin{proof}[Proof of Proposition \ref{theo:conv-of-A-D-Z}]~ 

The proof is going to be divided into two parts. In the first part, we prove the tightness of $(\hat{A}^r(t),\hat{D}^r(t),\hat{Z}^r(t))$. The second part demonstrates that all subsequences converge to a unique limit $(A^{\ast}(t),D^{\ast}(t),Z^{\ast}(t))$, where the processes $D^{\ast}(t)$, $A^{\ast}(t)$,  and $Z^{\ast}(t)$ are defined in \eqref{eq:Z*-A*-D*}-\eqref{eq:Z*-Q*}.
\paragraph*{Part 1: Tightness}~
For each $i\geq 1$ and $l \in \mathcal{K}$, let $N^{r,l}(i) = (N_{1}^{r,l}(i), N_{2}^{r,l}(i), \ldots, N_{K}^{r,l}(i) )$ be the vector of component $N^{r,l}_k(i)$ which is defined at the beginning of Section \ref{subsec:mult-queue}. For each $t\geq 0$, define
\begin{equation}
 \label{N-t}
\mathscr{N}^r(t)= \sum_{l=1}^K \left( \sum_{i=1}^{Z_l^r (0)} N^{r,l}(i) + \sum_{i=1}^{E_l^r (t)}  N^{r,l}(i) \right).
\end{equation}
For each $k\in\mathcal{K}$, $\mathscr{N}_k^r(t)$ represent the total number of visits to class $k$ by time $t$ , brought by external arrivals, as well as by initially present jobs. Note that if $V^r_{lk}(i,n)\leq S^r(U^r_{k}(i),t)$ for all $l\in\mathcal{K}$ and $n=1,\ldots, N^{r,l}(i)$ then the $i^{th}$ job has completed their services. So for each $k\in\mathcal{K}$,
\begin{equation}
\label{eq:d-N-Q}
D_k^r (t) = \mathscr{N}_k^r (t) - \mathbf{Q}_k^r(t),
\end{equation}
where $ \mathbf{Q}^r_k(t)=\langle 1,\mathcal{Q}^r_{k}(t)\rangle$ is  
the remaining number of visits to class $k$ for jobs present in the $r^{th}$ system
at time $t$. 
 By Theorem \ref{th:cv-visits}, we have the following convergence:
\begin{align}
\widehat{\mathbf{Q}}^r(\cdot)\Rightarrow Q\langle1,\Delta^{\nu}\rangle W^{\ast}(\cdot) ~~~~ \text{ as } ~~ r\rightarrow\infty.  \label{corr_Q}
\end{align}

Let $\widehat{N}^{r,l}(t)= \left(\sum_{i=1}^{ \lfloor r^2 t \rfloor} (N^{r,l} (i) - Q^r_{\bullet l})\right)/r$ and $\widehat{\mathscr{N}}^r(t)=\left(\mathscr{N}^r(r^2t)-\lambda^rr^2t\right)/r$ be the diffusion scaled processes. A diffusion scaled version of \eqref{N-t} can be written as
\begin{equation}
\label{widehat-N-t}
\widehat{\mathscr{N}}^r(t) =\sum_{l=1}^K\widehat{N}^{r,l}(Z_l^r(0)/r^2) +\sum_{l=1}^K\widehat{N}^{r,l}(\bar{E}_l^r(t))+ Q^r \widehat{Z}^r(0) +  Q^r \widehat{E}^r(t).
\end{equation}
The process $\{ \widehat{N}^{r,l}(\cdot); ~ l\in\mathcal{K} \}$ satisfies the functional central limit theorem,
\begin{equation}
\label{cv-term-E-N-l}
\widehat{N}^{r,l}(\cdot)\Rightarrow N^{*,l}(\cdot)~~~~\mbox{as}~r\to\infty,
\end{equation}
where $N^{*,l}(\cdot)$ is a  $K$-dimensional Brownian motion with the covariance matrix $B^l$ defined in Lemma \ref{lem:pro-N-V-lk}.
By \eqref{eq:cons-assum2} and \eqref{eq:assum-init1},
$$
Z^r(0)/r^2 \Rightarrow 0 ~~\mbox{and }~~ \bar{E}_l^r(t) \Rightarrow \alpha(t)=\alpha t~ \text{ for all } t\geq 0.
$$
These convergences, in conjunction with \eqref{cv-term-E-N-l}, imply, by the Time Change Theorem \cite{billingsley}, 
\begin{equation}
\label{cv-term-1-2}
\sum_{l=1}^K\widehat{N}^{r,l}(Z_l^r(0)/r^2)\Rightarrow 0~~\mbox{and }~~ \sum_{l=1}^K\widehat{N}^{r,l}(\bar{E}_l^r(t))\Rightarrow \sum_{l=1}^{K}  N^{*l} (\alpha_l t).
\end{equation}
Additionally, by assumptions \eqref{eq:assum5} and \eqref{eq:cons-assum1}, we have
\begin{equation}
 \label{cv-term-3-4}
  Q^r \widehat{Z}^r(0)\Rightarrow Q\bar{Z} (0)~~\mbox{and }~~Q^r\widehat{E}^r(\cdot)\Rightarrow QE^{\ast}(\cdot).
\end{equation}
Since the limits in \eqref{cv-term-1-2} and in \eqref{cv-term-3-4} have a continuous path, Lemma \ref{C-tight} (ii) and (iii) imply that the process $\hat{N}^r(\cdot)$ defined in \eqref{widehat-N-t} is tight.
Combining this with \eqref{corr_Q}, we deduce from Lemma \ref{C-tight} (iii) that the process $\widehat{D}^r(\cdot)$, the diffusion-scaled version of the process in \eqref{eq:d-N-Q}, is also tight.

\vspace*{0.15cm}

By taking equations 
(\ref{eq:arrivee-depart-1}) and (\ref{eq:nbre-client-1}) in diffusion scaling, we get
\begin{align}
\widehat{A}^r (t)&= \widehat{E}^r (t) + \sum_{l=1}^K \widehat{\Phi}^{r,l} (\bar{D}_l^r (t)) + P^{',r} \widehat{D}^r (t), 
\label{eq:A^r-D^r1}\\[0.2cm]
\widehat{Z}^r (t)&= \widehat{Z}^r (0) + \widehat{A}^r (t) - \widehat{D}^r (t).
\label{eq:Z^r-A^r-D^r1}
\end{align} 

By Theorem 5.1 in \cite{tahar2012fluid}, the fluid process $\bar{D}^r(\cdot)$ converges in distribution to the deterministic process $\bar{D}(\cdot)$, the second component in the triplet fluid solution. Combining this convergence with assumptions \eqref{eq:assum5}, \eqref{eq:cons-assum1}, and \eqref{eq:assum-init1}, and with the tightness of $\widehat{D}^r(\cdot)$, we deduce from Lemma \ref{C-tight} (iii) that the processes $\widehat{A}^r(\cdot)$ and $\widehat{Z}^r(\cdot)$ are tight. 
 Therefore, by Lemma \ref{C-tight} (iv), the sequence $(\hat{A}^r(\cdot), \hat{D}^r(\cdot), \hat{Z}^r(\cdot), r>0)$ is tight.
\vspace{0.15cm}

\paragraph*{Part 2: Convergence to a unique limit}~
Having established tightness, we now demonstrate that all subsequences of $(\hat{A}^r(t),\hat{D}^r(t),\hat{Z}^r(t))$ converge to the unique limit  $(A^{\ast}(t),D^{\ast}(t),Z^{\ast}(t))$ defined in \eqref{eq:Z*-A*-D*}-\eqref{eq:Z*-Q*}.
Let $(\hat{A}^{r_i}(t), \hat{D}^{r_i}(t), \hat{Z}^{r_i}(t)) $ be a convergent subsequence, such that
$$
(\hat{A}^{r_i}(t), \hat{D}^{r_i}(t), \hat{Z}^{r_i}(t)) \Rightarrow (\tilde{A}(t), \tilde{D}(t), \tilde{Z}(t)) ~~~\text{ as } ~~~ r \rightarrow \infty.
$$
By \eqref{widehat-N-t} and the convergences \eqref{cv-term-E-N-l}-\eqref{cv-term-3-4}, we have the following convergence:
\begin{equation*}
\hat{N}^r(t) \Rightarrow N^{\ast}(t):= \sum_{l=1}^{K} N^{l} (\alpha_l t) + Q\bar{Z}(0) + Q E^{\ast}(t).
\end{equation*}
Taking the diffusion scaling in \eqref{eq:d-N-Q} and using \eqref{corr_Q} and the above convergence, we obtain
\begin{align}
\tilde{D}(t):= N^{\ast}(t) - Q\langle1,\Delta^{\nu}\rangle W^{\ast}(t), \label{IdDD1}
\end{align}
as $r \rightarrow \infty$ along the subsequence. From \eqref{cov:Phi} and \eqref{eq:cov-B-H}, we have
$
\sum_{l=1}^{K} N^{\ast,l} (\alpha_l t)=Q\sum_{l=1}^K \Phi^{,l} (\lambda_l t),
$
leading to
\begin{align}
\tilde{D}(t) =Q \left( \bar{Z}(0) + \sum_{l=1}^K \Phi^{\ast,l} (\lambda_l t) + E^{\ast}(t) - \langle1,\Delta^{\nu}\rangle W^{\ast}(t) \right). \label{D*:2}
\end{align}
Since the limit in \eqref{IdDD1} is unique and does not depend on the chosen subsequence, we have
\begin{align}
\hat{D}^r(t) \Rightarrow \tilde{D}(t). \label{IdDD2}
\end{align}
Having established the convergence of $\hat{D}^r(t)$, we now turn our attention to $\tilde{A}(t)$ and $\tilde{Z}(t)$. By \eqref{eq:A^r-D^r1}-\eqref{eq:Z^r-A^r-D^r1}, and the fact that \eqref{IdDD2} implies $\bar{D}^r(t) \Rightarrow \lambda t$ as $r\to\infty$, we obtain
\begin{gather}
\tilde{A} (t) =E^{\ast}(t) + \sum_{l=1}^K \Phi^{\ast,l} (\lambda_l t) + P^{'} D^{\ast} (t), \label{A*:Z0} \\[0.2cm]
\tilde{Z} (t) = \bar{Z}(0) + \tilde{A}(t) - D^{\ast} (t). \label{A:Z*}
\end{gather}
Substituting \eqref{D*:2} into \eqref{A*:Z0}-\eqref{A:Z*} and using $Q(I-P')=I$, we get
\begin{gather}
\label{eq:A^r}
\tilde{A} (t) = Q \left( P^{\prime}(\bar{Z}(0)-\langle1,\Delta^{\nu}\rangle W^{\ast}(t)) + E^{\ast}(t) + \sum_{l=1}^{K} \Phi_l^{\ast,l}( \lambda_l t)\right),\\
\tilde{Z}(t) = \langle1,\Delta^{\nu}\rangle W^{\ast}(t). \label{Zast1}
\end{gather}
Analogously to the argument for $\hat{D}^r(t)$, the limits in \eqref{eq:A^r} and \eqref{Zast1} are unique and do not depend on the subsequence. Therefore, we have
\begin{align*}
(\hat{A}^{r}(t), \hat{D}^{r}(t), \hat{Z}^{r}(t)) \Rightarrow (A^{\ast}(t), D^{\ast}(t), Z^{\ast}(t)) ~~~\text{ as } ~~~ r \rightarrow \infty,
\end{align*}
where $D^{\ast}(t)$, $A^{\ast}(t)$, $Z^{\ast}(t)$ are respectively the limits in \eqref{D*:2}, \eqref{eq:A^r} and \eqref{Zast1} (with the notation changed).

\end{proof}

\paragraph*{\textbf{Proof of Theorem \ref{thm:main-result}}}

To prove the convergence of the diffusion-scaled measure $\hat{\mu}^r(t)$, we must establish state-space collapse, i.e., $\hat{\mu}^r (\cdot) \approx \Delta^{\nu} \widehat{W}^r(\cdot)$ for large $r$.  This result, combined with the convergence of the diffusion-scaled workload process (cf. Proposition \ref{cv:workload}), yields our main result via the convergence-together theorem \cite[Theorem 4.1]{billingsley}.
To achieve the state-space collapse, we use the framework of the shifted fluid scaled process introduced by \cite{bramson1998state},
where the diffusion-scaled process $\hat{\mu}^r(t)$ is expressed as $\hat{\mu}^r(t) = \bar{\mu}^{r,m}(s)$, with $\bar{\mu}^{r,m}(s) := \bar{\mu}^{r}(m+s)$. 
We must show that for large $r$, there exists a high-probability set such that:
i) The sample paths of the family of shifted fluid-scaled processes $\{\bar{\mu}^{r,m}(\cdot), m \leq rT\}$ are tight.
ii) These paths are uniformly approximated by fluid model solutions.
Furthermore, we must demonstrate that, for all initial states, the fluid model solutions $\bar{\mu}(t)$ converge uniformly to an invariant state as $t \to \infty$.

In Proposition \ref{theo:conv-of-A-D-Z}, we established the joint convergence $(\hat{A}^r, \hat{D}^r,\hat{Z}^r) \Rightarrow (A^*,D^*,Z^*)$. The limit process $A^*$, defined in \eqref{eq:input-output*}, is a diffusion process due to its continuous paths and the independence of the reflected Brownian motions (RBMs) $W^*$, $E^*$, and $\Phi^*$.  Consequently, we can adapt the proofs for the single-class case by substituting the arrival process $A^r_k$ for the exogenous arrival process $E^r_k$ and the arrival rate $\lambda_k$ for the exogenous arrival rate $\alpha_k$.
Furthermore, Proposition \ref{prop:gamma-global} and Equation \eqref{eq:mass-total} guarantee the boundedness of the fluid-scaled and shifted fluid-scaled total queue sizes, $e \cdot \bar{Z}^r(\cdot)$ and $e \cdot \bar{Z}^{r,m}(\cdot)$, respectively.  This provides crucial bounds for proving the precompactness of the shifted scaled processes \( \{ \bar{\mu}^{r,m}(\cdot), m \leq rT \} \). 

The proofs of points i) and ii) adapt standard arguments for single-class systems, particularly those in \cite{gromoll2007,zhang2011diffusion}, which offer refined and simplified proofs compared to \cite{gromoll2004diffusion}, to the specific context of processor sharing.
The remaining step to establish state-space collapse is proving the uniform convergence of the fluid model solutions to the invariant state. This is established in Theorem \ref{pro:unif-conv-mu}; its proof, provided in \ref{unif_conv_sect_app}, differs significantly from the single-class literature and uses the state descriptor $\mathcal{Q}$ as an intermediate step.





\clearpage


\appendix
\renewcommand{\thesubsection}{\Alph{section}.\arabic{subsection}}

\renewcommand{\thedefinition}{\Alph{section}.\arabic{definition}}
\renewcommand{\thetheorem}{\Alph{section}.\arabic{theorem}}
\renewcommand{\theproposition}{\Alph{section}.\arabic{proposition}}
\renewcommand{\thelemma}{\Alph{section}.\arabic{lemma}}

\section{Convergence to the invariant states of the fluid model}
\label{unif_conv_sect_app}

We consider the multiclass fluid model defined in Section \ref{sect:fluid-mod} associated with critical data $(\alpha,\nu,P)$ and initial state $\xi \in \mathcal{M}^{c,K}$.  The goal is to establish a uniform convergence to the invariant states of the fluid somutions $\bar{\mu}(\cdot)$ and $\bar{\mathcal{Q}}(\cdot)$  over a compact set of initial conditions defined by \eqref{comp-set_unif}, which will be needed in the proof of the state space collapse for the measure valued processes $\mu^r(\cdot)$ and $\mathcal{Q}^r(\cdot)$.

This section is organized as follows: Section \ref{inv-stat-sec} defines and characterizes invariant states. Section \ref{ma_re_Ap} states the main results. Section \ref{backg} provides background on important results essential for proving these main results and establishes several supporting lemmas. The proofs of the main results are then presented in Sections \ref{pr_asmp} and \ref{pr_convmu}.

\subsection{Invariant states}  \label{inv-stat-sec}~

\begin{definition}
Let $\xi \in \mathcal{M}^{c,K}$. Let $(\bar{A},\bar{D},\bar{\mu})$ be a fluid solution  for critical data $(\alpha,\nu,P)$ and initial state $\xi$. The measure $\xi$ is called an invariant state, if
\begin{equation}
\label{eq:def-invariant-state} \bar{\mu}(t)=\xi~~~~\text{ for all } t\geq 0.
\end{equation}
\end{definition}
\begin{proposition} 
\label{invst}
The measure $\xi$  is an invariant state, if and only if assumption \eqref{eq:assum-init4} holds. Furthermore for some $c \geq 0$, we have
 \begin{equation}
 \label{eq:invar-state-Q}
\bar{ \mathcal{Q}} (t)=c \, \mathcal{B}\ast\mathcal{V}^{e} M\lambda ~~ \text{ for all } t\geq 0,
 \end{equation}
where $\bar{ \mathcal{Q}}(\cdot)$ is the fluid solution to the equation \eqref{fluid-mathcal-Q}.
\end{proposition}
\proof{Proof} 
If $\xi = 0$, then \eqref{eq:assum-init4} holds, and $\bar{\mu}(t)=0$ is the fluid solution.
Let us consider $\xi \neq 0$. Suppose that $\xi$ is an invariant state, and show that  \eqref{eq:assum-init4} holds.
 By equations \eqref{eq:bar-arrivee-depart} and \eqref{eq:nbre-classek}, we have 
 $\bar{A}(t)=\bar{D}(t)=\lambda t.$
   From equation \eqref{eq:barS}, we have 
  $\bar{S}(t)=t/\langle1,e.\xi\rangle.$ Taking this in equation \eqref{eq:evol-barmu} involves, for all $x \in [0,\infty)$ and $t\geq 0$ 
\begin{equation}
\xi(I(x))=\xi(I(x+t/\langle1,e.\xi\rangle))+\langle1,e.\xi\rangle  M \Lambda \nu^e \left( [x,x+t/\langle1,e.\xi\rangle )  \right).
\label{eq:proof-invariant-state}
\end{equation}
Since $\xi$ is a finite measure, by letting  $t\longrightarrow \infty$, the first member in the right hand of \eqref{eq:proof-invariant-state} becomes null and
 $\xi (I(x))=\langle 1,e.\xi\rangle M \Lambda \nu^e(I(x)).$
Thus, the condition \eqref{eq:assum-init4} holds for $c= \langle1,e.\xi\rangle$.
\par Reciproquely, let $\xi\in \mathcal{M}^K$ be such that \eqref{eq:assum-init4} holds.
Since $\rho=1$, we have 
$\xi=\langle1,e.\xi\rangle M \Lambda \nu^e.$
Let 
$\bar{\mu}(t)=\xi,~~\bar{A}(t)=\bar{D}(t)=\lambda .$ 
By a straightforward computation $(\bar{A},\bar{D},\bar{\mu})$ is a fluid solution. Thus $\xi$  is an invariant state, if and only if assumption \eqref{eq:assum-init4} holds.
 
\par Let us prove the second part of the proposition. If $\xi = 0$, then by \eqref{fluid-mathcal-Q}, $\bar{\mathcal{Q}}(\cdot)=0$, therefore,  \eqref{eq:invar-state-Q} is true for $c=0$. Now, we consider the case where $\xi \neq 0$. Since $\bar{\mu}(t)=\xi$ for all $t\geq 0$, we have $\bar{S}(s,t)=c^{-1}(t-s)$ for all $s\leq t$. Therefore, the equation \eqref{fluid-mathcal-Q} becomes
  \begin{equation}
\label{eq:Q-invariant}
\bar{\mathcal{Q}}(t)(I(x))=\mathcal{B}\ast\mathcal{V}^0 (I(x+c^{-1}t)) \bar{Z}(0)+ c\int_{x}^{x+c^{-1}t} \mathcal{B}\ast \mathcal{V} (I(s))ds~\alpha.
\end{equation}
The following integral will be used in the above equation,
\begin{align}
 \int_{0}^{y}(\mathcal{B}\ast(I-B))(s) \, ds & = \int_{0}^{y}\int_{0}^{s}
 \dot{\mathcal{B}}(u)(I-B)(s-u)\,du \, ds 
 = \int_{0}^{y}\int_{u}^{y}\dot{\mathcal{B}}(u)(I-B)(s-u) \, ds \, du \nonumber \\[0.2cm]
 &= \int_{0}^{y}\dot{\mathcal{B}}(u)\left(\int_{u}^{y}(I-B)(s-u) \, ds \right) du 
 = \int_{0}^{y}\dot{\mathcal{B}}(u)B^{e}(y-u)\,du ~M \nonumber \\[0.2cm] \label{int-compu1}
 &=  \mathcal{B}\ast B^{e}(y)~M.
\end{align}
Note that $\bar{Z} (0):=\langle 1,\xi\rangle=c M \lambda $ and $\langle 1,e.\xi\rangle=c$. Replacing $\bar{Z}(0)$ by $c M\lambda$ in the first term on the right hand of \eqref{eq:Q-invariant} and replacing the second term by the above expression, one obtains 
\begin{align*}
 \bar{\mathcal{Q}}(t)(I(x))& = c\, (Q-\mathcal{B}\ast B^e)(x+c^{-1}t)\,  M\lambda+ c\, \mathcal{B}\ast B^e(x+c^{-1}t)\, M\lambda - c\, \mathcal{B}\ast B^e(x)\,M\lambda\\[0.2cm]
 &= c\, (Q-\mathcal{B}\ast B^e)(x)\, M \lambda  = c\, \mathcal{B}\ast\mathcal{V}^e(I(x))  M\lambda.
\end{align*}

\subsection{Uniform convergence to the invariant states}~ \label{ma_re_Ap}

For each $p,N>0$ , we define the following compact set:
\begin{align} \label{comp-set_unif}
\mathsf{B}^p_N:= \{ \xi \in \mathcal{M}^{c,K}: \left| \langle 1,\xi \rangle \vee \langle \chi,\xi \rangle \vee \langle \chi^{1+p},\xi \rangle \right| \leq N \}.
\end{align}

 The following two are the main results of this section, with Proposition \ref{asmp} serving as a key component in the proof of Theorem \ref{pro:unif-conv-mu}.

\vspace*{0.2cm}
\begin{theorem}
\label{pro:unif-conv-mu}
Fix $p,N>0$ and assume $ \langle \chi^2, \nu \rangle < \infty$. For each $\xi \in \mathsf{B}^p_N$, we denote by  
 $\bar{\mu}^{\,\xi}(\cdot)$  the third component of the fluid solution of the multiclass fluid model  with critical data $(\alpha,\nu,P)$ and initial state $\xi$. We have 
\begin{gather}
\lim_{t \rightarrow \infty} \, \sup_{\xi \in \mathsf{B}^p_N} \mathbf{d}\left(\bar{\mu}^{\,\xi} (t) , \Delta^{\nu} \bar{W}(0) \right) = 0.
\label{eq;cv-uniform-mu}
\end{gather}
\end{theorem}

\vspace*{0.1cm}

\begin{proposition} \label{asmp}
Fix $p,N>0$ and Assume $ \langle \chi^2, \nu \rangle < \infty$. For each initial state $\xi \in \mathsf{B}^p_N$ of the fluid model associated with critical data $(\alpha,\nu,P)$, we denote by $ \bar{\mathcal{Q}}^{\,\xi}(\cdot)$  the fluid solution of the equation \eqref{fluid-mathcal-Q} such that $\bar{\mathcal{Q}}^{\,\xi}(0)=\mathcal{B} \ast \xi$. We have
\begin{align}
\lim_{t \rightarrow \infty} \sup_{\xi \in \mathsf{B}^p_N} \mathbf{d} \left(  \bar{\mathcal{Q}}^{\,\xi}(t), \mathcal{B} \ast \Delta^{\nu} \bar{W}(0) \right) =0.
\label{eq:conv-gamma_lk}
\end{align}
\end{proposition}

\subsection{Background} \label{backg}~

To prove the results, we need to recall some definitions and results from \cite{tahar2012fluid}.
First, by  Lemma 4.1 in that paper, the cumulative function $\bar{S}(t)$ defined in \eqref{eq:barS}  
 is continuous, strictly increasing and differentiable on $[0,t^{\ast})$ where $t^*:= \inf \{ t: e \cdot \bar{\mu}(t)=0 \}$ represent the first time at which the fluid queue empties. Moreover, we have 
 $\lim_{t \rightarrow t^{*}} \bar{S}(t) =+\infty$. Later, in that paper, it was shown that for the critical case, $t^*= + \infty$. 
 We consider the function $T:[0,+\infty) \mapsto [0,+\infty)$ defined as
$$ 
T(u):=\bar{S}^{-1}(u):= \inf \{ t \geq 0:\bar{S}(t) > u \}. ~~~ 
$$
This function is also continuous and differentiable  and strictly increasing, such that
 $\dot{T}(u)= e.\bar{Z}(T(u))$.  In \cite{tahar2012fluid}, they obtain a differential equation, of the renewal type, of which the function $T$ is a solution
\begin{align}
\dot{T}(u)=e(I-P')(\mathcal{B} \ast C)(u) \bar{Z}(0) + (K \ast T)(u), \label{diff_eq_T}
\end{align}
with 
\begin{align}
K(u)=e(I-P')(\mathcal{B} \ast (I-B))(u) \lambda , ~~ C(u)=(I-B^0(u)) \bar{Z}(0)+ (I-B(t)) QP'.
\label{K_C:def}
\end{align}
We rewrite some expression:
$$
(\mathcal{B} \ast C)(t)= Q - (\mathcal{B} \ast B^0) (t) = \mathcal{B} \ast \mathcal{V}^0 (I(t)), ~~ \text{and} ~ (\mathcal{B} \ast (I-B))(t) Q = \mathcal{B} \ast \mathcal{V}(I(t)).
$$
The differential equation \eqref{diff_eq_T} is equivalent to say that the function $T$ satifies the following convolution equation 
\begin{equation} \label{conv_eq_T}
T(u)= H^{\xi}(u) + (B^e_s \ast T)(u),
\end{equation}
where 
$$
B_e^s(t)= \int_0^t K(u)du= e (I-P')(\mathcal{B} \ast \mathcal{V}^e)([0,t]) M \lambda,
$$
and 
$$ H^{\xi}(u) =  \int_0^u e(I-P')(\mathcal{B} \ast C)(s) \bar{Z}(0)ds = \int_0^u e(I-P') \mathcal{B} \ast \mathcal{V}^0 (I(s)) \bar{Z}(0)ds .
$$
The notation $\xi$ in $H^{\xi}$ is because the function $H^{\xi}$ depends on the initial state $\xi$ as follows:
\begin{equation} \label{H-fct-def}
H^{\xi}(t)= e (I-P') \int_{0}^t \mathcal{B} \ast \xi (I(y))dy.~~~\text{for all}~~~t\in \mathbb{R}_+.
\end{equation}
Define the renewal function
\begin{equation}
U(t)=\sum_{n=0}^{\infty}(B^e_s)^{\ast n}(t)~~~\text{for all}~~t\geq0,\label{rene:funct}
\end{equation} 
where $(B_e^s)^{*0}(\cdot) \equiv 1$ and $(B_e^s)^{*i}(\cdot) = \left( (B_e^s)^{*i-1} \ast B_e^s \right) (\cdot)$ for each $ i \in \{1,2,3, \cdots \}$. 
Then, the convolution function \eqref{conv_eq_T}, has a unique locally bounded solution given by
\begin{equation}
\label{eq:inv-cum-service} 
T(t)= (H^{\xi}\ast U)(t).
\end{equation}
Define the following measures:
\begin{align*}
\zeta = e(I-P') (\mathcal{B} \ast \mathcal{V}), ~~~\text{ and } ~~ \zeta^0 = e(I-P') (\mathcal{B} \ast \mathcal{V}^0). 
\end{align*}

The measure $(\alpha.\zeta)^e$ is the excess life time measure of $(\alpha \cdot \zeta)$ (cf. (ii) in Lemma \ref{lemma:zeta-properties} ), and 	$B^e_s$ is
the  excess life time distribution. The distribution function $B^e_s$  has a finite mean (cf. (i) in Lemma \ref{lemma:zeta-properties}),
\begin{align} \label{wss3}
\int_0^{\infty} y B^e_s (dy) = \langle \chi, (\alpha.\zeta)^e \rangle = \dfrac{\langle \chi^2, \alpha.\zeta \rangle}{2 \langle \chi, \alpha.\zeta \rangle} =  e \left( \frac{1}{2} M^{(2)} + M P' Q M \right) \lambda <\infty. 
\end{align}


\begin{lemma} \label{lemma:zeta-properties} We have
\begin{itemize}
\item[(i)] $\langle \chi, \alpha.\zeta \rangle=1$ and 
$ \langle \chi^2, \alpha.\zeta \rangle = e \left(  M^{(2)} + 2 M P' Q M \right) \lambda$.
\item[(ii)] $\langle 1_{ [0,x] } , (\alpha . \zeta)^e \rangle = e (I- P' ) (\mathcal{B} \ast \mathcal{V}^e) ([0,x]) M \lambda$. 
\end{itemize}

\end{lemma}
\begin{proof}
(i) follows directly from Lemma C.1  \cite{tahar2012fluid}.
Let us prove (ii). We have
\begin{align*}
 \int_{0}^{y}(\mathcal{B}\ast(I-B))(s) \, ds & = \int_{0}^{y} \int_{0}^{s}
 \dot{\mathcal{B}}(u)(I-B)(s-u)\,du \, ds 
 = \int_{0}^{y}\int_{u}^{y}\dot{\mathcal{B}}(u)(I-B)(s-u) \, ds \, du  \\[0.2cm]
 &= \int_{0}^{y}\dot{\mathcal{B}}(u)\left(\int_{u}^{y}(I-B)(s-u) \, ds \right) du 
 = \int_{0}^{y}\dot{\mathcal{B}}(u)B^{e}(y-u)\,du ~M  \\[0.2cm] 
 &=  \mathcal{B}\ast B^{e}(y)~M.
\end{align*}
Since $ \mathcal{B}\ast \mathcal{V} (I(x))= (\mathcal{B}\ast(I-B))(x)Q$, then
\begin{align}
\int_{0}^{y} \mathcal{B}\ast \mathcal{V} (I(s)) \, ds = \mathcal{B}\ast \mathcal{V}^{e}([0,y])~MQ. \label{int-compu1}
\end{align}

We have
\begin{align*}
(\alpha \cdot \zeta)^e ([0,x]) &= \dfrac{1}{\langle \chi, \alpha\cdot \zeta \rangle} \, \int_0^x (\alpha.\zeta) (I(y)) dy = \int_0^x \zeta (I(y)) dy \, \alpha \nonumber  \\[0.2cm]
&=  \int_0^x \, e (I- P' ) (\mathcal{B} \ast \mathcal{V}) (I(y)) \, \alpha \, dy  = e (I- P' ) (\mathcal{B} \ast \mathcal{V}^e) ([0,x])  MQ \alpha \nonumber \\[0.2cm] 
& =  e (I- P' ) (\mathcal{B} \ast \mathcal{V}^e) ([0,x]) M \lambda.
\end{align*}
\end{proof}

For each $x \geq 0$, we define the function
\begin{align} \label{Kx-fct-def}
K^x (y)= (\mathcal{B} \ast\mathcal{V})(I(x+y)) \, \alpha.
\end{align}

The following Lemma is needed in the proof of Proposition \ref{asmp}. Recall that in this lemma, the integrals and the convolution are componentwise.
\begin{lemma} \label{ver-unif-cond-lm2}
Fix $p,N>0$ and let $\mathsf{B}^p_N$ be the set defined in \eqref{comp-set_unif}. Let $U$ the renewal function defined in \eqref{rene:funct}. Let $H^{\xi}$ and $K^x$ the functions defined in 
\eqref{H-fct-def} and \eqref{Kx-fct-def} respectively. Then,
\begin{align} 
\sup_{\xi \in \mathsf{B}^p_N} \sup_{x \geq 0}   \left| \big(\left(K^x\ast H^{\xi} \right) \ast U \big) (t) - \dfrac{\bar{W}(0)}{\langle \chi, (\alpha.\zeta)^e \rangle}  \, \mathcal{B}\ast \mathcal{V}^e (I(x)) M \lambda \right| \longrightarrow 0 ~~ \text{ as } t \rightarrow \infty. \label{Kx_H_U}
\end{align}
\end{lemma}
\begin{proof}
Recall that from Equation \eqref{wss3}, we have $\langle \chi, (\alpha \cdot \zeta)^e \rangle < \infty $.
Since $K^x$ and $\dot{H}^{\xi}$ are nonincreasing, and $H^{\xi}$ is continuously differentiable, it follows that:
\begin{align}
\int_0^{\infty} \dot{H}^{\xi}(z)dz =  H^{\xi}(\infty) =   e (M^0+ MP'Q) \bar{Z}(0):= \bar{W}(0)<\infty ~~ \text{for all } \xi \in \mathsf{B}^q_N. \label{int:H_xi}
\end{align}
We will now prove inequality \eqref{Kx_H_U} for all $x\geq 0$ and $\xi \in \mathsf{B}^p_N$. We begin by bounding $(K^x \ast H^{\xi})(y)$:
\begin{align*}
(K^x \ast H^{\xi})(y) & \leq K^x (y/2) \int_0^{\infty} \dot{H}^{\xi}(u)du + \dot{H}^{\xi}(y/2) \int_{y/2}^y K^x(y-z) dz \\[0.2cm]
& \leq K^x (y/2) \, \bar{W}(0) + \dot{H}^{\xi}(y/2) \, \Big( (\mathcal{B}\ast \mathcal{V}^e) (I(x)) - (\mathcal{B}\ast \mathcal{V}^e) (I(x+y/2))
\Big) M \lambda.
\end{align*}
Since $ \langle 1, \mathcal{B} \ast \mathcal{V}^e \rangle= Q $, we have:
\begin{align}
(K^x \ast H^{\xi})(y) \leq K^x (y/2) \bar{W}(0) + \dot{H}^{\xi}(y/2) QM\lambda. \label{equiv1:K_H}
\end{align}
We have:
\begin{align*} 
\int_0^{\infty} (K^x \ast H^{\xi})(y) \, dy & =
\int_0^{\infty} \int_0^{y} K^x(y-z) dH^{\xi}(z) dy =
\int_0^{\infty} \int_0^{y}  K^x(y-z) \dot{H}^{\xi}(z) dz dy  \\[0.2cm]
& = \int_0^{\infty} \int_z^{\infty} K^x(y-z) dy \dot{H}^{\xi}(z) dz  = \int_0^{\infty} \int_0^{\infty} K^x(y) dy \dot{H}^{\xi}(z) dz \\[0.2cm]
& = 
\int_0^{\infty}  \dot{H}^{\xi}(z) dz \int_0^{\infty} K^x(y)dy. 
\end{align*} 
By \eqref{int-compu1}, we have:
\begin{align} 
\int_0^{\infty} K^x(y)dy =\int_x^{\infty} (\mathcal{B}\ast \mathcal{V}) (I(y)) \alpha dy =  (\mathcal{B}\ast \mathcal{V}^e)(I(x)) M \lambda <\infty ~~\text{ for all } x \geq 0. \label{K^x_int}
\end{align}
Therefore, \eqref{int:H_xi} with \eqref{K^x_int} imply that
\begin{align} 
\int_0^{\infty} (K^x \ast H^{\xi})(y) \, dy =  \bar{W}(0) \, (\mathcal{B}\ast \mathcal{V}^e)(I(x)) M \lambda <\infty. \label{equiv2:K_H}
\end{align}
Let us denote the functions 
\begin{align*}
f^x_{\xi}(y):= (K^x \ast H^{\xi})(y) ~~ \text{ and}  ~~ g^x_{\xi}(y):= K^x (y/2) \bar{W}(0) + \dot{H}^{\xi}(y/2) QM\lambda.
\end{align*}
By \eqref{equiv1:K_H} and \eqref{equiv2:K_H}, proving inequality \eqref{Kx_H_U} is equivalent to showing that:
\begin{align}
\lim_{t \rightarrow \infty} \, \sup_{|f^x_{\xi}| \leq g^x_{\xi}} \left| (f^x_{\xi} \ast U) (t) - \dfrac{1}{\langle \chi, (\alpha.\zeta)^e \rangle} \int_0^{\infty} f^x_{\xi}(s)ds \right| = 0 ~~~ \text{ for all } \xi \in \mathsf{B}^p_N ~ \text{ and } x \geq 0.
\end{align}
To prove this convergence, we will utilize Theorem 6.12 from \cite{lindvall2002lectures}. Based on this theorem, it suffices to show that for all $x\geq 0$ and $\xi \in \mathsf{B}^p_N$: (i) $g^x_{\xi}$ is bounded, (ii) $g^x_{\xi}$ is integrable, and (iii) $g^x_{\xi}(y)$ goes to $0$ as $y$ goes to infinity.
First, since $g^x_{\xi}$ is nonincreasing, we have for all $y \geq 0$:
\begin{align*}
g^x_{\xi}(y) & \leq g^x_{\xi}(0)=
(\mathcal{B} \ast \mathcal{V})(I(x)) \, \alpha \bar{W}(0) + e(I-P') \langle 1, \mathcal{B} \ast \xi \rangle QM \lambda \\
& \leq \lambda \bar{W}(0) + e(I-P') \langle 1, \mathcal{B} \ast \xi \rangle QM \lambda.
\end{align*}
The right-hand side of the preceding inequality is finite. Therefore, $g^x_{\xi}$ is bounded.
Secondly, we have for all  $\xi \in \mathsf{B}^p_N$ and  $x \geq 0$,
\begin{align*}
\int_0^{\infty} g^x_{\xi} (y) dy =2 \left( \int_0^{\infty} K^x(y) dy \, \bar{W}(0) +  \int_0^{\infty} \dot{H}^{\xi}(y)dy \, QM \lambda \right)<\infty.
\end{align*}
Therefore, $g^x_{\xi}$ is integrable for all $x \geq 0$
and $\xi \in \mathsf{B}^p_N$.
Finally, it is obvious that $g^x_{\xi}(y)$ goes to $0$ as $y$ goes to infinity.
\end{proof}


\subsection{Proof of Proposition \ref{asmp}}~ \label{pr_asmp}

The fluid limit equation \eqref{fluid-mathcal-Q} of the process $\mathcal{Q}(t)$ can be written as 
\begin{equation} \label{eq:bar-gamma-lk-conv}
\bar{\mathcal{Q}}(t) \left( I(x) \right) = \mathcal{B}\ast\mathcal{V}^0 (I(x+\bar{S}(t))) \bar{Z}(0) +(K^x\ast
H^{\xi} \ast U)(\bar{S}(t)).
\end{equation}
To see this, we start from \eqref{fluid-mathcal-Q}:
\begin{align*}
 \bar{\mathcal{Q}}(t)(I(x)) = \mathcal{B}\ast\mathcal{V}^0 (I(x+\bar{S}(t))) \bar{Z}(0)+\int_{0}^{t} 
K^x(\bar{S}(t)-\bar{S}(s) )ds.
\end{align*}
Applying the change of variable $y=\bar{S}(s)$ and recalling that $T(\cdot) = \bar{S}^{-1}(\cdot)$, we obtain
\begin{align*}
\bar{\mathcal{Q}}(t)(I(x)) = \mathcal{B}\ast\mathcal{V}^0 (I(x+\bar{S}(t))) \bar{Z}(0)+ 
(K^x \ast T)(\bar{S}(t)).
\end{align*}
Thus, using \eqref{eq:inv-cum-service}, we arrive at \eqref{eq:bar-gamma-lk-conv}.
Since $\lim_{ t \rightarrow \infty} \bar{S} (t) = + \infty $, Lemma \ref{ver-unif-cond-lm2} implies
\begin{align*} \
\sup_{\xi \in \mathsf{B}^p_N} \sup_{x \geq 0}   \left| \big(\left(K^x\ast H^{\xi} \right) \ast U \big) \left(\bar{S} (t)\right) - \dfrac{\bar{W}(0)}{\langle \chi, (\alpha.\zeta)^e \rangle}  \, \mathcal{B}\ast \mathcal{V}^e (I(x)) M \lambda \right| \longrightarrow 0 ~~ \text{ as } t \rightarrow \infty.
\end{align*}
Moreover,
$$
\sup_{\xi \in \mathsf{B}^p_N} \sup_{x \geq 0} \left|  \mathcal{B}\ast\mathcal{V}^0 (I(x+\bar{S}(t))) \bar{Z}(0) \right| \longrightarrow 0 ~~ \text{ as } t \rightarrow \infty.
$$
Therefore, by \eqref{wss3} and \eqref{eq:bar-gamma-lk-conv}, there exists $t_b \geq 0$ such that for all $t \geq T(t_b)$ and for all $\varepsilon >0$,
  \begin{align} 
\sup_{\xi \in \mathsf{B}^p_N} \sup_{x \geq 0}  \left| \bar{\mathcal{Q}}^{\,\xi}(t)(I(x))  - \mathcal{B} \ast \Delta^{\nu} (I(x))  \bar{W}(0) \right| \leq \varepsilon.
\end{align}
Finally, using \cite[Lemma C.1]{zhang2011diffusion}, we conclude that
\begin{align*}
\sup_{\xi \in \mathsf{B}^p_N} \mathbf{d} \left(  \bar{\mathcal{Q}}^{\,\xi}(t),\mathcal{B} \ast \Delta^{\nu} \bar{W}(0) \right) \longrightarrow 0 ~~ \text{ as } t \rightarrow \infty.
\end{align*} 

\subsection{Proof of Theorem \ref{pro:unif-conv-mu}} ~ \label{pr_convmu}

Let
$
\tilde{A}(t)= \bar{A}(T(t)), ~  \tilde{D}(t)= \bar{D}(T(t)), ~ \tilde{Z}(t)= \bar{Z}(T(t)).
$
Performing the change of variables and functions in \eqref{eq:bar-arrivee-depart}-\eqref{eq:evol-barmu} gives the new functional equations:
\begin{align*}
    \tilde{A}(t) &= \alpha T(t)+P^{\prime }\tilde{D}(t)
    \\[0.2cm]
    \tilde{Z}(t) &=\bar{Z}(0) +\tilde{A}
    (t)-\tilde{D}(t) \\[0.2cm]
    \bar{\mu}^{\xi}(T(t))(I(x)) &=
    (I-B^0)(x+t)) \bar{Z}(0)
    +\int_{0}^{t}  \left(I-B(x+t-s) \right)  \, d\tilde{A}(s).
  \end{align*}   
The two first above equations imply 
\begin{align*}
\tilde{A}(t)= \lambda T(t)+ QP'( \bar{Z}(0) - \tilde{Z}(t)).
\end{align*}
Making use of this equation and the fact that $\dot{T}(t)= e \cdot \tilde{Z}(t)$ and integrating by parts,
one deduces:
\begin{align}
\bar{\mu}^{\,\xi}(T(t))(I(x)) = C(t+x) \bar{Z}(0) - (I-B(x))P'Q \tilde{Z}(t) +\left( G^x \ast  Q\tilde{Z} \right)  (t),  \label{barmu-eq-I-0}
\end{align}
where $C(t)$ has been defined in \eqref{K_C:def} and $G^x (t)$ is the matrix defined as
\begin{align}
G^x (t) & := \int_x^{x+t} (I-B(u)) du \lambda e (I-P') + (B(x+t)-B(x))P'. \label{Gx_def}
\end{align}
By \eqref{glob_mathcal_Q} and  Proposition \ref{asmp},
\begin{align*}
\sup_{\xi \in \mathsf{B}^p_N} \left|  Q\tilde{Z} (t)   -  \dfrac{\bar{W}(0)}{e( \frac{1}{2}  M^{(2)}+MP^{'} QM)  \lambda}  \, Q M \lambda  \right| \longrightarrow 0 ~~ \text{ as } t \rightarrow \infty. 
\end{align*}
Thus, 
\begin{align*}
\sup_{\xi \in \mathsf{B}^p_N} \left|  \left( G^x \ast  Q \tilde{Z} \right)  (t)  -   G^x (\infty) \,  \dfrac{\bar{W}(0)}{e( \frac{1}{2}  M^{(2)}+MP^{'} QM)  \lambda}  \, Q M \lambda \right| \longrightarrow 0 ~~ \text{ as } t \rightarrow \infty,
\end{align*}
with
\begin{align*}
G^x(\infty) = M \Lambda \nu^e (I(x)) \, e (I-P') + (I-B(x))P'.
\end{align*}
One obtains for all $x \geq 0$,
\begin{align}
\lim_{t \rightarrow \infty} \, \sup_{\xi \in \mathsf{B}^p_N} \left|  \left( G^x \ast  Q \tilde{Z} \right)  (t)   -  \dfrac{\bar{W}(0)}{e( \frac{1}{2}  M^{(2)}+MP^{'} QM)  \lambda}  \, M \Lambda \nu^e (I(x)) -  (I-B(x))P'Q \tilde{Z}(\infty)  \right| =0. \label{hx-conv-ug}
\end{align}
On the other hand, for all $x \geq 0$ we have
\begin{align*}
\sup_{\xi \in \mathsf{B}^p_N} \, \sup_{x \geq 0} \, \left| C(t+x) \bar{Z}(0) \right| \longrightarrow 0 ~~ \text{ as } t \rightarrow \infty.
\end{align*}
Then by \eqref{barmu-eq-I-0} and \eqref{hx-conv-ug}, there exists $t_a \geq 0$ such that for all $t \geq T(t_a)$ and for all $\varepsilon >0$
 \begin{align*}
 \sup_{\xi \in \mathsf{B}^p_N} \, \sup_{x \geq 0} \, \left|  \bar{\mu}^{\,\xi}(t)(I(x))  -  \Delta^{\nu} (I(x)) \bar{W}(0) \right| \leq \varepsilon.
 \end{align*}
Thus, 
\begin{align*}
\sup_{\xi \in \mathsf{B}^p_N} \mathbf{d} \left(  \bar{\mu}^{\,\xi}(t),\Delta^{\nu} \bar{W}(0) \right) \longrightarrow 0 ~~ \text{ as } t \rightarrow \infty.
\end{align*}

\section{Proof of Theorem \ref{ssc2}}
\label{Sec5Q}

To prove Theorem \ref{ssc2}, we use the framework of the shifted fluid scaled process introduced by \cite{bramson1998state} and \cite{williams1998diffusion}.
Define $\bar{\mathcal{Q}}^r (t) = \mathcal{Q}^r (r t) /r$. Since $\hat{\mathcal{Q}}^r(t) = \bar{\mathcal{Q}}^r(rt)$, then studying the diffusion limit on $[0,T]$ for a fixed $T>0$ is equivalent to study the fluid process on $[0,rT]$.
We cover the interval $[0,rT]$ by a set of overlapping intervals $[m,m+L]$ where $m=0, \cdots,\lfloor rT \rfloor$ and $L>1$.
For any $t \in [0,T]$, there exists an $m\in \{0,\ldots,\lfloor rT\rfloor\}$ and an $s\in [0,L]$ such that $r^2 t = r(m+s)$ and
\begin{equation}
\label{eq:ssh-fluid-appro}
\hat{\mathcal{Q}}^r(t) = \bar{\mathcal{Q}}^{r,m}(s),
\end{equation} 
where $\bar{\mathcal{Q}}^{r,m}(\cdot):=\bar{\mathcal{Q}}^{r} (m+\cdot)$ represent the shifted fluid scaled version of the process $\mathcal{Q}^r(\cdot)$. 
The proof of Theorem \ref{ssc2} entails two main steps.
The first step is to establish the precompactness of the sequence of shifted fluid-scaled processes $\{\bar{\mathcal{Q}}^{r,m}(\cdot);\, r>0, m=0,\ldots,\lfloor rT\rfloor \}$ (cf. Sect \ref{sec:precomp}).
The second step involves demonstrating that, for large $r$, there exists a set with high probability such that the family of shifted fluid-scaled processes $\{ \bar{\mathcal{Q}}^{r,m}(\cdot), m \leq rT \}$ evaluated at some sample path in this set, are uniformly approximated on $[0,L]$ by fluid model solutions to the equation \eqref{fluid-mathcal-Q} (cf. Sect \ref{sec:ssc}). 
In addition to the two steps mentioned above, a result on the uniform convergence of the fluid solution to the invariant state is required, as established in the previous section (see Proposition \ref{asmp}).

\subsection{Precompactness of the family of the shifted scaled versions of the process $\mathcal{Q}^r(\cdot)$}
\label{sec:precomp}

The aim of this section is to prove Theorem \ref{preco}. Following \cite[Theorem 3.6.3]{ethier1986markov}, we need to establish the compact containment result (Lemma \ref{cocn}) and the oscillation bound result (Lemma \ref{Os lm}).
In the first section, we present dynamic equations satisfied by the shifted fluid-scaled processes $\{\bar{\mathcal{Q}}^{r,m}(\cdot), r>0, m \leq \lfloor rT \rfloor  \}$.
In Section \ref{presub3}, we establish a uniform functional weak law of large numbers. This is crucial for the subsequent proofs and will be frequently employed alongside the dynamic equations.
Section \ref{presub4} contains two upper bound estimates, which lead to the compact containment result in Section \ref{presub5}.
Section \ref{presub6} presents the asymptotic regularity result, which is utilized in Section \ref{presub7} to prove the oscillation bound result. Finally, we demonstrate the precompactness result in Section \ref{presub8}.
\subsubsection{Dynamic equations} \label{presub2}

Fix $r,T>0, L>1$ and $m \leq \lfloor rT \rfloor$. For $t' \in [0,L]$, we
denote by $
 \bar{\gamma}^{r,m}(t')=\bar{\gamma}^r(m+t')$ and $
\bar{\mathcal{Q}}^{r,m}(t')=\bar{\mathcal{Q}}^{r}(m+t') 
$ the shifted fluid scaled processes of  $ \gamma^r(\cdot) $ and $\mathcal{Q}^r(\cdot)$.
 Let $g:\mathbb{R}_+\rightarrow\mathbb{R}$ be a Borel-measurable function. The dynamic equations \eqref{eq:mathcal-Q} can be written in a shifted fluid scaled version as
\begin{eqnarray}
 & &  \langle g,\bar{\mathcal{Q}}^{r,m}(t'+h)\rangle  =
 \langle (g \, 1_{(0,\infty)})(\cdot - \bar{S}^{r,m}(t',t'+h)),\bar{\mathcal{Q}}^{r,m}(t') \rangle 
 \nonumber \\
  & & +\frac{1}{r} \, \sum_{l=1}^K\sum_{i=r\bar{E}_{l}^{r,m}(t')+1}^{r\bar{E}
    _{l}^{r,m}(t'+h)} \langle g \,1_{(0,\infty)} \left( \cdot -  \bar{S}^{r,m}(U_{l}^{r}(i)r^{-1}-m,t'+h) \right),  \vartheta^{r,l}(i) \rangle.~~~~~~~~~
   \label{eq:shifted-dynamic-equation2}  
\end{eqnarray}


\subsubsection{Uniform Functional weak law of large numbers} \label{presub3}
Let $\theta$ be the constant given in conditions \eqref{eq:assum3},\eqref{eq:assum6} and \eqref{eq:assum8}; and let $q>0$ be a constant such that 
\begin{equation}
\label{const-q}
2q^2+6q<\theta.
\end{equation}
The following set of functions will be used below
\begin{align} 
\mathcal{A}=\lbrace 1_{(x,\infty)} :  x \in \mathbb{R}_+  \rbrace \cup \lbrace \chi^{1+q},\chi^{2+2q} \rbrace . 
\label{class}
\end{align}

\begin{proposition}
\label{GLVLM}
Assume \eqref{eq:assum1}-\eqref{eq:assum10} and \eqref{eq:assum-init1}-\eqref{eq:assum-init4}. Fix $T>0,~L>1$, and let $\varepsilon, \eta > 0$. Then,
\begin{equation}
\liminf_{r \rightarrow \infty} \mathbb{P}^r \Bigg(\sup_{\scriptstyle{\substack{ 0\leq s\leq t\leq L+1 \\  m \leq \lfloor rT \rfloor \\ g \in \mathcal{A}} }} \Bigg\rVert\frac{1}{r}  \sum_{l=1}^K\sum_{i=r\bar{E}^{r,m}_{l}(s) + 1}^{r \bar{E}^{r,m}_{l}(t)} \langle g,\vartheta^{r,l}(i) \rangle
   -(t-s)\langle g, \mathcal{B}^r \ast \mathcal{V}^r \rangle \alpha \Bigg \lVert \leq \varepsilon  \Bigg ) \geq 1-\eta   
\label{glika}
\end{equation}
Denote $\Omega_{LN}^r$ the event in the above.
\end{proposition} 
\par The following Lemmas are needed for the proof of Proposition \ref{GLVLM}. The proof of Lemma \ref{exell} is omitted since it follows immediatly  from \eqref{eq:cons-assum1} and the definition of shifted scaled process $\bar{E}^{r,m} (\cdot)$. 
\begin{lemma} 
\label{nurgs}
Assume conditions \eqref{eq:assum5}, \eqref{eq:assum7} and \eqref{eq:assum9}. For each $l,k \in \mathcal{K}$, we have 
\begin{equation}
  \lim_{N \rightarrow \infty } \, \sup_{r \in \mathbb{R}^+} \; \sup_{g\in \mathcal{A}} \; \langle g^2 \, 1_{\lbrace g >N \rbrace}, \mathcal{B}_{kl}^r \ast \nu_l^r \rangle = 0.  
\label{envcond}
\end{equation}
\end{lemma}
\begin{proof} 
We have 
$
\langle g^2 1_{\lbrace g>N  \rbrace}, \mathcal{B}_{kl}^r \ast \nu_l^r \rangle   \leq \frac{1}{N^q} \langle g^{2+q},\mathcal{B}_{kl}^r \ast \nu_l^r \rangle.
$
Thus, to prove \eqref{envcond} it suffices to prove
\begin{equation}
\label{eq:limsup-BB}
\limsup_{r \rightarrow \infty} \; \sup_{g\in \mathcal{A}} \;  \langle g^{2+q}, \mathcal{B}^r_{kl} \ast \nu^r_l \rangle < \infty.
\end{equation}
By \eqref{eq:b-ast-nu0} and assumptions \eqref{eq:assum5} and \eqref{eq:assum7}, we have 
\begin{equation}
\langle 1,\mathcal{B}^r_{kl} \ast \nu^r_l \rangle= Q^r  \longrightarrow \langle 1,\mathcal{B}_{kl} \ast \nu_l \rangle = Q.  
\label{B*v:conv}
\end{equation} 
Therefore,
\begin{align*}
\limsup_{r \rightarrow \infty} \; \sup_{x\geq 0}  \; \langle 1^{2+q}_{(x,\infty)}, \mathcal{B}^r_{kl} \ast \nu^r_l \rangle \leq \limsup_{r \rightarrow \infty} \;  \langle 1, \mathcal{B}^r_{kl} \ast \nu^r_l \rangle < \infty. 
\label{gone}
\end{align*}
By assumption \eqref{eq:assum9} and Lemma \ref{lem:conv-zeta-r}, we obtain 
\begin{equation}
\limsup_{r \rightarrow \infty} \langle \chi^{4+\theta} , \mathcal{B}^r_{kl} \ast \nu^r_l  \rangle <\infty.   
\label{mom:4+theta}
\end{equation}   
Therefore, by definition of the constant $q$ in \eqref{const-q}, we have
\begin{equation*}
 \label{gtwo}
 \limsup_{r \rightarrow \infty} \langle \left(\chi^{1+q} \right)^{2+q} , \mathcal{B}^r_{kl} \ast \nu^r_l  \rangle <\infty ~~~\mbox{and}~~~\limsup_{r \rightarrow \infty}\langle \left(\chi^{2+2q}\right)^{2+q} , \mathcal{B}^r_{kl} \ast \nu^r_l \rangle <\infty.
\end{equation*}
\end{proof}
\begin{lemma} 
\label{exell}
Assume \eqref{eq:assum2}-\eqref{eq:assum4}. Fix $T>0,~ L >1$. For all $\varepsilon >0$, we have
\begin{equation*}
\label{vcv}
\lim_{r \rightarrow \infty} \mathbb{P}^r \Bigg( \sup_{\scriptstyle{ \substack{ 0\leq s\leq t\leq L+1 \\  m \leq \lfloor rT \rfloor} }} ~  \left \lVert (\bar{E}^{r,m} (t) - \bar{E}^{r,m}(s) ) - (t-s) \alpha \right \rVert \leq \varepsilon \Bigg) =1. 
\end{equation*} 
\end{lemma}
Denote by $\Omega^r_E$ the event in the above set.
\begin{proof}[Proof of Proposition \ref{GLVLM}]
Following a basic arithmetic, we get
\begin{multline}
 \sup_{\scriptstyle{\substack{ 0\leq s\leq t\leq L+1 \\  m \leq \lfloor rT \rfloor \\ g \in \mathcal{A}} }}\Bigg\rVert\frac{1}{r} \sum_{l=1}^K\sum_{i=r\bar{E}^{r,m}_{l}(s) + 1}^{r \bar{E}^{r,m}_{l}(t)} \langle g,\vartheta^{r,l}(i) \rangle-(t-s)\langle g, \mathcal{B}^r \ast \mathcal{V}^r \rangle \alpha \Bigg \lVert 
\\
 \leq \sup_{\scriptstyle{\substack{ 0\leq s\leq t\leq L+1\\  m \leq \lfloor rT \rfloor \\ g \in \mathcal{A}} }}\Bigg\rVert\frac{1}{r} \sum_{l=1}^K\sum_{i=r\bar{E}^{r,m}_{l}(s) + 1}^{r \bar{E}^{r,m}_{l}(t)} \left(\langle g,\vartheta^{r,l}(i) \rangle-\langle g, \mathcal{B}^r_{\bullet l} \ast \mathcal{V}^r_l \rangle \right)\Bigg\lVert
 \nonumber\\
+\sup_{ g \in \mathcal{A}}\rVert \langle g, \mathcal{B}^r \ast \mathcal{V}^r \rangle\rVert 
\sup_{\scriptstyle{\substack{0\leq s\leq t\leq L+1\\  m \leq \lfloor rT \rfloor } }}\rVert \bar{E}^{r,m}(t)-\bar{E}^{r,m}(s)-(t-s)\alpha\lVert
 \end{multline}
Label the two terms on the right-hand side of the preceding inequality as $I^r$ and $II^r$. Let $\varepsilon>0$ and $\eta>0$, to prove \eqref{glika}, it suffices to prove that 
\begin{equation}
\label{eq:I-II}
 \liminf_{r\to\infty}\mathbb{P}^r(I^r<\varepsilon/2)\geq 1-\eta/2 ~~\text{ and }~~\liminf_{r\to\infty}\mathbb{P}^r(II^r<\varepsilon/2)\geq 1-\eta/2.
\end{equation}
Equation \eqref{eq:limsup-BB} implies that there exists a constant $C_1>0$ such that 
\begin{equation}
 \sup_{ g \in \mathcal{A}}\rVert \langle g, \mathcal{B}^r \ast \mathcal{V}^r \rangle\rVert \leq C_1~~\text{ for large } r.
\end{equation}
Given that
\begin{equation*}
 \mathbb{P}^r(II^r<\varepsilon/2)\leq \mathbb{P}^r\Bigg(\sup_{\scriptstyle{\substack{ 0\leq s\leq t\leq L+1\\  m \leq \lfloor rT \rfloor } }}\rVert \bar{E}^{r,m}(t)-\bar{E}^{r,m}(s)-(t-s)\alpha\lVert\geq \varepsilon/2C_1\Bigg),
\end{equation*}
and Lemma \ref{exell} implies that the term on the right-hand side of the preceding expression exceeds $1-\eta/2$. Therefore, the second inequality in \eqref{eq:I-II} holds. After a straightforward computation, the set $\left\{I^r\leq \varepsilon/2 \right \}$ contains
\begin{equation}
 \bigcap_{l,k\in\mathcal{K}}\Bigg\{\sup_{\scriptstyle{ \substack{ 0\leq s\leq t\leq L+1\\  m \leq \lfloor rT \rfloor \\ g \in \mathcal{A}} }} \Bigg \lvert \dfrac{1}{r}\sum_{i=r \bar{E}^{r,m}_l (s) + 1}^{r \bar{E}^{r,m}_l (t)} \left( \langle g,\vartheta_{lk}^r(i) \rangle-\langle g,\mathcal{B}_{kl}^r \ast \nu_l^r \rangle \right) \Bigg \rvert \leq \varepsilon/2K \Bigg\}
\end{equation}
Thus, to show the first inequality of \eqref{eq:I-II} it suffices to show that
for each $l,k \in \mathcal{K}$
\begin{equation}
\liminf_{r \rightarrow \infty} \mathbb{P}^r \Bigg(\sup_{\scriptstyle{ \substack{ 0\leq s\leq t\leq L+1\\  m \leq \lfloor rT \rfloor \\ g \in \mathcal{A}} }} \Bigg \lvert \dfrac{1}{r}\sum_{i=r \bar{E}^{r,m}_l (s) + 1}^{r \bar{E}^{r,m}_l (t)} \left( \langle g,\vartheta_{lk}^r(i) \rangle-\langle g,\mathcal{B}_{kl}^r \ast \nu_l^r \rangle \right) \Bigg \rvert \leq \frac{\varepsilon}{2K} \Bigg) >1-\frac{\eta}{2K^2}. \label{WLLN2} 
\end{equation}
Fix $M_1= 2 \| \alpha \| (T+1) $ and $L_1= 2  (L + 1) \| \alpha \|$.
For any $s\leq t$, denote $l^{\prime}= \bar{E}^{r,m}_l (t) - \bar{E}^{r,m}_l (s)$ and  $\ell= r \bar{E}^{r,m}_l (s)$. According to Lemma \ref{exell}, the conditions  $0\leq s\leq t\leq L+1$ and $m \leq \lfloor  rT \rfloor$ imply $0 <\ell< r^2 M_1$  and $0\leq l^{\prime}\leq L_1$ for sufficiently large $r$. 
Hence, demonstrating \eqref{WLLN2} requires proving
\begin{equation}
\liminf_{r \rightarrow \infty}\mathbb{P}^r \Bigg(\sup_{\scriptstyle{ \substack{ 0 <\ell< r^2 M_1  \\ 0\leq l^{\prime}\leq L_1 \\ g \in \mathcal{A}} }}
 \left \lvert \dfrac{1}{r} \sum_{i=\ell+1}^{\ell+ \lfloor rl^{\prime} \rfloor }\langle g, \vartheta_{lk}^r(i)\rangle - l^{\prime} \langle g,\mathcal{B}_{kl}^r \ast \nu_l^r \rangle \right \rvert \leq \frac{\varepsilon}{2K} \Bigg)>1-\frac{\eta}{2K^2}.
\label{zhgg}
\end{equation}
For each $n \geq 1$, denote by $\nu_{lk}^r(n)$ the law of $V^r_{lk}(1,n)$ conditioning on the event $\lbrace N_{k}^{r,l} (1) \geq n \rbrace$. Note that $\widetilde{\nu}_{kl}^r(m)=\nu_{kl}^r(m)\mathbb{P}(N_{k}^{r,l}(1)\geq m)$ is the law of $\widetilde{V}^r_{lk}(i,m)=V_{lk}^r(i,m)1_{\{N_{k}^{r,l}(i)\geq m\}}$.
By \eqref{es:mathca-X} and (ii) Lemma \ref{lem:pro-N-V-lk} we have for all $i \geq 1$, 
 \begin{equation}
  \langle g, \vartheta_{lk}^r(i)\rangle = \sum_{m=1}^{\infty}g ( \widetilde{V}_{lk}^r(i,m))~~\mbox{and} ~~
  \langle g,\mathcal{B}_{kl}^r \ast \nu_l^r \rangle=\sum_{m=1}^{\infty}\langle g, \widetilde{\nu}_{lk}^r(m)\rangle 
  \label{alter-X-g}
 \end{equation}
We replace $\langle g, \vartheta_{lk}^r(i)\rangle$ and $\langle g,\mathcal{B}_{kl}^r \ast \nu_l^r \rangle$ by their expressions of \eqref{alter-X-g} in \eqref{zhgg}. Then, it suffices to prove 
\begin{equation}
\label{funlw}
\limsup_{r \rightarrow \infty} \, \mathbb{P}^r \Bigg( \sup_{\scriptstyle{ \substack{ 0 <\ell< r^2 M_1  \\ 0\leq l^{\prime}\leq L_1 \\ g \in \mathcal{A}} }} \left \lvert \; \sum_{m=1}^{\infty} \, \Big(\dfrac{1}{r} \sum_{i=\ell+1}^{\ell+ \lfloor rl^{\prime} \rfloor } g \,( \widetilde{V}^r_{lk}(i,m)) - l^{\prime}\langle g,\widetilde{\nu}_{kl}^r(m) \rangle \Big) \right \rvert > \frac{\varepsilon}{2K}  \Bigg) < \frac{\eta}{2K^2}.
\end{equation}
Let $l,k\in\mathcal{K}$ and $m \geq 1$ be fixed, and denote $\bar{g}:= \sup \{g: g\in \mathcal{A}\}$. By Lemma \ref{nurgs}, we have 
$$\lim_{N \rightarrow \infty }\sup_{r \in \mathbb{R}^+}\langle \bar{g}^2 \, 1_{\lbrace g >N \rbrace}, \tilde{\nu}_{kl}^r(m) \rangle = 0.$$
 Consequently, the sequence $(\widetilde{V}^r_{lk}(i,m),i\geq 1)$ satisfies the conditions of the Glivenko-Cantelli estimate  in Lemma D.1 \cite{zhang2011diffusion}. Hence,
\begin{equation}
\limsup_{r \rightarrow \infty} \, \mathbb{P}^r \Bigg(  \sup_{\scriptstyle{ \substack{ 0 <\ell< r^2 M_1  \\ 0\leq l^{\prime}\leq L_1 \\ g \in \mathcal{A}} }} \left \lvert \dfrac{1}{r} \sum_{i=\ell+1}^{\ell+ \lfloor rl^{\prime} \rfloor } g \,( \widetilde{V}^r_{lk}(i,m)) - l^{\prime}\langle g,\widetilde{\nu}_{kl}^r(m) \rangle \right \rvert > \frac{\varepsilon}{2^{m+1}K} \Bigg) < \frac{\eta}{2^{m+1}K^2} 
\label{lemD1-zheng}
\end{equation}
Given that
\begin{multline*}
\mathbb{P}^r \Bigg(  \sup_{\scriptstyle{ \substack{ 0 <\ell< r^2 M_1  \\ 0\leq l^{\prime}\leq L_1 \\ g \in \mathcal{A}} }} \; \left \lvert \; \sum_{m=1}^{\infty} \, \Big(\dfrac{1}{r} \sum_{i=\ell+1}^{\ell+ \lfloor rl^{\prime} \rfloor } g \,( \widetilde{V}^r_{lk}(i,m)) - l^{\prime}\langle g,\widetilde{\nu}_{kl}^r(m) \rangle \Big) \right \rvert > \frac{\varepsilon}{2K} \Bigg) \\
\leq \sum_{m=1}^{\infty}\mathbb{P}^r \Bigg( \sup_{\scriptstyle{ \substack{ 0 <\ell< r^2 M_1  \\ 0\leq l^{\prime}\leq L_1 \\ g \in \mathcal{A}} }} \left \lvert \dfrac{1}{r} \sum_{i=\ell+1}^{\ell+ \lfloor rl^{\prime} \rfloor } g \,( \widetilde{V}_{lk}^r(i,m) ) - l^{\prime}\langle g,\widetilde{\nu}_{kl}^r(m) \rangle \right \rvert > \frac{\varepsilon}{2^{m+1}K} \Bigg),
\end{multline*}
as $r\rightarrow\infty$. Then, it follows from \eqref{lemD1-zheng} that the limit inferior of the term on the right-hand side of the above equality is bounded above by $\eta/2K^2$. Therefore, \eqref{funlw} holds.
\end{proof}

\subsubsection{Preliminary estimates} \label{presub4}~

In this section, we offer estimates necessary for assessing the tightness of $\bar{\mathcal{Q}}^{r,m}(\cdot)$. The subsequent two lemmas provide an upper bound for both the moment and the total mass of the shifted fluid scaled process $\bar{\mathcal{Q}}^{r,m} (\cdot)$.
\begin{lemma}
\label{tilde-M-T}
Assume \eqref{eq:assum1}-\eqref{eq:assum10} and \eqref{eq:assum-init1}-\eqref{eq:assum-init4}. Fix $T>0$ and $L>1$, and let $q$ be the constant defined in \eqref{const-q}. For each $\eta > 0$, there  
exists a constant $\widetilde{M}_{T,L}>0$ such that 
\begin{align}
\liminf_{r \rightarrow \infty} \mathbb{P}^r \left( \max_{m \leq \lfloor rT \rfloor}  \, \left\| \,
 \langle  \chi^{1+q}, \bar{\mathcal{Q}}^{r,m} (\cdot) \rangle \, \right\|_L \leq \widetilde{M}_{T,L} \right) \geq 1 - \eta 
 \label{1+q}
\end{align} 
\end{lemma}

\begin{proof}
Let $t \in [0,L]$. By bounding the dynamic equation  \eqref{eq:shifted-dynamic-equation2} with $t'=0, h=t$ and $g=\chi^{1+q}$, we have 
\begin{align}
\left \lvert \langle \chi^{1+q}, \bar{\mathcal{Q}}^{r,m} (t) \rangle \right \rvert  \leq \left \lvert \langle \chi^{1+q},\bar{\mathcal{Q}}^{r} (0) \rangle \right \rvert 
 +  \left \lvert \dfrac{1}{r} \, \sum_{l=1}^K \, \sum_{i=1}^{r\bar{E}^{r}_l(m+t)} \langle \chi^{1+q} \left( \cdot -  \bar{S}^r (U^r_l (i)/r ,m+t ) \right),  \vartheta^{r,l}(i) \rangle \right \rvert.  \label{eq:dyn-moment}
\end{align}
By \eqref{bound:Bv0} and definition of the constant $q$ in \eqref{const-q}, we obtain ,
$$ \langle \chi^{1+q}, \mathcal{B}^r \ast \mathcal{V}^{0,r} \rangle \longrightarrow \langle \chi^{1+q}, \mathcal{B} \ast \mathcal{V}^0 \rangle < \infty.$$ 
Furthermore, by definition of $ \mathcal{Q}^r(0)$ in \eqref{eq:mathcal-Q-0}, and by applying Lemma D.5 from \cite{tahar2012fluid} with the replacements $V^r_{lk}(i,n)$ by $V^{0,r}_{lk}(i,n)$ and $\bar{E}^r(\cdot)$ by $\bar{Z}(0)$, we have
 \begin{equation}
 \langle \chi^{1+q}, \bar{\mathcal{Q}}^r(0) \rangle \Rightarrow \langle \chi^{1+q}, \mathcal{B} \ast \mathcal{V}^0  \rangle \bar{Z}(0)< \infty.
 \end{equation}
 Therefore, there exists a constant $M_{0,q}$, such that
\begin{equation}
\label{M-0-q}
\liminf_{r \to \infty} \mathbb{P}^r\left ( \lvert \langle \chi^{1+q} ,\bar{\mathcal{Q}}^{r} (0) \rangle  \rvert < M_{0,q} \right)\geq 1-\eta.
\end{equation}
Let $\Omega^r_{0,q}$ denote the event described in the preceding set. Under this condition, the first term on the right-hand side of \eqref{eq:dyn-moment} is bounded by $M_{0,q}$. To bound the second term on the right-hand side of \eqref{eq:dyn-moment}, we first examine the event $\Omega^r_{LN}$ defined in \eqref{glika}. Subsequently, we establish an event wherein the shifted fluid cumulative service process $\bar{S}^{r,m}(\cdot)$ is lower bounded. Indeed, by Proposition \ref{prop:gamma-global} and the shifted scaling property, there exists a constant $M_{\gamma}>0$ such that
\begin{equation}
\label{const:M1}
 \liminf_{r\to\infty}\mathbb{P}^r \left( \underset{ m \leq \lfloor rT \rfloor}{\max}  \, \left\| \, \langle 1, \bar{\gamma}^{r,m} (\cdot) \rangle \, \right\|_L \leq M_{\gamma} \right)\geq 1-\eta.
\end{equation}
Denote the set under the aforementioned limit by $\Omega^r_{\gamma}$. Referring to \eqref{glika}, \eqref{M-0-q}, and \eqref{const:M1}, we obtain
\begin{equation} \label{Mom-q-ev}
\liminf_{r\to\infty} \mathbb{P}^r( \Omega^r_{M,q} ) \geq 1-\eta, 
\end{equation}
where $\Omega^r_{M,q} = \Omega^r_{\gamma} \cap \Omega^r_{0,q} \cap \Omega^r_{LN}$.
Let $\omega$ be a fixed sample path in  $\Omega^r_{M,q} $. Throughout the remainder of the proof, all random variables are evaluated at this
$\omega$. 
To bound the second term on the right hand side of \eqref{eq:dyn-moment}, an issue arises. We cannot utilize the fact that we are on the event $\Omega^r_{LN}$ due to the summation ranging from $i=1$ to $i=r\bar{E}^{r}(m+t)$. It is worth noting that this can be expressed as the sum from $i=1+r\bar{E}(0)$ to $i=1+r\bar{E}^r(m+t)$. However, the problem lies in the upper bound of $m$, which is constrained by $\lfloor rT \rfloor$, and this value tends to infinity as $r$ approaches infinity.
To address this, we consider two cases:  $m \in \lbrace 0,1 \rbrace$ and $m \geq 2$. For the first case, the second term on the right hand side of \eqref{eq:dyn-moment} is bounded by 
\begin{align*}
\left \lvert \dfrac{1}{r} \, \sum_{l=1}^K \, \sum_{i=1+r\bar{E}^{r}_l(0)}^{r\bar{E}^{r}_l(m+t)} \langle \chi^{1+q} , \vartheta^{r,l} \rangle \right \rvert \leq 2 (L+1) \,  \left \lvert \langle  \chi^{1+q},( \mathcal{B} \ast \mathcal{V} ) \alpha  \rangle \right \rvert.
\end{align*} 
For $m \geq 2$, the second term on the right hand side of \eqref{eq:dyn-moment} is bounded by 
\begin{multline}
   \left \lvert  \dfrac{1}{r} \, \sum_{l=1}^K \, \sum_{i=1}^{r\bar{E}^{r}_l(m-1)}  \langle \chi^{1+q} \left( \cdot -  \bar{S}^r (U^r_l (i)/r ,m+t ) \right),  \vartheta^{r,l}(i) \rangle \right \rvert \\
 + \left \lvert \dfrac{1}{r} \, \sum_{l=1}^K \, \sum_{i=1+r\bar{E}^{r}_l(m-1)}^{r\bar{E}^{r}_l(m+t)} \langle \chi^{1+q} \left( \cdot -  \bar{S}^r (U^r_l (i)/r ,m+t ) \right),  \vartheta^{r,l}(i) \rangle \right \rvert.  \label{eqo0}
\end{multline}
The second term in \eqref{eqo0} is bounded by 
$ 2 (L+1) \,  \left\lvert \langle \chi^{1+q},( \mathcal{B} \ast \mathcal{V} ) \alpha  \rangle \right\rvert.
$
Now, let us bound the first term of \eqref{eqo0}. To do this, we partition the interval $[0,m-1]$  into subintervals $[m-j-1,m-j]$, where the integer $j \in \lbrace 1 ,2, \cdots, m-1 \rbrace $. Thus, the first term is bounded by
\begin{align}
\sum_{j=1}^{m-1} \dfrac{1}{r} \, \sum_{l=1}^{K} \, \sum_{i=1+r\bar{E}^{r}_l(m-j-1)}^{r\bar{E}^{r}_l(m-j)}  \left \lvert \langle \chi^{1+q} \left( \cdot -  \bar{S}^r (U^r_l (i)/r ,m+t ) \right),  \vartheta^{r,l}(i) \rangle \right\rvert.   \label{eqo1}
\end{align}
Considering the definition of $\vartheta^{r,l}(i)$ in \eqref{es:mathca-X}, it becomes apparent that
\begin{align*}
\langle \chi^{1+q} \left( \cdot -  \bar{S}^r (U^r_l (i)/r ,m+t ) \right),  \vartheta^{r,l}_k(i) \rangle = \sum_{n=1}^{N^{l,r}_k(i)} \chi^{1+q} \big( V^r_{lk}(i,n)
- \bar{S}^r(U^r_l(i)/r,m+t) \big).
\end{align*}
Suppose there exists $t^*\in [0,m+t]$ such that $e.\bar{Z}^r(t^*)=0$. It follows that
\begin{align*}
\sum_{n=1}^{N^{l,r}_k(i)} \chi^{1+q} \big( V^r_{lk}(i,n)
- \bar{S}^r(U^r_l(i)/r,m+t) \big)  \leq  \sum_{n=1}^{N^{l,r}_k(i)} \chi^{1+q} \big( V^r_{lk}(i,n)
- \bar{S}^r(U^r_l(i)/r,t^*) \big).
\end{align*}
By applying the dynamic equation \eqref{eq:shifted-dynamic-equation2} with $m=0$, $t'=U^r_l(i)/r, h=t^*-U^r_l(i)/r$ and $g=\chi^{1+q}$, we obtain 
$$\sum_{n=1}^{N^{l,r}_k(i)} \chi^{1+q} \big( V^r_{lk}(i,n)
- \bar{S}^r(U^r_l(i)/r,t^*) \big)=0.$$
Hence, the quantity in \eqref{eqo1} is zero.
Now, consider the scenario where $e.\bar{Z}^r(s)>0$ for all $s \in [0,m+t]$. Fix $l\in \mathcal{K}$ and $j \in \lbrace 1, 2 , \cdots, m-1 \rbrace$. For $r\bar{E}^{r}_l(m-j-1) < i \leq r\bar{E}^{r}_l(m-j)$, it follows that $ U_l^r(i) /r \in \left]m-j-1,m-j \right]$. Therefore
$$
\bar{S}^r(U^r_l(i)/r,m+t) \geq \bar{S}^r(m-j,m) \geq j/M_{\gamma}.
$$
As a result, for each $k\in \mathcal{K}$
\begin{align*}
\sum_{n=1}^{N^{l,r}_k(i)} \chi^{1+q} \big( V^r_{lk}(i,n)
- \bar{S}^r(U^r_l(i)/r,m+t) \big)  \leq  \sum_{n=1}^{N^{l,r}_k(i)} \chi^{1+q} \big( V^r_{lk}(i,n)
- j/M_{\gamma} \big)   \\
 \leq  \sum_{n=1}^{N^{l,r}_k(i)} \big( 1_{\lbrace [j/M_{\gamma}, \infty) \rbrace}  \chi^{-(1+q)} \big) \big(V^r_{lk}(i,m) \chi^{2+2q} (V^r_{lk}(i,m)\big). 
\end{align*}
Hence, for each $k \in \mathcal{K}$
\begin{align}
\langle \chi^{1+q} \left( \cdot -  \bar{S}^r (U^r_l (i)/r ,m+t ) \right),  \vartheta^{r,l}_k(i) \rangle 
 \leq \left( \frac{M_{\gamma}} {j} \right)^{1+q} \, \langle \chi^{2+2q} , \vartheta^{r,l}_k(i) \rangle .  
 \label{eqo2}
\end{align}
Using \eqref{eqo1} and \eqref{eqo2}, we can bound the second term on the right-hand side of \eqref{eqo0} by
\begin{align*}
\sum_{j=1}^{m-1} \left( \left( \dfrac{M_{\gamma}}{j} \right)^{1+q} \, \dfrac{1}{r} \, \sum_{l=1}^{K}  \, \sum_{i=1+r\bar{E}^{r}_l(m-j-1)}^{r\bar{E}^{r}_l(m-j)} \left \lvert \langle \chi^{2+2q} , \vartheta^{r,l}(i) \rangle \right\rvert \right) 
 \leq 2 M_{\gamma}^{1+q} \, \left(1+1/q \right) \, \left\lvert \langle \chi^{2+2q}, ( \mathcal{B} \ast \mathcal{V} ) \alpha \rangle \right \rvert, 
\end{align*}
where the last inequality is derived from the event $\Omega^r_{LN}$ and by upper-bounding the $p$-series by
$
\left(1+ \int_1^{\infty} (1/x^{1+p}) \, dx \right) \leq (1+q^{-1}).
$
Finaly, we obtain 
$$
 \left\| \,  \langle \chi^{1+q}, \bar{\mathcal{Q}}^{r,m} (\cdot) \rangle    \, \right \|_L  \leq \widetilde{M}_{T,L},
$$
with
$$
\widetilde{M}_{T,L}= M_{0,q} + 2 M_{\gamma}^{1+q} \, \left\lvert \langle \chi^{2+2q}, ( \mathcal{B} \ast \mathcal{V} ) \alpha \rangle \right \rvert 
\left(1+ 1/q \right) + 2 (L+1) \, \left\lvert \langle \chi^{1+q},( \mathcal{B} \ast \mathcal{V} ) \alpha  \rangle \right \rvert.
$$
\end{proof}

\begin{lemma} \label{mass total gamma lk}
Assume \eqref{eq:assum1}-\eqref{eq:assum10} and \eqref{eq:assum-init1}-\eqref{eq:assum-init4}. Fix $T>0, L>1$ and $\eta > 0$, there exists a constant $\bar{M}_{T,L}>0$ such that
\begin{align}
\liminf_{r \rightarrow \infty} \,
  \mathbb{P}^r \, \left( \max_{m \leq \lfloor rT \rfloor} \,  \left\| \,  \langle 1, \bar{\mathcal{Q}}^{r,m} (\cdot)  \rangle  \, \right\|_L \leq \bar{M}_{T,L} \right) \geq 1-\eta \label{eq: mass total gamma lk}
\end{align}
\end{lemma}

\begin{proof}

Firstly, it should be noted that proving
\begin{align}
\liminf_{r \rightarrow \infty} \, \mathbb{P}^r \, \left(  \,  \left\| \langle 1, \bar{\mathcal{Q}}^{r} (t) \rangle \right\|_{\lfloor rT \rfloor + L} \leq \bar{M}_{T,L} \right) \geq 1-\eta
\end{align}
is equivalent to proving \eqref{eq: mass total gamma lk}.
By Lemmas \ref{exell} and \ref{tilde-M-T} , and \eqref{glika}, \eqref{const:M1}  we have
\begin{align*}
\liminf_{r \rightarrow \infty} \, \mathbb{P}^r \, \left( \Omega^r_E \cap \Omega_M^r \cap  \Omega_{LN}^r \cap \Omega^r_{\gamma} \right) \geq 1- \eta,
\end{align*} 
where $\Omega_M^r$ denotes the event in  \eqref{1+q}. 
Let $\omega$ be a fixed sample path in  $ \Omega^r_E \cap \Omega_M^r \cap  \Omega_{LN}^r \cap \Omega^r_{\gamma} $. In the remainder of the proof, all random objects are evaluated at this
$\omega$.
\par Let us cover the interval $[0,\lfloor rT \rfloor + L]$, by the overlapping intervals $I_{k'} = [k',k'+1]$ where $k' \in \lbrace 0,1, \cdots, \lfloor t_1 \rfloor \rbrace$, and $t_1=\lfloor rT \rfloor + L$. 
The concept here is that if the total mass is bounded by the same constant $\bar{M}_{T,L} > 0$ on each interval $I_{k'}$, then it remains bounded by this constant over the entire interval $[0,\lfloor rT \rfloor + L]$. Hence, it is enough to demonstrate that for every $k' \in \lbrace 0,1, \cdots, \lfloor t_1 \rfloor \rbrace$, we have
\begin{align*}
\sup_{t \in I_{k'}} \left| \langle 1, \bar{\mathcal{Q}}^{r} (t) \rangle \right| \leq \bar{M}_{T,L}. 
\end{align*}
Let $t \in I_{k'}$, by using the dynamic equation \eqref{eq:shifted-dynamic-equation2} with $m=0, \, g= 1_{\mathbb{R}_+}, \, t'=k', \, h=t-k' $, and by bounding the second term of the right-hand side of \eqref{eq:shifted-dynamic-equation2} by its total mass, we obtain
 \begin{align}
 \langle 1,\bar{\mathcal{Q}}^{r}(t) \rangle \leq
  \bar{\mathcal{Q}}^{r}(k')\left(  I(\bar{S}^{r}(k',t)) \right) 
  +\frac{1}{r} \sum_{l=1}^{K} \sum_{i=r\bar{E}_{l}^{r}(k')+1}^{r\bar{E}_{l}^{r}(t)} \, \langle 1, \vartheta^{r,l}(i) \rangle \label{eq prin}
\end{align}
As we are within the event $\Omega^r_{LN}$, the second term of the right-hand side of \eqref{eq prin} is bounded by 
\begin{align*}
 \left| \frac{1}{r} \sum_{l=1}^{K} \sum_{i=r\bar{E}_{l}^{r}(k')+1}^{r\bar{E}_{l}^{r}(k'+1)} \, \langle 1, \vartheta^{r,l}(i) \rangle \right| \leq \left| Q \alpha \right| + 1 = \left| \lambda \right| +1.
\end{align*}
Now, let us bound the first  term of the right-hand side of \eqref{eq prin}. Firstly, consider the scenario where $ e.\bar{Z}^r(s)>0$ for any $s \in [k',k'+1]$.
Given that we are within the event $\Omega^r_{\gamma}$, the total mass $ \langle 1,\bar{\gamma}^{r}(t) \rangle  $ is bounded by $M_{\gamma}$ over the interval $[0,\lfloor rT \rfloor +L]$. Consequently, $ \langle 1,\bar{\gamma}^{r}(t) \rangle  $ is bounded on each interval $I_{k'}$ with $k'\in \lbrace 0, 1, \cdots, \lfloor t_1 \rfloor -1 \rbrace$. However, it is necessary to bound $ \langle 1,\bar{\gamma}^{r}(t) \rangle  $ on $I_{\lfloor t_1 \rfloor}$. 
Using \eqref{eq:gamma-r} and considering that we are within the event $\Omega^r_E$, we have for $t \in I_{\lfloor t_1 \rfloor}$
\begin{align*}
\langle 1, \bar{\gamma}^{r}(t) \rangle  \leq  \langle 1, \bar{\gamma}^{r}(\lfloor t_1 \rfloor) \rangle +  \sum_{k=1}^{K} (\bar{E}^{r}_{k}(\lfloor t_1 \rfloor +1)- \bar{E}^{r}_{k}(\lfloor t_1 \rfloor)) \leq M_{\gamma} + e.\alpha.
\end{align*}
This suggests that for each $k'\in \lbrace 0, 1, \cdots, \lfloor t_1 \rfloor  \rbrace$, we have
$ \sup_{t \in I_{k'}} \langle 1, \bar{\gamma}^{r} (t) \rangle  \leq M_{\gamma} + e.\alpha$.
Therefore
$
\bar{S}^r(k',k'+1)= \int_{k'}^{k'+1} 1/ (\langle 1, \bar{\gamma}^{r} (s) \rangle) ds \geq 1/(M_{\gamma}+e.\alpha).
$
Thus, the first term on \eqref{eq prin} is bounded by
\begin{align*}
\left| \bar{\mathcal{Q}}^{r} (k') \left( I(\bar{S}^r(k',k'+1) ) \right) \right| & \leq \left| \bar{\mathcal{Q}}^{r} (k') \left( I \left( 1/(M_{\gamma}+ e.\alpha) \right)  \right) \right|  \\ 
 & \leq
 \left( M_{\gamma} + e.\alpha \right)^{1+q} \left| \langle \chi^{1+q}, \bar{\mathcal{Q}}^{r} (k') \rangle \right| 
  \leq \left( M_{\gamma}+ e.\alpha \right)^{1+q} \widetilde{M}_{T,L},  
\end{align*}
where the second inequality is derived from Markov's inequality and the last one arises from fact that we are within the event $\Omega^r_M$. 
Now, assume there exists an $s \in [k',k'+1]$ such that $e.\bar{Z}^r(s)=0$. In that scenario by applying the dynamic equation \eqref{eq:shifted-dynamic-equation2} with $m=0, \, g= 1_{\mathbb{R}_+}, \, t'=k', \, h=s-k' $, we obtain $ \bar{\mathcal{Q}}^{r} (k') ( I(\bar{S}^r(k',s) ) = 0$, and since $ \bar{S}^r(k',s) \leq \bar{S}^r(k',k'+1) $, then
$$
\bar{\mathcal{Q}}^{r} (k') \left( I(\bar{S}^r(k',k'+1) \right) = 0 \leq \left( M_{\gamma}+ e.\alpha \right)^{1+q} \widetilde{M}_{T,L}.
$$
Finally, we have our result with the constant 
$
\bar{M}_{T,L} = \left( M_{\gamma}+ e.\alpha \right)^{1+q} \widetilde{M}_{T,L} + |\lambda | +1 .
$

\end{proof}


\subsubsection{Compact containment} \label{presub5}~

 The following lemma establishes the compact
containment of the shifted fluid scaled process $\bar{\mathcal{Q}}^{r,m}(\cdot)$ on $[0,L]$, which is the first step to prove the compactness. 
\begin{lemma} 
\label{cocn}
Assume \eqref{eq:assum1}-\eqref{eq:assum10} and \eqref{eq:assum-init1}-\eqref{eq:assum-init4}. Fix $T>0,~ L>1$ and $\eta > 0$, there exists a compact subset $\mathbb{K}_{T,L} \subset \mathcal{M}^K$ such that
$$ 
\underset {r \rightarrow \infty}{\liminf} \,
  \mathbb{P}^r \, \left( \bar{\mathcal{Q}}^{r,m} (t) \in \mathbb{K}_{T} \quad \forall  m \leq \lfloor rT \rfloor \quad t\in [0,L] \,
   \right) \geq 1- \eta 
$$
\end{lemma}

\begin{proof}
From Lemmas \ref{tilde-M-T} and  \ref{mass total gamma lk}, we obtain
\begin{equation}
\liminf_{r\to \infty}
  \mathbb{P}^r\left( \max_{m \leq \lfloor rT \rfloor} \,  \left\| \,  \langle 1,\bar{\mathcal{Q}}^{r,m} (\cdot) \rangle \vee \langle \chi^{1+q},\bar{\mathcal{Q}}^{r,m} (\cdot) \rangle  \,  \right\|_L \leq M_T\right) = 1, 
  \label{event:moment}
\end{equation}
where 
$M_{T,L}=\widetilde{M}_{T,L} \vee \bar{M}_{T,L}$. 
Define 
$\mathbb{C}_{T,L} = \lbrace \xi \in \mathcal{M} :  \langle 1 , \xi \rangle  \vee \langle \chi^{1+q}, \xi \rangle  \leq M_T \rbrace$
such that for $m \leq \lfloor rT \rfloor$ and $t \in [0,L]$,
we have
$
\bar{\mathcal{Q}}^{r,m} (t) \in \left(\mathbb{C}_{T,L}\right)^K.  
$
By the Markov's inequality, we have  
$$ \underset{\xi \in \mathbb{C}_{T,L}}{ \text{ sup }} \langle 1_{[c, \infty)} , \xi \rangle \rightarrow 0 ~~ \text{as} ~~ c \rightarrow \infty,$$
which implies that $\mathbb{C}_{T,L}$ is relatively compact. Thus, the set $\mathbb{K}_{T,L}=\mathbb{C}^K_{T,L}$ is relatively compact.
\end{proof}

\subsubsection{Asymptotic regularity}~ \label{presub6}

In the following lemma, we show that the shifted fluid scaled process $\bar{\mathcal{Q}}^{r,m}(\cdot)$ assigns arbitrarily small mass to small intervals; this is essential to establish the oscillation bound. 
\begin{lemma} \label{as lm} Assume  \eqref{eq:assum1}-\eqref{eq:assum10} and \eqref{eq:assum-init1}-\eqref{eq:assum-init4}. 
Fix $T > 0$ and $L > 1$. For each $\varepsilon ,\eta > 0$ there exists a $\kappa > 0$ (depending on $\varepsilon$ and $\eta$) such that 
\begin{equation}
\underset{r \rightarrow \infty}{\liminf} \, \mathbb{P}^r \left( \underset{m \leq \lfloor rT \rfloor}{\max} \; \sup_{x \in \mathbb{R}_+}  \, \left \| \,  \bar{\mathcal{Q}}^{r,m} (\cdot) \left( [x,x+\kappa] \right)  \, \right \|_L \leq \varepsilon \right) \geq 1-\eta \label{asymp reg}
\end{equation}
\end{lemma}	
 
\begin{proof}

By definition of the shifted scaled process $\bar{\mathcal{Q}}^{r,m}(\cdot)$ it suffices to prove  
 \begin{align*}
 \underset{r \rightarrow \infty}{\liminf} \,  \mathbb{P}^r \left( \, \sup_{x \in \mathbb{R}_+}  \, \left\| \bar{\mathcal{Q}}^{r} (t) \left( [x,x+\kappa] \right) \right\|_{\lfloor rT \rfloor + L} \leq \varepsilon \right) \geq 1-\eta.
 \end{align*}
Define the event
$
\Omega_z^r = \displaystyle{ \left \lbrace \underset{ x \in \mathbb{R}_+}{\sup }  \, \left| \bar{\mathcal{Q}}^{r} (0) ([x,x+\kappa])  \right| \leq \varepsilon/2  \right \rbrace } .
$
By employing a reasoning similar to that used to derive (79) in \cite{zhang2009law}, we obtain
\begin{align}
\liminf_{r \rightarrow \infty} \, \mathbb{P}^r \left( \Omega^r_{z} \right) \geq 1-\eta. \label{asymptotic 1}
\end{align}
From equations \eqref{glika}, \eqref{const:M1}, \eqref{event:moment}, and \eqref{asymptotic 1}, we derive
\begin{align}
\underset{r \rightarrow \infty}{\liminf} \,
\mathbb{P}^r (\Omega_A^r) \geq 1-\eta. \label{ev:const-A}
\end{align}  
Here, the event $\Omega^r_A:= \Omega_z^r \cap \Omega_{LN}^r \cap \Omega^r_{\gamma} \cap \Omega^r_C$, with $\Omega^r_C$ representing the event in \eqref{event:moment}.
Let $\omega$ be a fixed sample path in  $ \Omega_A^r $. In the remainder of the proof, all random variables are evaluated at this
$\omega$.
\par Let $t \in [0,\lfloor rT \rfloor + L]$, and recall the constants $M_{\gamma}$ and $M_T$ defined in the events $\Omega^r_{\gamma}$ and $\Omega^r_C$, respectively. Define  
\begin{align*}
t_1 = \max \Big( t_0, t - M_{\gamma} \left( 2 M_T / \varepsilon \right)^{1/(1+q)} \Big),
\end{align*}
where $ t_0= \sup  \lbrace s \leq t :  \langle 1,\bar{\mathcal{Q}}^{r} (s) \rangle  \leq \varepsilon/4 \rbrace $ if the supremum exists, and zero otherwise.
Applying the  dynamic equation \eqref{eq:shifted-dynamic-equation2} with $g= 1_{\lbrace [x,x+\kappa] \rbrace}, \, t'=t_1$ and $h=t-t_1$; we have for each $x \in \mathbb{R}_+$ 
\begin{align}
\bar{\mathcal{Q}}^{r} (t) ([x,x+\kappa]) & = \bar{\mathcal{Q}}^{r} (t_1) ([x,x+\kappa] + \bar{S}^{r} (t_1,t) ) \nonumber \\
& + \dfrac{1}{r} \sum_{l=1}^K \sum_{i= r \bar{E}^{r}_l (t_1) + 1 }^{ r \bar{E}^{r}_l (t)} \, \vartheta^{r,l} (i)  
([x,x+ \kappa] + \bar{S}^{r} ( U_l^r (i)/r,t)). \label{dy}
\end{align}  
If $t_1=0$, then by \eqref{asymptotic 1}  for each $x\in \mathbb{R}_+$,
$
\left| \bar{\mathcal{Q}}^{r} (0) ([x,x+\kappa] + \bar{S}^{r} (t)) \right| \leq \varepsilon /2.
$
If $t_1=t_0 >0$ then  for each $\delta>0$, there exists $s \in \: (t_1-\delta, t_1 ]$ such that 
\begin{align}
 \left| \langle 1,\bar{\mathcal{Q}}^{r} (s) \rangle \right|  \leq \varepsilon /4 . \label{Qbar}
 \end{align}
By applying \eqref{eq:shifted-dynamic-equation2} to $g= 1_{\lbrace [x,x+ \kappa] + \bar{S}^{r} ( t_1,t) \rbrace}$, $t'=s$ and $h=t_1-s$, we have
\begin{align*} 
\bar{\mathcal{Q}}^{r} (t_1) ([x,x+ \kappa] + \bar{S}^{r} ( t_1,t))  \leq \langle 1, \bar{\mathcal{Q}}^{r} (s) \rangle + \dfrac{1}{r} \sum_{l=1}^K \sum_{i=r \bar{E}^{r}_l (s) + 1 }^{ r \bar{E}^{r}_l (t_1)}  \langle 1, \vartheta^{r,l} (i) \rangle.  
\end{align*}
Given that we are on the event $\Omega^r_A$, especially $\Omega^r_{LN}$, then for large $r$,  it follows that
$$
\left| \bar{\mathcal{Q}}^{r}(t_1) ([x,x+ \kappa] + \bar{S}^{r} ( t_1,t)) \right| \leq \varepsilon /4 + 
2 \,  \delta \, \left| \lambda \right|.   
$$
 If we choose $\delta$ sufficiently small such that $ 2 \,  \delta \, | \lambda | < \varepsilon /4$, we  obtain
$$
\left| \bar{\mathcal{Q}}^{r}(t_1) ([x,x+ \kappa] + \bar{S}^{r} ( t_1,t)) \right| < \varepsilon /2.
$$
\item If $t_1= t - M_{\gamma} \left( 2 M_T / \varepsilon \right)^{\frac{1}{1+q}} > 0$, then $ \bar{S}^{r} ( t_1,t) \geq (t-t_1)/ M_{\gamma} \geq \left( 2 M_T / \varepsilon \right)^{\frac{1}{1+q}}$. Hence,
\begin{align*}
\left| \bar{\mathcal{Q}}^{r}(t_1) ([x,x+ \kappa] + \bar{S}^{r} ( t_1,t)) \right| & \leq \left| \bar{\mathcal{Q}}^{r}(t_1) \left(I \left(\left( 2 M_T / \varepsilon \right)^{\frac{1}{1+q}} \right) \right) \right| \\
& \leq \left| \langle \chi^{1+q} , \bar{\mathcal{Q}}^r(t_1) \right| \, (\varepsilon / 2 M_T)  \leq \varepsilon /2,
\end{align*}
where the second inequality follows from Markov's inequality, and the final one is a consequence of the event $\Omega^r_C$.
As a result, the first term in the right-hand side of \eqref{dy} is bounded by $\varepsilon /2$ . 
\par Now let us bound the second term on the right-hand side of \eqref{dy}. Denote
\begin{align}
I^r:= \dfrac{1}{r} \, \sum_{l=1}^K \sum_{i= r \bar{E}^{r}_l (t_1) + 1 }^{ r \bar{E}^{r}_l (t)} \, \vartheta^{r,l} (i)  
([x,x+ \kappa] + \bar{S}^{r} ( U_l^r (i)/r,t)). 
\end{align}
Fix $\delta>0$ and $0< \kappa< \delta/M_{\gamma}$. Let $ t_1, \ldots, t_N = t$ be a partition of the interval $[t_1,t]$ such that 
$t_{j+1}-t_j = \delta$ for all $j=1, \ldots, N-1$.  By definition of $t_1$ , we have  
$ N \leq (M_{\gamma}/\delta) \, \left(2 M_T/\varepsilon \right)^{1/(1+q)}. $ 
We can write $I^r$ as 
\begin{align*}
I^r= \sum^{N-1}_{j=1} \, \dfrac{1}{r} \, \sum_{l=1}^K \sum_{i= r \bar{E}^{r}_l (t_j) + 1 }^{ r \bar{E}^{r}_l (t_{j+1})} \, \vartheta^{r,l} (i) ([x,x+ \kappa] + \bar{S}^{r} ( U_l^r (i)/r,t)).
\end{align*}
We have
$$
\left| I^r \right|  \leq \;  \sum^{N-1}_{j=0} \left| \dfrac{1}{r} \sum_{l=1}^K \sum_{i= r \bar{E}^{r}_l (t_j) + 1 }^{ r \bar{E}^{r}_l (t_{j+1})} \,  \vartheta^{r,l} (i) (C_j)  \right| ,
$$
where 
$$ C_j=[x+\bar{S}^r(t_{j+1},t), x + \kappa + \bar{S}^r(t_{j},t)].
$$
Choose 
$$\varepsilon_1= \frac{\delta \varepsilon^{1+1/(1+q)}}{4 M_{\gamma} (2M_T)^{1/(1+q)}}.$$ 
For large $r$ we have
$$
| I^r |  \leq \; \delta  \sum^{N-1}_{j=0} \left|   (\mathcal{B}^r \ast \mathcal{V}^r)  (C_j) \alpha \right| + N \varepsilon_1.
$$
The intervals $\lbrace C_j \rbrace_{j=1}^{N-1}$ are pairwise disjoints. In fact, 
$$  \bar{S}^r(t_{j+2},t) - \bar{S}^r(t_{j},t) - \kappa = \bar{S}^r(t_{j+2},t_{j}) - \kappa \geq 2 \delta/M_{\gamma} - \kappa >0, $$
where the last inequality follows from the definition of $\delta$ and $\kappa$. Therefore,
$$
| I^r |  \leq \; \delta \left|   (\mathcal{B}^r \ast \mathcal{V}^r)  (\bigcup^{N-1}_{j=0} C_j) \alpha \right| + N \varepsilon_1.
$$
Since $ \mathcal{B}^r \ast \mathcal{V}^r$ is a positive finite measure, then if we choose $\kappa, \delta<1$ small, we get
$$
| I^r |  \leq \; 2 N \varepsilon_1 \leq \varepsilon/2.
$$
\end{proof}

\subsubsection{Oscillation bound} \label{presub7}~

 In this section, we establish the oscillation bound result, which is the second major step to prove the precompactness. For  $L>1$, $\zeta (\cdot) \in \mathbf{D}([0,\infty), \mathcal{M}^K)$ and $\delta > 0$, we define the modulus of continuity of $\zeta (\cdot)$ on $[0 , L ]$ as
$$
	\mathbf{w}_{L} (\zeta (\cdot)
	 , \delta) = \, \sup_{s,t \in [0,L ], \lvert s-t \rvert < \delta}  \mathbf{d}
	    \left(\zeta(s) , \zeta(t)\right).
$$

\begin{lemma} \label{Os lm}
Assume \eqref{eq:assum1}-\eqref{eq:assum10} and \eqref{eq:assum-init1}-\eqref{eq:assum-init4}. Fix $T>0,~ L>1$. For each $\varepsilon, \eta >0$, there exists a $\delta > 0$ such that 
\begin{align}
\liminf_{r \rightarrow \infty} \mathbb{P}^r 
\left(\max_{m \leq \lfloor rT \rfloor } \mathbf{w}_{ L } (\bar{\mathcal{Q}}^{r,m} (\cdot)
	 , \delta) \leq \varepsilon \right) \geq 1- \eta \label{OS}
\end{align}
\end{lemma}

\begin{proof}
Define the event
$
\Omega^r_{As} =\left \lbrace  \max_{m \leq \lfloor rT \rfloor} \,  \left\| \,  \bar{\mathcal{Q}}^{r,m} (\cdot)  ([0,\kappa])  \, \right\|_L \leq \varepsilon /4 \right \rbrace.
$
According to Lemma \ref{as lm},  for each $\varepsilon,\eta>0$, there exists a 
$\kappa>0$ such that
\begin{align}
\liminf_{r \rightarrow \infty} \,
\mathbb{P}^r \left( \Omega^r_{As} \right) \geq 1-\eta. \label{AS0} 
\end{align}
From \eqref{glika} and \eqref{AS0}, we have
$
\liminf_{r \rightarrow \infty} \,
\mathbb{P}^r \left( \Omega^r_{LN} \cap \Omega^r_{As}   \right) \geq 1-\eta. 
$
Fix $r>0$ and let $\omega $ be a fixed sample path in
$\displaystyle{ \Omega^r_{As} \cap \Omega^r_{LN} } $.
For the remainder of the proof  all random quantities with index $r$ are evaluated at this $\omega$. 
Fix $m \leq \lfloor rT \rfloor$ and define
\begin{gather*}
I:=\lbrace u \in [0,L]: e. \bar{Z}^{r,m}(u)=0 \rbrace \\
m_{\gamma}:= \inf  \{ e. \bar{Z}^{r,m}(u): ~ u \in [0,L]\setminus I  \}.
\end{gather*}
Choose $ \delta= \min \left\lbrace  m_{\gamma} \kappa, m_{\gamma} \varepsilon, \varepsilon/4|\lambda| \right\rbrace $ and 
consider $0 \leq s < t \leq L$ with
$t-s < \delta$. Fix $k \in \mathcal{K}$ and let $B \subset \mathbb{R}^+$ a closed Borel set.  By definition of the  metric $\mathbf{d} (\cdot,\cdot)$, it
suffices to show  the following two inequalities 
\begin{align}
  \bar{\mathcal{Q}}^{r,m}_k (s) (B) & \leq  \bar{\mathcal{Q}}^{r,m}_k (t) (B^{\varepsilon}) +  \varepsilon, \label{fi} \\ 
  \bar{\mathcal{Q}}^{r,m}_k(t) (B) & \leq  \bar{\mathcal{Q}}^{r,m}_k(s) (B^{\varepsilon}) +  \varepsilon. \label{se}
\end{align} 
\vspace{0.15cm}
\textbf{Case 1:} If $I \cap [s,t] \neq \oslash$, then $e.\bar{Z}^{r,m}(u)>0$ on the interval $[s,t]$.
By definition of the event $\Omega_{As}^r$, we have
\begin{align}
\bar{\mathcal{Q}}^{r,m}_k (s) (B)  \leq \bar{\mathcal{Q}}^{r,m}_k (s) ([0,\kappa]) + \bar{\mathcal{Q}}^{r,m}_k (s) (B \cap I( \kappa))  \leq  \varepsilon /4 + \bar{\mathcal{Q}}^{r,m}_k (s)(B \cap I( \kappa)). \label{eqelee}
\end{align} 
Considering the definitions of $m_{\gamma}$ and $\delta$, we obtain
\begin{equation}
\bar{S}^{r,m}(s,t) \leq \dfrac{(t-s)}{m_{\gamma}} \leq \dfrac{ \delta}{m_{\gamma}} \leq \min \lbrace \kappa , \varepsilon \rbrace. \label{cdd}
\end{equation}
Consequently, 
 $ B \cap I( \kappa) \subset B \subset B^{\varepsilon} + \bar{S}^{r,m} (s,t).$
 From \eqref{eqelee} and by applying \eqref{eq:shifted-dynamic-equation2} with $t'=s, \, h=t-s, \, g=1_{ B^{\varepsilon} }$ , we obtain
$$
\bar{\mathcal{Q}}^{r,m}_k (s) (B)  \leq  \dfrac{\varepsilon}{4} + \bar{\mathcal{Q}}^{r,m}_k (s) (B^{\varepsilon} + \bar{S}^{r,m} (s,t))  \leq \varepsilon + \bar{\mathcal{Q}}^{r,m}_k (t) (B^{\varepsilon}).
$$
Therefore \eqref{fi} is proven.  
\par Let us show \eqref{se}. 
By using \eqref{eq:shifted-dynamic-equation2} with $t'=s$ and $h=t-s$, and by bounding the the second term by its total mass, we get
\begin{align*}
\bar{\mathcal{Q}}^{r,m}_k (t) (B)  & \leq \bar{\mathcal{Q}}^{r,m}_k (s) (B + \bar{S}^{r,m}(s,t)) + \dfrac{1}{r} \sum_{l=1}^K \sum_{i= r \bar{E}^{r}_l (s) + 1 }^{ r \bar{E}^{r}_l (t)} \, \langle 1 , \vartheta^r_{lk} (i) \rangle \\
 & \leq \bar{\mathcal{Q}}^{r,m}_k (s) (B^{\varepsilon}) +
  \dfrac{\varepsilon}{2}, 
\end{align*} 
where the second inequality above comes from the fact that $(\ref{cdd})$ implies $ B + \bar{S}^{r,m} (s,t) \subset B^{\varepsilon}  $, and by definition of $\delta$. Therefore \eqref{se} is proven.

\vspace{0.2cm}
\textbf{ Case 2: $I \cap [s,t] \neq \oslash$.}  Let $\tau= \inf I \cap [s,t]$, consequently, $e.\bar{Z}^{r,m}(u)>0$ on the interval $[s,\tau)$.
Therefore,  
$$ \bar{S}^{r,m}(s,\tau) \leq \frac{\delta}{m_{\gamma}} < \kappa. $$
Given that we are in the event $\Omega^r_{As}$, we have
\begin{align*}
\bar{\mathcal{Q}}^{r,m}_k (s) (B)   \leq \bar{\mathcal{Q}}^{r,m}_k (s)([0, \kappa]) +  \bar{\mathcal{Q}}^{r,m}_k (s)(I(\kappa)) 
  \leq \varepsilon/4+ \bar{\mathcal{Q}}^{r,m}_k (s) (I(\bar{S}^{r,m}(s,\tau))).
\end{align*}
By using \eqref{eq:shifted-dynamic-equation2} with $g=1_{(0, \infty)}$, $t'=s$ and $h=\tau-s$, and by using the fact that $\langle 1, \bar{\mathcal{Q}}^{r,m}_k(\tau) \rangle=0$, we get that 
$$ \bar{\mathcal{Q}}^{r,m}_k (s) (I(\bar{S}^{r,m}(s,\tau)))=0.$$
Hence, 
$ \bar{\mathcal{Q}}^{r,m}_k (s) (B) \leq \varepsilon/4.$ Therefore \eqref{fi} is proven.
Let us establish \eqref{se}. Once again, by  using \eqref{eq:shifted-dynamic-equation2} with $g=1_B$, $t'=\tau$ and $h=t-\tau$, and by bounding the second term by its total mass, we obtain
\begin{align*}
\bar{\mathcal{Q}}^{r,m}_k (t) (B)  & \leq \bar{\mathcal{Q}}^{r,m}_k (\tau) (B+ \bar{S}^{r,m}(\tau,t)) +  \dfrac{1}{r} \sum_{l=1}^K \sum_{i= r \bar{E}^{r}_l (s) + 1 }^{ r \bar{E}^{r}_l (t)} \, \langle 1 , \vartheta^r_{lk} (i) \rangle   \\
 & \leq 2 \delta |\lambda| \leq \varepsilon/2 ,
\end{align*}
where the second inequality follows from the fact that $\langle 1, \bar{\mathcal{Q}}^{r,m}_k (\tau) \rangle=0$ and the condition that we are within the event $\Omega^r_{LN}$, while the last one stems from the definition of $\delta$. Therefore \eqref{se} is proven.
\end{proof}

\subsubsection{Precompactness} \label{presub8}~ 

Building on the findings from earlier sections, we define an event on which any sequence of simple paths of the shifted scaled process $\bar{\mathcal{Q}}^{r,m} (\cdot)$ is relatively compact.

\par By Lemma $\ref{Os lm}$, for each $n \in \mathbb{N}$ and for all $\eta>0$, there exists $\mathcal{C}=\lbrace (\delta_k) \rbrace^{\infty}_{k=1}$ such that
$$
\liminf_{r \rightarrow \infty} \, \mathbb{P}^r \big(\bigcap_{k=1}^{n}  \Omega^r (k,\mathcal{C})\big) \geq 1-\eta, 
$$
where $
\Omega^r (k,\mathcal{C})= \left \lbrace \max_{m \leq \lfloor rT \rfloor}  \mathbf{w}_L ( \bar{\mathcal{Q}}^{r,m} (\cdot), \delta_k)  \leq 1/k \right \rbrace.
$
Then, for each $n \in \mathbb{N}$  there exists a $r(n) \geq 0$ such that
$$
\mathbb{P}^r \big(\bigcap_{k=1}^{n}  \Omega^r (k,\mathcal{C})\big) \geq 1-\eta \quad \forall r>r(n). 
$$
Define 
$
 n(r)=
\sup \lbrace n\in \mathbb{N}: r(n) <r \rbrace \text{ if } r>r(1)$ and $0$ otherwise; 
and let
$ \Omega^r(\mathcal{C})= \bigcap_{k=1}^{n(r)} \Omega^r (k,\mathcal{C}). 
$
On one hand  $n(r)$ goes to infinity as $r \rightarrow \infty$. On the other hand the set $\Omega^r (\mathcal{C})$ is non empty since $n(r)>1$  for large $r$. This implies
\begin{align} \label{nr} 
\liminf_{r \rightarrow \infty} \mathbb{P}^r (\Omega^r (\mathcal{C})) \geq 1-\eta.
\end{align}

As previously demonstrated in Section \ref{Singleclass}, the processor $\gamma^r$ behaves as a single class processor sharing queue. Consequently, according to  \cite[Lemma 4.4]{gromoll2004diffusion}, there exists an event $\Omega^r_{\Gamma}$ such that
\begin{align} \label{ev_Omg}
\liminf_{r \rightarrow \infty} \mathbb{P}^r (\Omega^r_{\Gamma}) \geq 1-\eta,
\end{align}
where the process
  $\lbrace \bar{\gamma}^{r,m} (\cdot), m \leq \lfloor rT \rfloor \rbrace$ is uniformaly approximated  on $[0,L]$ by a certain process
 $\tilde{\gamma}(\cdot)$ which is a fluid solution of \eqref{eq:bar-gamma}. Let $(r_n)_{n \in  \mathbb{N}}$ a sequence on $\mathbb{R_+}$ which goes to infinity, we have 
\begin{align}
\label{unif_gam}
\lim_{n \rightarrow \infty} \sup_{t \in [0,L]}
\max_{m_n \leq \lfloor r_n T \rfloor} \boldsymbol\rho \, \left( \bar{\gamma}^{r_n,m_n} (\omega,t) , \tilde{\gamma} (t) \right) = 0, \\
\label{unif_W}
\lim_{n \rightarrow \infty} \sup_{t \in [0,L]}
\max_{m_n \leq \lfloor r_n T \rfloor}  \left \lvert \bar{W}^{r_n,m_n} (\omega,t) - \tilde{W} (t) \right \rvert = 0,
\end{align}
where $ \tilde{W} (t) = \langle \chi,\tilde{\gamma}(t) \rangle $. 
Let $\Omega^r_{CC}$ denote the event in   
\eqref{event:moment} and let $
\Omega^r_F=\Omega^r_{LN} \cap \Omega^r _{CC} \cap \Omega^r (\mathcal{C}) \cap \Omega^r_{\Gamma}
$. From \eqref{glika},  
\eqref{event:moment} and \eqref{nr}-\eqref{ev_Omg}, It follows that
\begin{equation}
\liminf_{r \rightarrow \infty} \mathbb{P}^r (\Omega^r_F) \geq 1-\eta .  \label{comp}
\end{equation}
Denote by $\mathcal{D}^r_L$ the set of $ \xi^r (\cdot) \in \mathbf{D} ( \, [0,L], \mathcal{M}^K)$ such that $\xi^r (\cdot) \equiv   \bar{\mathcal{Q}}^{r,m}(\omega,\cdot)$ with 
$\omega \in \Omega^{r}_F$ and $m \leq \lfloor r T \rfloor$. 
Now, we  we present the precompactness result of $\bar{\mathcal{Q}}^{r,m} (\cdot)$, as derived from \cite[Theorem 3.6.3]{ethier1986markov}.

\begin{theorem} 
\label{preco}
Fix $T >0 $ and assume \eqref{eq:assum1}-\eqref{eq:assum10} and \eqref{eq:assum-init1}-\eqref{eq:assum-init4}. Let $(r_n)_{n \in  \mathbb{N}}$ a sequence on $\mathbb{R_+}$ approaching infinity. Any sequence $\left(\xi^{r_n} (\cdot)\right)_{n\in \mathbb{N}}$ such that  $\xi^{r_n} (\cdot) \in \mathcal{D}^{r_n}_L$ for each $n \in \mathbb{N}$, is precompact in $\mathbf{D} ( \, [0,L], \mathcal{M}^K)$. Furthermore, any subsequence  $\left(\xi^{r_{n_{i}}}(\cdot) \right)_{i\in \mathbb{N}}$ converges uniformly to a continuous process $\tilde{\mathcal{Q}}(\cdot)$. In other words
\begin{align}
\lim_{i \rightarrow \infty} \, \sup_{0 \leq t \leq L} \, \mathbf{d} \left(\xi^{r_{n_{i}}} (t),\tilde{\mathcal{Q}}(t) \right) = 0. \label{metd1}
\end{align} 
\end{theorem}



\subsection{ State space collapse proof}~ 
\label{sec:ssc}

In this section, we prove the state space collapse stated in Theorem \ref{ssc2}. For the rest of this paper, $\omega$ is fixed on the event $\Omega^r_F$.
\par We consider the multiclass fluid model defined in Section \ref{sect:fluid-mod} associated with critical data $(\alpha,\nu,P)$ and initial state $\xi \in \mathcal{M}^{c,K}$. Let $\widetilde{\mathcal{D}}_{L}$ be the set of the convergent
subsequences in  Theorem \ref{preco}, which is non-empty.
The following lemma establishes that the elements of this set are fluid solutions to equation \eqref{fluid-mathcal-Q}. 

\begin{lemma} \label{lm1} 
Assume \eqref{eq:assum1}-\eqref{eq:assum10} and \eqref{eq:assum-init1}-\eqref{eq:assum-init4}. Let $q$ the constant defined in \eqref{const-q}. Fix $L>1$ and $T>1$, any element $\widetilde{\mathcal{Q}}(\cdot) \in \widetilde{\mathcal{D}}_{L}$ is a fluid  solution of the equation \eqref{fluid-mathcal-Q}.
\end{lemma}


\par The following lemma  states that for sufficiently large $r$, any shifted fluid scaled process $\bar{\mathcal{Q}}^{r,m}(\cdot)$ evaluated at a some sample path in $\Omega^r_F$ is uniformly approximated on $[0,L]$ by some element of $\widetilde{\mathcal{D}}_{L}$. 

\begin{lemma} \label{lmf} Assume \eqref{eq:assum1}-\eqref{eq:assum10} and \eqref{eq:assum-init1}-\eqref{eq:assum-init4}. Fix $L>1$. For any $\xi^{r} \in  \mathcal{D}^{r}_L$ there exists  $\widetilde{\mathcal{Q}}(\cdot) \in \widetilde{\mathcal{D}}_{L}$, such that
\begin{align*}
\lim_{r \rightarrow \infty}  \, \sup_{t \in [0,L]} \mathbf{d}\left( \xi^{r}(t) , \widetilde{\mathcal{Q}} (t) \right) = 0.
\end{align*}
\end{lemma}

The proof of those previous lemmas are omitted as it closely resembles arguments from the literature for single-class systems.

\subsubsection*{Proof of Theorem \ref{ssc2}}~ 

Fix $\omega \in \Omega^r_F$, and assume that all random objects are evaluated at this $\omega$. According to \eqref{comp} it is sufficient to show that 
for all $\varepsilon>0$ there exists $r'>0$ such that for all $r>r_0$
\begin{align}
\sup_{t \in [0,T]} \mathbf{d} \left( \widehat{\mathcal{Q}}^r (\omega,t) , \mathcal{B}\ast\Delta^{\nu} \,  \widehat{W}^r (\omega,t) \right ) \leq \varepsilon. \label{mainn}
\end{align}
Fix $ \varepsilon>0$. By Lemma \ref{lm1}, any $\widetilde{\mathcal{Q}}(\cdot) \in \widetilde{\mathcal{D}}_{L}$ is 
a fluid solution to the equation \eqref{fluid-mathcal-Q}. By \eqref{eq:assum6} and \eqref{eq:assum-init4}, we have $\langle\chi^{3+\theta},\mathcal{V}^0\rangle<\infty$. Consequently, by the definition of the constant $q$ in \eqref{const-q}, we derive $\langle \chi^{1+q}, \xi \rangle$. Therefore, it follows from Proposition \ref{asmp} that
$$ \mathbf{d} \left ( \widetilde{\mathcal{Q}} (t) , \mathcal{B}\ast\Delta^{\nu} \widetilde{W}(t) \right ) \xrightarrow[t \rightarrow +\infty]{} 0, $$
where $\widetilde{W}(\cdot) = \langle \chi,\widetilde{\gamma} (\cdot) \rangle$ and such that $\widetilde{\gamma} (\cdot)$ is a fluid solution of \eqref{eq:bar-gamma}.
This implies the existence of  $L^{\ast} > 0$ such that for all $s \geq  L^{\ast}$,
\begin{equation}
\mathbf{d} \left ( \widetilde{\mathcal{Q}} (s) , \mathcal{B}\ast\Delta^{\nu} \, \widetilde{W}(s) \right)  < \varepsilon /3. \label{aw1} 
\end{equation} 
Fix a constant $L>L^{\ast}+1$. 
Considering that the interval $[0,rT]$ is encompassed by the union of the interval $[0,L^{\ast}]$ and the set of overlapping intervals $ [m+L^{\ast},m+L]$ with $m \leq \lfloor rT \rfloor$, demonstrating \eqref{mainn} requires showing
\begin{align}
\max_{m \leq \lfloor rT \rfloor} \sup_{t \in [L^{\ast},L]} \, \mathbf{d} \, \left( \bar{\mathcal{Q}}^{r,m} (\omega,t) , \mathcal{B}\ast\Delta^{\nu} \, \bar{W}^{r,m} (\omega,t) \right) & \leq \varepsilon, \label{sl1} \\
\sup_{t \in [0,L^{\ast}]} \, \mathbf{d} \,  \left( \bar{\mathcal{Q}}^{r,0} (\omega,t) , \mathcal{B}\ast\Delta^{\nu} \, \bar{W}^{r,0} (\omega,t)  \right) & \leq \varepsilon. \label{sl2}
\end{align}
From Lemma \ref{lmf}, there exists $\widetilde{\mathcal{Q}} (\cdot) \in  \widetilde{\mathcal{D}}_{L}$ and $r_0>0$    such that for all $r>r_0$
\begin{align}
\max_{m \leq \lfloor rT \rfloor} \, \sup_{t \in [0,L]} \mathbf{d} \left(  \bar{\mathcal{Q}}^{r,m} (\omega,t) ,  \widetilde{\mathcal{Q}} (t) \right ) < \varepsilon /3. \label{aw2}
\end{align}
Having \eqref{aw2} and \eqref{aw1}, it remains to prove \eqref{aw3} in order to show $(\ref{sl1})$. 
Let $a$ and $b$ two vectors in $ \mathbb{R}_+^K$, and let $A\subset \mathbb{R}_+$ a Borel set. 
Given that $\langle 1, \mathcal{B} \ast \mathcal{V}^e \rangle =Q$, a simple computation yields
$$\mathbf{d} [ (\mathcal{B} \ast \mathcal{V}^e) \, a , (\mathcal{B} \ast  \mathcal{V}^e ) \, b ] \leq   \lvert b-a \rvert \, \lvert Q \rvert . $$
Using the previous formula, we obtain for all $ r > 0$ and $t \in [0,L]$
\begin{align*}
\max_{m \leq \lfloor rT \rfloor} \, \mathbf{d} \left( \mathcal{B}\ast\Delta^{\nu} \, \overline{W}^{r,m} (\omega,t) ,  \mathcal{B}\ast\Delta^{\nu} \, \widetilde{W} (t) \right) 
 \leq \dfrac{\lvert Q \rvert \, \lvert M \lambda \rvert }{e.(\frac{1}{2} M^{(2)}+M P^{\prime}QM)\lambda} \, \left \lvert \overline{W}^{r,m} (\omega,t) - \widetilde{W} (t)  \right \rvert.
\end{align*}
Therefore, by \eqref{unif_W} there exists $r_1>0$ such that for all $r > r_1$,
\begin{align}
\max_{m \leq \lfloor rT \rfloor} \sup_{t \in [0,L]} \mathbf{d} \left( \mathcal{B}\ast\Delta^{\nu} \, \bar{W}^{r,m} (\omega,t) ,  \mathcal{B}\ast\Delta^{\nu} \, \widetilde{W} (t) \right) < \varepsilon/3.  \label{aw3}
 \end{align}
As a result, \eqref{sl1} follows for all $r > r'= \min \lbrace r_0,r_1 \rbrace$. 
To prove \eqref{sl2}, it sufficies to show,
\begin{align*}
\sup_{t \in [0,L^*]} \, \mathbf{d} \left( \widetilde{\mathcal{Q}} (t) ,  \mathcal{B}\ast\Delta^{\nu} \, \widetilde{W} (t) \right)  < \varepsilon /3,
\end{align*} 
It follows from assumption \eqref{eq:assum-init4} and   Proposition \ref{invst} and definition of the lifting map $\mathcal{B}\ast\Delta^{\nu}$, that
\begin{align*}
\widetilde{\mathcal{Q}} (t) = \mathcal{B}\ast\Delta^{\nu} \, \widetilde{W} (0) = \mathcal{B}\ast\Delta^{\nu} \, \widetilde{W} (t) ~~\text{ for all } t \geq 0.
\end{align*}
Thus the result is proved.

\section*{Index of Notation}
\vspace*{0.2cm}

\begin{tabular}{cp{10cm}}

$\mathcal{K}$ & The set of all classes \\
$I(x)$ & The interval $[x, \infty)$  \\
$e$ &  Row vector of ones \\
$P'$ & The transpose of the routing matrix $P$ \\
$Q$ & The matrix $I+P'+P'^2 + \cdots = (I-P')^{-1} $ \\
$\mathcal{M}$ & The space of finite,
positive Borel measures on $\mathbb{R}_{+}$. \\
$\mathcal{M}^K$ & The $K$-ary Cartesian power of $\mathcal{M}$ \\
$\langle g,\nu \rangle$ & $\int g d\nu$ \\
$\nu(A)$ & $\langle 1_A,\nu \rangle$ \\
$\langle  g(\cdot - a),\nu \rangle$ & $ \int_{[a,\infty)} g(x-a) \nu(dx)$ \\
$\mathbf{D} ([0,\infty), S)$ & The space of all right continuous
$S$-valued functions with finite left limits defined on the interval  $[0,\infty)$\\[0.13cm]

$v_{k}(i)$ & The
service times for the $i^{th}$ class $k$ job \\
$v_{k}^{0}(i)$ & The service times requirement of the $i^{th}$
initial job at class $k$.\\
$\nu_k$ & The Borel probability measure of $v_k$ \\
$\nu_k^0$ & The Borel probability
measure of $v_k^0$ \\
$\nu$ & The probability measure vector $(\nu_1, \nu_2, \cdots,\nu_K)$ \\
$\nu^0$ & The probability measure vector $(\nu^0_1, \nu_2^0, \cdots,\nu_K^0)$ \\
$M$& $diag\{\langle \chi,\nu_k \rangle, k\in\mathcal{K}\}$ \\
$M^{(2)}$ & $diag\{\langle \chi^2,\nu_k \rangle, k\in\mathcal{K}\}$ \\
$M^{0}$ & $diag\{\langle \chi,\nu_k^0 \rangle, k\in\mathcal{K}\}$\\
$B_k(t)$ & The distribution function of $\nu_k$ \\
$B(t)$ & $diag\{B_k(t),~k\in\mathcal{K}\}$\\
$ \mathcal{V} $ & $diag \{\nu_k , k \in \mathcal{K} \}$\\
 $ \mathcal{V}^0 $ & $diag \{\nu_k^0 , k \in \mathcal{K} \}$\\
$F^{*n}$ & Convolution of $F$ $n$ times \\
$\mathcal{B}(t)$ & The matrix function $\sum_{n \geq 0} (BP')^{*n}$ \\
$\Delta^{\nu}$ & The lifting map associated with the measure $\nu$ \\
$\sigma_{k}(i)$ &
The arrival epoch of the $i^{th}$ job to arrive at class $k$\\[0.13cm]

$N^l_k(i)$ & Total number of visits to class $k$ by the $i^{th}$ job entering the system as a job of class $l$. \\
$V_k(i)$ & The total service time required by the $i$th exogenous job of class $k$ until their departure from the server \\
$V_k^{0}(i)$ & The total service times required by the $i$th initial job of class $k$ \\
$U_k(i)$ & The time at which the $i$th arrival of class $k$ enters the system \\
$V_{lk}(i,n)$ & The sum of service times required by the $i^{th}$ job of class $l$ from its arrival until its $n^{th}$ visit to class $k$ (included) \\
$V^{0}_{lk}(i,n)$ & Total service time required by the $i^{th}$ initial job of class $l$ until its $n^{th}$ visit to class $k$ \\
$U_{l}(i)$ & The exogenous arrival epoch of the $i^{th}$ job of class $l$ \\
$S(t)$ & The amount
of service received in the interval $[0,t]$   \\
$S(s,t) $ & The amount
of service received in the interval $[s,t]$ \\[0.2cm] 
 
$\nu^e_k$ & The excess lifetime distribution associated
with $\nu$ \\
$\nu^e$ & The vector $(\nu^e_1, \nu^e_2, \cdots, \nu^e_K)$ \\
$\mathcal{V}^e$ & $diag \{\nu_k^e , k \in \mathcal{K} \}$ \\
$B^e_k (t)$ & The distribution function of $\nu^e_k$ \\
$B^e (t)$ & The matrix $diag \{B_k^e(t) , k \in \mathcal{K} \}$\\
$\dot{\mathcal{B}}(t)$ & The derivative of $\mathcal{B}(t)$

\end{tabular}

\bibliographystyle{plain} 
\bibliography{citspre_Ap} 

\end{document}